\documentclass[11pt,a4paper,reqno]{amsart}
\usepackage{vmargin,color}
\usepackage[latin1]{inputenc}
\usepackage{amsmath,amsfonts,amsthm,epsfig,graphicx,amssymb}
\usepackage{esint}
\usepackage{caption}
\newcommand{\QQ}{\mathbf{Q}}
\newcommand{\uu}{\mathbf{u}}
\newcommand{\ttt}{\mathbf{t}}
\newcommand{\TT}{\mathbf{T}}

\newcommand{\F}{\mathcal F}
\newcommand{\G}{\mathcal G}

\renewcommand{\H}{\mathcal{H}}
\renewcommand{\L}{\mathcal{L}}

\renewcommand{\P}{\mathcal{P}}

\newcommand{\V}{\mathcal{V}}

\newcommand{\jj}{\mathbf{J}}
\newcommand{\GGG}{\mathbf{G}}
\newcommand{\UU}{\mathbf{U}}

\newcommand{\vv}{\mathbf{v}}

\newcommand{\N}{\mathbb{N}}

\newcommand{\Q}{\mathbb{Q}}
\newcommand{\R}{\mathbb{R}}
\renewcommand{\SS}{\mathbb{S}}

\newcommand{\Om}{\Omega}

\renewcommand{\a}{\alpha}
\renewcommand{\b}{\beta}
\newcommand{\g}{\gamma}
\newcommand{\de}{\delta}
\newcommand{\e}{\varepsilon}

\renewcommand{\l}{\lambda}
\newcommand{\s}{\sigma}

\newcommand{\om}{\omega}
\newcommand{\vphi}{\varphi}
\newcommand{\Lip}{{\rm Lip}}

\newcommand{\Div}{{\rm div}\,}

\newcommand{\dist}{{\rm dist}}
\newcommand{\sdist}{{\rm sd}}
\newcommand{\loc}{{\rm loc}}
\newcommand{\diam}{{\rm diam}\,}

\newcommand{\spt}{{\rm spt}}

\newcommand{\weak}{\rightharpoonup}
\newcommand{\weakstar}{\stackrel{\scriptscriptstyle{*}}{\rightharpoonup}}
\newcommand{\toloc}{\stackrel{\scriptscriptstyle{{\rm loc}}}{\to}}
\newcommand{\ov}{\overline}

\newcommand{\pa}{\partial}

\newcommand{\cc}{\subset\subset}

\newcommand{\cl}{\mathrm{cl}\,}

\newcommand{\KK}{\mathcal{K}}
\newcommand{\T}{\mathcal{T}}

\newcommand{\pp}{\mathbf{p}}

\newcommand{\h}{\H^n}

\newcommand{\C}{\mathcal{C}}

\newcommand{\ac}{\mathcal{AC}_\e}
\newcommand{\acj}{\mathcal{AC}_{\e_j}}

\newcommand\restr[2]{{% we make the whole thing an ordinary symbol
  \left.\kern-\nulldelimiterspace % automatically resize the bar with \right
  #1 % the function
  \right|_{#2} % this is the delimiter
  }}
\newcommand{\K}{\mathcal{K}}

\newcommand{\RR}{\mathcal{R}}

\newcommand{\one}{{\scriptscriptstyle{(1)}}}
\newcommand{\zero}{{\scriptscriptstyle{(0)}}}
\newcommand{\half}{{\scriptscriptstyle{(1/2)}}}

\newcommand{\mres}{\mathbin{\vrule height 1.6ex depth 0pt width %measure restriction
0.13ex\vrule height 0.13ex depth 0pt width 1.3ex}}

\theoremstyle{plain}% default
\newtheorem{theorem}{Theorem}[section]
\newtheorem{lemma}[theorem]{Lemma}

\newtheorem{proposition}[theorem]{Proposition}
\newtheorem*{theorem*}{Theorem}
\newtheorem*{corollary*}{Corollary}

\theoremstyle{definition}

\newtheorem{remark}[theorem]{Remark}

\newtheorem*{notation*}{Notation}

\numberwithin{equation}{section}
\numberwithin{figure}{section}
\newcommand{\id}{{\rm id}\,}

\newcommand{\wire}{\mathbf{W}}
\newcommand{\shn}{\overset{\scriptscriptstyle{\H^n}}{\subset}}
\newcommand{\ehn}{\overset{\scriptscriptstyle{\H^n}}{=}}

\newcommand{\tnl}{\theta^{n}_*}

\allowdisplaybreaks

\title[]{A hierarchy of Plateau problems \\ and the approximation of Plateau's laws \\ via the Allen--Cahn equation}

    \author{Francesco Maggi}
    \address{Department of Mathematics, The University of Texas at Austin, Austin, TX, United States of America}
    \email{maggi@math.utexas.edu}
    \author{Michael Novack}
    \address{Department of Mathematical Sciences, Carnegie Mellon University, Pittsburgh, PA, United States of America}
    \email{mnovack@andrew.cmu.edu}
    \author{Daniel Restrepo}
    \address{Department of Mathematics, Johns Hopkins University, Baltimore, MD, United States of America}
    \email{drestre1@jh.edu}

\begin{document}

\begin{abstract} {\rm We introduce a diffused interface formulation of the Plateau problem, where the Allen--Cahn energy $\mathcal{AC}_\varepsilon$ is minimized under a volume constraint $v$ and a spanning condition on the level sets of the densities. We discuss two singular limits of these Allen--Cahn Plateau problems: when $\varepsilon\to 0^+$, we prove convergence to the Gauss' capillarity formulation of the Plateau problem with positive volume $v$; and when $\varepsilon\to 0^+$, $v\to 0^+$ and $\varepsilon/v\to 0^+$, we prove convergence to the classical Plateau problem (in the homotopic spanning formulation of Harrison and Pugh). As a corollary of our analysis we resolve the incompatibility between Plateau's laws and the Allen--Cahn equation implied by a regularity theorem of Tonegawa and Wickramasekera. In particular, we show that  Plateau-type singularities can be approximated by energy minimizing solutions of the Allen--Cahn equation with a volume Lagrange multiplier and a transmission condition on a spanning free boundary.}
\end{abstract}

\maketitle

\setcounter{tocdepth}{1}

\tableofcontents

\section{Introduction} \subsection{Overview} The convergence of solutions to the Allen--Cahn equation $\e^2\,\Delta u=W'(u)$ to limit minimal surfaces is a result of basic importance in the study of the van der Waals-Cahn-Hilliard theory of phase transitions \cite{gurtin,modicaARMA,sternberg,kohnsternberg,HT}. A regularity result of Tonegawa and Wickramasekera \cite{tonegawaCAG,TW} shows that, in low dimensions, minimal surfaces arising as limits of stable solutions to the Allen--Cahn equation are necessarily {\it smooth}. While this result makes the Allen--Cahn equation a useful tool for constructing minimal surfaces in Riemannian manifolds, see e.g.~ \cite{guaracoJDG}, it also stands as a limitation to its descriptive power when studying {\it soap films}. Indeed, according to Plateau's laws, soap films can be modeled as two-dimensional smooth minimal surfaces joining in threes at equal angles along lines of ``$Y$-points'', which, in turn, are either closed or meet in fours at isolated ``$T$-points'' where they asymptotically form regular tetrahedral angles. The Tonegawa--Wickramasekera theorem implies in particular that no minimal surface with Plateau-type singularities can arise as the limit of stable solutions to the Allen--Cahn equation.

\medskip

Here we prove that minimal surfaces with Plateau-type singularities can indeed be approximated by energy-minimizing solutions to the Allen--Cahn equation modified by the inclusion of a Lagrange multiplier term corresponding to a {\it small volume constraint}, and with the introduction of a {\it transmission condition} along a ``spanning'' level set. These solutions are constructed as minimizers of a ``diffused interface'' soap film model $\Upsilon(v,\e,\de)$, which is introduced here for the first time. The introduction of a small volume constraint has its origin in the Physics literature, where a distinction between ``dry'' and ``wet'' soap films is made \cite{weaireBOOK,foamchapter}. While dry soap films are two dimensional surfaces obeying Plateau's laws, in the wet soap film model Plateau-type singularities are resolved as {\it Plateau borders} -- constant mean curvature channels of liquid developing around lines of $Y$-points, that are supposed to attach tangentially to smooth interfaces with zero mean curvature; see Figure \ref{fig ske} below. In the companion paper \cite{MNR1} we have recently validated the wet soap film model in the framework of Gauss' capillarity theory. The diffused interface soap film model introduced here thus completes a hierarchy of Plateau-type problems including wet and dry soap film models.

\medskip

The main result of this paper is showing how one can move along this hierarchy of models by taking singular limits. In more concrete terms, and coming back to the problem of approximating Plateau-type singularities by solutions to the Allen--Cahn equation, our main results can be roughly described as follows. {\it First}, given a compact set $\wire\subset\R^{n+1}$ (the ``wire frame''), a non-degenerate double-well potential $W:[0,1]\to[0,\infty)$, a related volume potential $V(t)=(\int_0^t\sqrt{W})^{(n+1)/n}$, and interface length scales $\e_j\to 0^+$ and volumes $v_j\to 0^+$ with $\e_j/v_j\to 0^+$, we construct energy minimizing solutions $\{u_j\}_j$ to the free boundary problems
\begin{equation}\label{ace modified}
  \left\{
  \begin{split}
    &2\,\e_j^2\,\Delta u_j=W'(u_j)-\e_j\,\l_j\,V'(u_j)\,,&\quad\mbox{on $\Om\cap\{u_j<1\}$}\,,
    \\
    &|\pa_\nu^+u_j|=|\pa_\nu^-u_j|\,,&\quad\mbox{on $\Om\cap\{u_j=1\}$}\,,
    \\
    &\mbox{subject to $\int_\Om V(u_j)=v_j$ and $\{u_j=1\}$ spans $\wire$}\,,
  \end{split}
  \right .
\end{equation}
where $\Om=\R^{n+1}\setminus\wire$, $\l_j\in\R$ are suitable Lagrange multipliers with $\e_j\,\l_j\to 0$, and $\pa_\nu^\pm$ denote the one-sided directional derivative operators along the hypersurfaces $\{u_j=1\}$. {\it Second}, we show that, up to extracting subsequences in $j$, for every such $\{u_j\}_j$ there is a (possibly singular) minimal surface $S$, which is area minimizing among surfaces spanning $\wire$, and is such that, as $j\to\infty$,
\begin{equation}
  \label{limit}
  \frac12\,\int_\Om\,\vphi\,\Big\{\e_j\,|\nabla u_j|^2+\frac{W(u_j)}{\e_j}\Big\}\to2\,\int_S \vphi\,d\H^n
\end{equation}
for every $\vphi\in C^0_c(\R^{n+1})$; see
\begin{figure}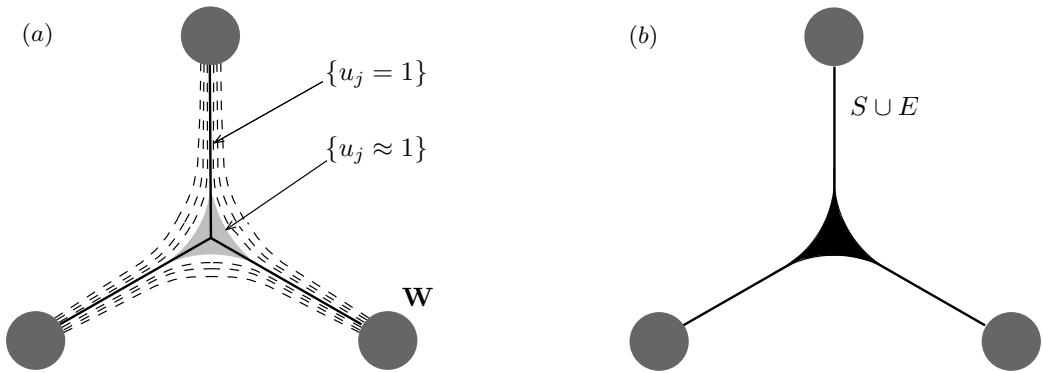\caption{\small{When $\wire$ consists of three disks in the plane, the only possible limit $S$ in \eqref{limit} consists of three segments, each orthogonal to one of the disks, and meeting at a common endpoint at equal angles: (a) Heuristic arguments suggest that, if $u_j$ is a solution to \eqref{ace modified}, then $\{u_j=1\}$ should be equal to $S$, with $u_j$ taking values close to $1$ on a negatively curvilinear triangle $E$ centered at the triple point of $S$  (depicted in gray), and then sharply transitioning to near zero values on a small neighborhood of $S\cup E$ (depicted by dashed lines). The normal derivatives $\pa_\nu^+u_j$ and $\pa_\nu^-u_j$ of $u_j$ should take non-zero, non-constant and opposite values along $S$. (b) As $j\to \infty$, the pointwise limit of $u_j$ should be equal to $1$ on $S\cup E$ (depicted in black). }}\label{fig uj}\end{figure}
Figure \ref{fig uj}.  For various choices of $\wire$ there will be only one such area minimizing surface $S$, which will indeed possess Plateau-type singularities. Actually, since our construction passes through the intermediate wet soap film model of \cite{MNR1}, in a situation where the Plateau problem defined by $\wire$ admits multiple area minimizing surfaces, some smooth and some with Plateau-type singularities, the only possible limits $S$ in \eqref{limit} will be surfaces with Plateau-type singularities.

\medskip

One can of course think of other possible modifications of the Allen--Cahn equation that lead to a PDE-description of Plateau-type singularities. A well-known possibility consists in working with an Allen--Cahn {\it system} \cite{sistobaldo}. From the physical viewpoint this approach corresponds to describing the three regions locally defined by a $Y$-singularity as occupied by three different immiscible fluids. In this sense, the approach followed here, which insists on the use of a single {\it scalar} equation and is based on the introduction of a small volume constraint and of a spanning condition, seems more true to the actual nature of soap films. The emergence, in this approach, of the physically meaningful wet soap film model studied in \cite{MNR1}, is yet another indication of its naturalness.

\medskip

In Section \ref{section homotopic spanning intro} we recall the homotopic spanning formulation $\ell$ of the Plateau problem introduced by Harrison and Pugh in \cite{harrisonpughACV}. In Section \ref{section capillarity intro} and Section \ref{section the model} we introduce, respectively, the capillarity approximation $\Psi_{\rm bk}(v)$ of $\ell$ studied in \cite{MNR1} and the new diffused interface problems $\Upsilon(v,\e,\de)$. In Section \ref{section statements} we state the main result of this paper, Theorem \ref{theorem upsilon existence convergence}, where we prove the existence of minimizers of $\Upsilon(v,\e,\de)$ and their convergence towards minimizers of $\Psi_{\rm bk}(v)$ and $\ell$ in the limits as $\e\to 0^+$, and as $\e\to0^+$, $v\to 0^+$ and $\e/v\to 0^+$, respectively. Additionally, in Theorem \ref{theorem main regularity} we derive the distributional form of \eqref{ace modified}, see \eqref{EL outer intro}, and in Proposition \ref{proposition conditional regularity} we deduce \eqref{ace modified} from its distributional form under some conditional regularity assumptions.

\medskip
 
The results of this paper open the study of Plateau's laws by means of free boundary problems. This point, which seems very interesting, is discussed in detail in Section \ref{section future}.

\subsection{Plateau's laws, the Plateau problem, and homotopic spanning}\label{section homotopic spanning intro} The properties of solutions to Plateau's problem of finding area minimizing surfaces with a given boundary depend subtly on the notions of ``area'' and ``boundary'' employed. The classical formulation of the Plateau problem based on the theory of currents leads, in physical dimensions, to {\it smooth} area minimizing surfaces, so that surfaces with Plateau singularities will be ``invisible'' even when having lower area.

\medskip

Finding a formulation of the Plateau problem whose minimizers may actually show Plateau-type singularities is a delicate task, with a long history, see \cite{davidshouldwe}. An effective approach has been proposed by Harrison and Pugh in \cite{harrisonpughACV}, with the introduction of the notion of {\it homotopic spanning}. Following the presentation given in \cite{DLGM}, given a closed set $\wire\subset\R^{n+1}$ (the wire frame to be spanned), and setting $\Om=\R^{n+1}\setminus\wire$, we say that a family $\C$ of smooth embeddings $\g:\SS^1\to\Om$ defines a {\bf spanning class for $\wire$} if $\Phi(\cdot,1)\in\C$ whenever $\Phi\in C^\infty(\SS^1\times[0,1];\Om)$, $\Phi(\cdot,t)$ is a smooth embedding of $\SS^1$ into $\Om$ for every $t$, and $\Phi(\cdot,0)\in\C$. Then a relative closed set $S\subset\Om$ is said to be {\bf $\C$-spanning $\wire$} if
\begin{equation}
  \label{spanning closed}
  S\cap\g(\SS^1)\ne\varnothing\,,\qquad\forall \g\in\C\,,
\end{equation}
and the following homotopic spanning formulation of {\bf the Plateau problem} is given,
\begin{equation}
  \label{def of ell}
  \ell=\inf\big\{\H^n(S):\mbox{$S$ is relatively closed in $\Om$, $S$ is $\C$-spanning $\wire$}\big\}\,,
\end{equation}
where $\H^n$ denotes the $n$-dimensional Hausdorff measure in $\R^{n+1}$. Minimizers of $\ell$ exist as soon as $\ell<\infty$ \cite{harrisonpughACV,DGM}, and they are {\bf Almgren minimal sets}, that is to say,  they satisfy $\H^n(S)\le\H^n(f(S))$ whenever $f$ is a Lipschitz map, not necessarily injective, with $\{f\ne\id\}\cc\Om$. As proved by Taylor \cite{taylor76}, when $n=2$ an Almgren minimal set $S$ is locally $C^{1,\a}$-diffeomorphic either to a plane, or to a $Y$-cone, or to a $T$-cone, that is, it obeys Plateau's laws. An analogous result is available, by elementary means, in the other important physical case, $n=1$; and similar results also hold in dimension $n\ge 3$, see \cite{ColomboEdelenSpolaor}. In particular, in physical dimensions $n=1,2$, minimizers of $\ell$ may satisfy Plateau's laws, and, for suitable choices of $\wire$ and $\C$, one can prove that this indeed the case when $n=1$ -- see also \cite{bernsteinmaggi} for an analysis of the appearance of singular catenoids when $n=2$. For all these reasons, our analysis will be based on the Harrison--Pugh formulation of the Plateau problem.

\subsection{Capillarity approximation of the Plateau problem}\label{section capillarity intro} In \cite{MaggiScardicchioStuvard,KingMaggiStuvard} a model for soap films as three-dimensional regions with small but positive volume has been introduced, based on Gauss' capillarity theory. Let us recall that, in Gauss' capillarity theory, one minimizes $\H^n(\Om\cap\pa E)$ among open sets $E\subset\Om$ with smooth boundary under a volume constraint $|E|=v$. When $v$ is small such minimizers are close to half-balls \cite{maggimihaila}. To avoid droplet-like minimizers, and actually observe soap film-like minimizers, in \cite{MaggiScardicchioStuvard,KingMaggiStuvard} the following problem
\[
\psi(v)=\inf\Big\{\H^n(\Om\cap\pa E):\mbox{$|E|=v$ and $\Om\cap\pa E$ is $\C$-spanning $\wire$}\Big\}\,,
\]
where $E$ ranges among subsets of $\Om$ with Lipschitz regular boundary, has been introduced. As proved in \cite{KingMaggiStuvard,KMS2,KMS3}, $\psi(v)$ admits minimizers only in a generalized sense, and such generalized minimizers converge to minimizers of $\ell$, with $\psi(v)\to 2\,\ell$ as $v\to 0^+$. The existence of generalized minimizers in $\psi(v)$ corresponds to the actual physical description \cite{weaireBOOK,foamchapter} of soap films as either ``dry'' soap films (minimizers of $\ell$) or ``wet'' soap films (minimizers of $\psi(v)$); see
\begin{figure}
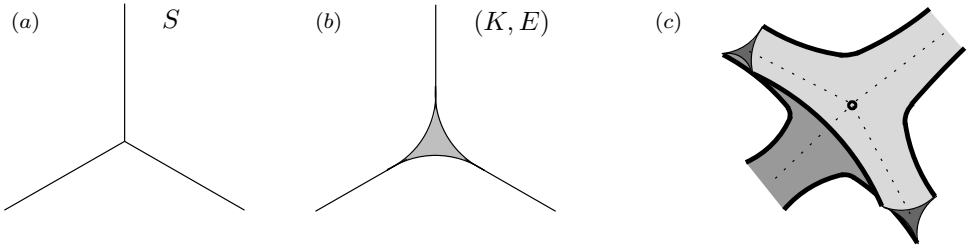
\caption{\small{(a) A ``dry'' soap film $S$ in $\R^2$ with a $Y$-type singularity; (b) a corresponding ``wet'' soap film $(K,E)$: the $Y$-singularity has been wetted by a negatively curved region $E$ (``planar'' Plateau border); (c) in $\R^3$, nearby a $T$-point of a dry film $S$, a wet film $(K,E)$ is placing a negatively curved tube-like structure $E$ (Plateau border). Plateau borders are important, for example, to understand drainage phenomena in soap films.}}
\label{fig ske}
\end{figure}
Figure \ref{fig ske}. Establishing the sharp regularity of these generalized minimizers, and in particular the validity of a sort of ``third Plateau law'' for their characteristic structures known as {\it Plateau borders}, is the subject of \cite{MNR1}. We now review the approach developed in \cite{MNR1}, which is crucial for setting up the Allen--Cahn Plateau problem studied in this paper.

\medskip

The starting point of \cite{MNR1} is reinterpreting the notion of $\C$-spanning set introduced in \eqref{spanning closed}, which is a condition sensitive to pointwise modifications, so to make it stable under modifications by $\H^n$-null sets and under the operation of taking weak limits in the sense of Radon measures. Postponing to Section \ref{section basic properties of C spanning sets} a detailed discussion of this matter, it suffices to notice here that the work done in \cite{MNR1} gives a meaning to the statement ``$S$ is $\C$-spanning $\wire$'' whenever $S$ is a Borel subset of $\Om$, and does so in such a way that: (i) if $S$ is relatively closed in $\Om$, the new condition is equivalent to \eqref{spanning closed}; (ii) if $S$ is $\C$-spanning $\wire$ and $S'$ is $\H^n$-equivalent to $S$, then $S'$ is $\C$-spanning $\wire$; (iii) if $S_j$ are $\H^n$-finite sets that are $\C$-spanning $\wire$, $\mu$ is a Radon measure in $\Om$, and $\H^n\mres S_j\weakstar\mu$ as Radon measures in $\Om$ as $j\to\infty$, then
\begin{eqnarray*}
&&\mbox{$S=\{x\in\Om:\theta_n^*(\mu)(x)\ge 1\}$ is $\C$-spanning $\wire$}\,,
\\
&&\mbox{and}\,\,\,\H^n(S)\le\liminf_{j\to\infty}\H^n(S_j)\,;
\end{eqnarray*}
and (iv) the homotopic spanning Plateau problem $\ell_{\rm B}$ obtained by minimizing $\H^n(S)$ among Borel sets $S$ that $\C$-spans $\wire$ actually coincides with problem $\ell$ introduced in \eqref{def of ell}, that is, $\ell=\ell_{\rm B}$ and the two problems have the same minimizers (modulo $\H^n$-equivalence of sets).

\medskip

Based on this definition we can directly consider Gauss' capillarity energy under homotopic spanning conditions in the class of sets of finite perimeter, and formulate the problem
\[
\psi_{\rm bk}(v)=\inf\Big\{\H^n(\Om\cap\pa^*E):\mbox{$|E|=v$ and $\Om\cap(E^\one\cup\pa^* E)$ is $\C$-spanning $\wire$}\Big\}\,,
\]
where $\pa^*E$ denotes the reduced boundary of a set of locally finite perimeter $E\subset\Om$ and $E^\one$ is the set of points of density $1$ of $E$. The subscript ``bk'' stands for ``bulk'' to reflect the fact that, in formulating $\psi_{\rm bk}(v)$, we are now imposing the burden of achieving the spanning condition {\it not} on the boundary of $E$ alone, as done with $\psi(v)$, but rather on the whole bulk of $E$. The two approaches are evidently related (to the point that one naturally conjectures they should lead to the same minimizers when $v$ is small enough), and we are not aware of a physical reason to prefer one to the other. However, the bulk variant is much more natural to work with in view of the formulation of an Allen--Cahn Plateau problem, which is the decisive reason for us to work with the bulk spanning condition, and to consider $\psi_{\rm bk}(v)$ in place of $\psi(v)$ in \cite{MNR1}, and in what follows.

\medskip

It is now convenient to recall the main result from \cite{novackGENMIN}, a companion paper to \cite{MNR1}. Introducing the class
\[
\KK_{\rm B}
\]
of those pairs $(K,E)$ of Borel subsets of $\Om$ such that
\begin{equation}
  \label{def of KB}
  \mbox{$E$ is of locally finite perimeter in $\Om$ and $\Om\cap\pa^*E$ is $\H^n$-contained in $K$}\,,
\end{equation}
and the {\it relaxed energy}
\begin{equation}
  \label{def of Fb}
  \F_{\rm bk}(K,E;A)=\H^n(A\cap\pa^*E)+2\,\H^n(A\cap K\cap E^\zero)\,,\qquad \F_{\rm bk}(K,E):=\F_{\rm bk}(K,E;\Om)\,,
\end{equation}
(where\footnote{It is important to keep in mind that when $E$ is of locally finite perimeter in $\Om$, then $\{E^\one, E^\zero,\Om\cap\pa^* E\}$ is an $\H^n$-partition of $\Om$ by a theorem of Federer.} $E^\zero$ is the set of points of density $0$ of $E$, and with $A\subset\Om$), the main result proved in \cite{novackGENMIN} is that $\psi_{\rm bk}(v)=\Psi_{\rm bk}(v)$, where
\begin{equation}
  \label{def of Psibk}
  \Psi_{\rm bk}(v):=\inf\Big\{\F_{\rm bk}(K,E;\Om):\mbox{$(K,E)\in\KK_{\rm B}$, $|E|=v$, $K\cup E^\one$ is $\C$-spanning $\wire$}\Big\}\,.
\end{equation}
Notice that if $E$ is a competitor of $\psi_{\rm bk}(v)$, then $(\varnothing,E)\in\KK_{\rm B}$ with $\F_{\rm bk}(\varnothing, E)=\H^n(\Om\cap\pa^*E)$, so that, trivially $\psi_{\rm bk}(v)\ge\Psi_{\rm bk}(v)$. The equality  $\psi_{\rm bk}(v)=\Psi_{\rm bk}(v)$ thus expresses the fact that minimizers of $\psi_{\rm bk}(v)$ may fail to exist in a proper sense, and may thus be found only in a relaxed sense as minimizers of $\Psi_{\rm bk}(v)$.
%Notice that if we take $v=0$ in $\psi(v)$ or in $\psi_{\rm bk}(v)$ we obtain trivial variational problems, while $\Psi_{\rm bk}(0)=2\,\ell_{\rm B}=2\,\ell$.
The following theorem summarizes \cite[Theorem 1.5, 1.6, B.1]{MNR1}:

\begin{theorem}[Main results from \cite{MNR1}]
  \label{theorem from MNR1} If $\wire\subset\R^{n+1}$ is compact, $\C$ is a spanning class for $\wire$, and $\ell<\infty$, then $\Psi_{\rm bk}(v)\to 2\,\ell=2\,\ell_{\rm B}=\Psi_{\rm bk}(0)$ as $v\to 0^+$. Moreover:

  \medskip

  \noindent {\bf (i):} for every $v>0$ there exist a minimizer $(K,E)$ of $\Psi_{\rm bk}(v)$, and up to an $\H^n$-null modification of $K$ and a Lebesgue null modification of $E$, $K$ is relatively closed in $\Om$, $E$ is open with $\Om\cap\cl(\pa^*E)=\Om\cap\pa E\subset K$, $K\cup E$ is $\C$-spanning $\wire$, and $K\cap E^\one=\varnothing$, so that, with disjoint unions,
  \[
  K=\big[K\setminus\pa E\big]\cup\big[\Om\cap(\pa E\setminus\pa^*E)\big]\cup[\Om\cap\pa^*E]\,;
  \]
  moreover, there exists a closed set $\Sigma\subset K$, with $\Sigma=\varnothing$ if $1\le n\le 6$, $\Sigma$ locally finite in $\Om$ if $n=7$, and $\H^s(\Sigma)=0$ for every $s>n-7$ if $n\ge 8$, such that:

  \medskip

  \noindent {\bf (a):} $(K\setminus\pa E)\setminus\Sigma$ is a smooth minimal surface;

  \medskip

  \noindent {\bf (b):} $\Om\cap\pa^*E$ is a smooth hypersurface with constant mean curvature denoted by $\l$ if computed with respect to the outer unit normal $\nu_E$ to $E$;

  \medskip

  \noindent {\bf (c):} if $\Om\cap(\pa E\setminus\pa^*E)\setminus\Sigma\ne\varnothing$, then $\l<0$, and for every $x\in \Om\cap(\pa E\setminus\pa^*E)\setminus\Sigma$ there is $r>0$ such that $K\cap B_r(x)$ is the union of two ordered $C^{1,1}$-graphs which detach tangentially along $\Om\cap(\pa E\setminus\pa^*E)$; moreover, $\Om\cap(\pa E\setminus\pa^*E)$ is locally $\H^{n-1}$-rectifiable;

  \medskip

  \noindent {\bf (ii):} if $v_j\to 0^+$ and $(K_j,E_j)$ are minimizers of $\Psi_{\rm bk}(v_j)$, then, up to extracting subsequences, there is a minimizer $S$ of $\ell$ such that, as $j\to\infty$,
  \[
  \int_{\Om\cap\pa^*E_j}\vphi\,d\H^n+ 2\, \int_{\Om\cap K_j\cap E_j^\zero}\vphi\,d\H^n\to 2\,\int_S\,\vphi\,d\H^n\,,
  \]
  for every $\vphi\in C^0_c(\R^{n+1})$.

\medskip

\noindent {\bf (iii):} if, in addition, $\wire$ is the closure of a bounded open set with $C^2$-boundary, then for every $v>0$ and every minimizing sequence $\{(K_j,E_j)\}_j$ of $\Psi_{\rm bk}(v)$ there is a minimizer $(K,E)$ of $\Psi_{\rm bk}(v)$ such that $K$ is $\H^n$-rectifiable and, up to extracting subsequences and as $j\to\infty$,
\begin{equation}
\label{what K and E do}
E_j\to E\,,\qquad \mu_j\weakstar  \H^n\mres(\Om\cap \pa^*E)+ 2\,\H^n\mres(K\cap E^\zero)\,,
\end{equation}
where $\mu_j=\H^n\mres(\Om\cap \pa^*E_j)+2\,\H^n\mres(\RR(K_j)\cap E_j^\zero)$.
\end{theorem}

\subsection{A diffused interface formulation of the Plateau problem}\label{section the model} In the diffused interface approximation of capillarity theory, the position of a liquid at equilibrium is represented, rather than by a set $E\subset\Om$, by a density function $u:\Om\to[0,1]$. Surface tension energy is then represented by the Allen--Cahn energy of $u$,
\[
\ac(u;\Om)=\int_\Om\mathrm{ac}_\e(u(x))\,dx\,,\qquad\mathrm{ac}_\e(u)=\e\,|\nabla u|^2+\frac{W(u)}\e\,,
\]
where $\e>0$ has the dimensions of a length (in particular, $\ac(u)$ has the dimensions of surface area), and $W\in C^{\,2,1}[0,1]$ is a (dimensionless) double-well potential. We assume $W$ to satisfy the basic structural properties
\begin{equation}\label{W nondegeneracy assumptions}
W(0)=W(1)=0\,,\qquad\mbox{$W>0$ on $(0,1)$}\,,\qquad W''(0), W''(1)>0\,,
\end{equation}
as well as the normalization
\begin{equation}\label{W normalization}
\int_0^1 \sqrt{W(t)}\,dt =1\,.
\end{equation}
We now introduce volume and homotopic spanning constraints on densities $u$.

\medskip

\noindent {\bf Volume constraint:} To impose a volume constraint on $u$, we consider a (dimensionless) ``volume density potential'' $V:[0,1]\to[0,\infty)$, with $V(0)=0$ and $V$ increasing and positive on $(0,1]$. Given a choice of $V$, $u$ corresponds to a soap film of volume $v$ if
\[
\V(u;\Om)=v\,,\qquad\mbox{where}\quad \V(u;\Om):=\int_{\Om}V(u(x))\,dx\,.
\]
The choice of $V$ is really a matter of convenience, since any choice of $V$ leads to recover the correct volume constraint in the sharp interface limit $\e\to 0^+$, and since the model is purely phenomenological. When working on bounded domains $\Om$, a common choice of $V$ made in the literature is taking $V(t)=t$. This choice does not work well on unbounded domains, since in that case $\ac(u;\Om)$ can be made arbitrarily small (while keeping $\int_\Om u$ fixed) by simply ``spreading'' $u$. Following the treatment of the {\it diffused interface isoperimetric problem on $\R^{n+1}$} naturally associated with $\ac$, see \cite{maggirestrepo}, we will set
\[
V(t)=\Phi(t)^{(n+1)/n}\,,\qquad\Phi(t)=\int_0^t\,\sqrt{W(s)}\,ds\,,
\]
for $t\in[0,1]$ and $u\in L^1_{{\rm loc}}(\Om)$. This choice is of course motivated by the $BV$-Sobolev embedding and by the ``Modica--Mortola identity''
\begin{equation}\tag{MM}
  \label{modica mortola identity}
  \ac(u;\Om)=2\,|D(\Phi\circ u)|(\Om)+\int_\Om\Big(\sqrt\e\,|\nabla u|-\sqrt{W(u)/\e}\Big)^2\ge2\,|D(\Phi\circ u)|(\Om)\,.
\end{equation}
Notice also that we have $\Phi(1)=V(1)=1$ thanks to the normalization \eqref{W normalization} on $W$.

\medskip

\noindent {\bf Homotopic spanning constraint:} Deciding how to impose an homotopic spanning conditions on densities $u$ is of course a delicate choice in the setting of our model. The idea explored here is requiring, given $\de\in(1/2,1]$, that all the superlevel sets $\{u\ge t\}$ corresponding to\footnote{The lower bound $t>1/2$ is assumed here for the sake of definiteness. It could have been replaced by any other positive lower bound since the condition of being $\C$-spanning is monotone by set inclusion.} $t\in(1/2,\de)$ are $\C$-spanning $\wire$. Having extended the notion of $\C$-spanning from a pointwise unstable condition to an $\H^n$-stable condition is of course a crucial feature to discuss the existence of minimizers\footnote{An alternative approach would have course been working on $W^{1,2}\cap C^0$ and the original definition by Harrison and Pugh. Since this approach requires proving the regularity of minimizers in the process of showing their existence, it seems somehow conceptually less direct and certainly less flexible than first discussing a robust weak formulation, and then proving regularity statements.}. This kind of stability is natural in our problem since $W^{1,2}(\Om)$ is the natural energy space for working with the Allen--Cahn energy and since the {\bf Lebesgue representative} $u^*$ of a Sobolev function $u\in W^{1,2}(\Om)$ is well-defined $\H^n$-a.e. on $\Om$, so that, given two functions $u_1,u_2\in W^{1,2}_{\rm loc}(\Om)$ that are $\L^{n+1}$-equivalent (and thus have same $\ac$ energy), the Borel sets $\{u_1^*\ge t\}$ and $\{u_2^*\ge t\}$ will be $\H^n$-equivalent for every $t\in[0,1]$.

\medskip

All this said, we come to introduce the following family of {\bf Allen--Cahn Plateau problems},
\begin{eqnarray}
  \label{new model intro}
  \Upsilon(v,\e,\de)\!\!\!&=&\!\!\!\inf\big\{\ac(u;\Om)\big/2:\,\,\,u\in W^{1,2}_{\rm loc}(\Om)\,,\,\V(u;\Om)=v\,,
   \\\nonumber
   &&\hspace{2.6cm}\mbox{$\{u^*\ge t\}$ is $\C$-spanning $\wire$ for every $t\in(1/2,\de)$}\big\}\,,
\end{eqnarray}
where $v$ and $\e$ are positive parameters and where $\de\in(1/2,1]$.

\medskip

For arbitrary values of $(v,\e,\de)$, we do not expect minimizers of $\Upsilon(v,\e,\de)$ to have anything to do with soap films. In other words, we need to identify a  {\bf soap film regime} for $(v,\e,\de)$. A first constraint is that $v$ should not be too large with respect to the size of the boundary wire frame $\wire$: indeed, we want to avoid the ``isoperimetric regime'', where minimizers will tend to look like droplets touching $\wire$, rather than like soap films (see \cite{MN}). A second constraint, borne out by heuristic calculations\footnote{In \eqref{justifies SFR} it is rigorously proved that $\Upsilon(v(\e),\e,\de)\to+\infty$ if $v(\e)/\e\to 0$ as $\e\to 0^+$, so that one definitely wants to require, to the least, that $\e\le C\,v$.} involving the optimal Allen-Cahn profile, and aimed at ensuring the boundedness of the minimum energy at small values of $v$ and $\e$, is that $\e<\!\!<v$. Correspondingly, given positive $\tau_0\ge\tau_1>0$, we introduce the family of triples $(v,\e,\de)\in(0,\infty)\times(0,\infty)\times(1/2,1]$ defined by
\begin{eqnarray}
\nonumber
{\rm SFR}(\tau_0,\tau_1)=\Big\{(v,\e,\delta)\!\!\!\!\!\!\!\!
&&: 0<\frac{v}{(\diam\wire)^{n+1}}\le \tau_0\,,\quad 0<\e\,(\diam\wire)^n  \leq \tau_1\,v\,,
\\ \label{def of physical regime}
&&\hspace{3.1cm}\min\Big\{1-\delta, \frac{v}{(\diam \wire)^{n+1}}\Big\} \leq \tau_1 \Big\}\,.
\end{eqnarray}
Given $\tau_0>0$ we will work with $\tau_1$ sufficiently small in terms of $\tau_0$ (and $\wire$, $\C$ and $W$). From this viewpoint, the third constraint defining ${\rm SFR}$ reflects the fact that if we want to keep $v$ ``of order one'', then, in order to be close to the soap film capillarity model with bulk spanning condition $\Psi_{\rm bk}(v)$, we need $\de$ to be sufficiently close to $1$; if, otherwise, we wish to keep the possibility of working with $\de$ close to $1/2$ (thus imposing the spanning condition only on a thin layer of level sets around $t=1/2$), then we will need to work with $v$ sufficiently small.

\subsection{Main results for the diffused interface model}\label{section statements} We are now in the position of formally stating the main results of our paper.

\begin{theorem}\label{theorem upsilon existence convergence}
If $\wire\subset \mathbb{R}^{n+1}$ is the closure of an open bounded set with smooth boundary, $\C$ is a spanning class for $\wire$ such that $\ell<\infty$, $\tau_0>0$, and $W\in C^{\,2,1}[0,1]$ satisfies \eqref{W nondegeneracy assumptions} and \eqref{W normalization}, then there exists $\tau_1>0$, depending on $W$, $\wire$, $\C$, and $\tau_0$ with the following properties:

\medskip

\noindent {\bf (i) Existence of minimizers:}  if $(v,\e,\delta)\in {\rm SFR}(\tau_0,\tau_1)$, then there are minimizers $u$ of $\Upsilon(v,\e,\delta)$, which, for suitable $\l\in\R$, satisfy
\begin{equation}
  \label{EL inner}
      \int_\Omega\mathrm{ac}_\e(u)\,\Div X - 2\,\e
    \nabla u\,\cdot\, \nabla X[\nabla u]=\lambda\, \int_\Omega V(u)\,\Div X\,,
\end{equation}
whenever $X\in C_c^\infty(\mathbb{R}^{n+1};\mathbb{R}^{n+1})$ with $X \cdot \nu_\Omega= 0$ on $\partial \Omega$;

\medskip

\noindent {\bf (ii) Convergence to bulk-spanning capillarity:} if $\e_j\to 0^+$, $v_j\to v_0>0$, and $\de_j\to 1^-$ as $j\to\infty$, and if $u_j$ are minimizers of $\Upsilon(v_j,\e_j,\de_j)$, then there is a minimizer $(K,E)$ of $\Psi_{\rm bk}(v_0)$ such that, up to extracting subsequences, $u_j\to 1_E$ in $L^1(\Om)$ and
\begin{eqnarray*}
&&\frac{\mathrm{ac}_{\e_j}(u_j)}2\,\L^{n+1}\mres \Omega \weakstar
2\, \,\H^n\mres \big(K\cap E^{\zero}\big)+\H^n\mres\pa^*E
\end{eqnarray*}
as Radon measures in $\Om$. In particular, for every $v_0>0$,
\[
\lim_{{\rm SFR}(\tau_0,\tau_1)\ni(v,\e,\de)\to (v_0,0,1)}\Upsilon(v,\e,\de)=\Psi_{\rm bk}(v_0)\,;
\]

\medskip

\noindent {\bf (iii) Convergence to the Plateau problem:} if $v_j\to 0^+$, $\e_j/v_j\to 0^+$, and $\de_j\to \de_0\in[1/2,1]$ as $j\to\infty$,  and if $u_j$ are minimizers of $\Upsilon(v_j,\e_j,\de_j)$, then there is a minimizer $S$ of $2\,\ell=\Psi_{\rm bk}(0)$ such that, up to extracting subsequences,
\begin{eqnarray*}
&&\frac{\mathrm{ac}_{\e_j}(u_j)}2\,\,\L^{n+1}\mres \Omega \weakstar
2\,\Phi(\delta_0)\,\H^n\mres S
\end{eqnarray*}
as Radon measures in $\Om$ and $\Upsilon(v_j,\e_j,\de_j)\to 2\,\Phi(\de_0)\,\ell$, as $j\to\infty$;

\medskip

\noindent {\bf (iv) Equipartition of energy:} in both conclusions (ii) and (iii), we also have
\begin{equation}
  \label{equipartition}
  \lim_{j\to \infty} \e_j\,\int_\Omega |\nabla u_j|^2 = \lim_{j\to \infty} \frac1{\e_j}\,\int_\Omega  W(u_j)\,.
\end{equation}
\end{theorem}

Theorem \ref{theorem upsilon existence convergence} establishes the existence of minimizers of $\Upsilon(v,\e,\de)$ if the soap film regime and organizes problems $\ell$, $\Psi_{\rm bk}(v)$, and $\Upsilon(v,\e,\de)$ into a {\it hierarchy of Plataeu problems}. The first two problems corresponds to modeling soap films as dry or wet accordingly to the physics descriptions given in \cite{weaireBOOK,foamchapter}, while the last problem can be used to provide a diffused interface approximation of both problems which is of definite mathematical interest both from the theoretical and the numerical viewpoint.

\medskip

Theorem \ref{theorem upsilon existence convergence} does not discuss which Allen--Cahn-type equation is solved by minimizers of $\Upsilon(v,\e,\de)$, nor discusses any qualitative property of such minimizers, like their regularity, but for their convergence as Radon measures to minimizers of $\Psi_{\rm bk}(v)$ and $\ell$. In the following theorem we answer the first question (at distributional level) and derive some basic regularity properties.

\begin{theorem}[Euler--Lagrange equation for minimizers of $\Upsilon(v,\e,\delta)$]\label{theorem main regularity} Let $\wire\subset\R^{n+1}$ be compact and let $\C$ be a spanning class for $\wire$. If $u$ is a minimizer of $\Upsilon(v,\e,\de)$ for some $v>0$, $\e>0$ and $\de\in (1/2, 1]$, then, in the sense of distributions, we have (with $\l$ as in \eqref{EL inner}),
\begin{equation}\label{EL outer intro}
(\delta-u)\,\big\{2\e^2\Delta u-W'(u)-\e\lambda V'(u)\big\}=0\,,\qquad\mbox{in $\Omega$}\,;
\end{equation}
that is, for every $\varphi \in C_c^\infty(\Omega)$,
\begin{equation}\label{EL outer}
2\,\e\,\int_{\Omega}|\nabla u|^2\,\varphi  = 	\int_{\Omega} (\delta-u)\,\Big\{ 2\,\e\, \nabla u\cdot \nabla \varphi +\Big(\frac{W'(u)}{\e}-\lambda\, V'(u)\Big)\,\varphi\Big\}\,,
\end{equation}
In particular, $u$ is lower-semicontinuous in $\Om$. Moreover, if $\Om'$ is a connected component of $\Om$, then either $u\equiv 0$ on $\Om'$, or $u>0$ in $\Om'$; and, if $\de<1$, then $u<1$ in $\Om$.
\end{theorem}

The relation between \eqref{ace modified} and \eqref{EL outer intro} is clarified in the following {\it conditional} regularity statement.

\begin{proposition}[Strong form of the Euler--Lagrange equation]
  \label{proposition conditional regularity} Under the assumptions of Theorem \ref{theorem main regularity}:

  \medskip

  \noindent {\bf (i):} if $u$ is continuous in $\Om$, then $u\in C^{3,\a}_{\rm loc}(\{u\neq \delta\})$ for every $\a<\min\{1,2/n\}$;

  \medskip

\noindent {\bf (ii):} if in addition $|\{u=\de\}|=0$, then
\begin{equation}\label{eq fb condition}
\lim_{t\to 0^+} \int_{\partial^*\{u>\delta +t\}}|\nabla u|\,(X\cdot\nabla u)\,d\mathcal{H}^n
-\int_{\partial^*\{u<\delta -t\}} \,|\nabla u|\,(X\cdot \nabla u)\, d\mathcal{H}^n=0\,,
\end{equation}
where the limit is taken along those values of $t>0$ such that $\{u>\de+t\}$ and $\{u<\de-t\}$ are sets of finite perimeter (i.e., a.e. $t>0$);

\medskip

\noindent {\bf (iii):} if in addition $u(x_0)=\de$, $\{u=\de\}$ is a $C^1$-hypersurface in a neighborhood $U$ of $x_0$ with unit normal $\nu\in C^0(\{u=\de\}\cap U;\SS^n)$, and $u\in C^1(\{u\le\de\}\cap U)\cap C^1(\{u\ge\de\}\cap U)$, then
\[
|\pa^+_\nu u(x)|=|\pa^-_\nu u(x)|\,,\qquad\mbox{$\forall x\in\{u=\de\}\cap U$}\,,
\]
where we have set
\[
\pa^\pm_\nu u(x)=\lim_{t\to 0^+}\frac{u(x\pm t\,\nu(x))-u(x)}t\,.
\]
\end{proposition}

\subsection{Plateau's laws and free boundary problems}\label{section future} We finally describe some future directions that naturally stem from the main results of this paper, and that generally concern the study of Plateau's laws in the context of free boundary problems.

\medskip

A first natural class of problems concerns the regularity of solutions to \eqref{EL outer intro} needed to trigger Proposition \ref{proposition conditional regularity}. For example, continuity of minimizers (conditional assumption (i)) is expected in general, and, indeed, it is possible to show that minimizers in the planar case $n=1$ are locally H\"older continuous in $\Om$. The regularity of the free boundaries $\{u=\de\}$ seems also very interesting. Heuristic considerations (based on the maximum principle) suggest that $\{u=\de\}$ should always have zero Lebesgue measure (conditional assumption (ii)), but, in general, we definitely do not expect $\{u=\de\}$ to a be a $C^1$-hypersurface (conditional assumption (iii)). It is actually natural to conjecture that, in physical dimensions $n=1$ and $n=2$, $\{u=\de\}$ should obey Plateau's laws. Should this be correct, do solutions $u$ to \eqref{EL outer intro} have canonical blow-ups at $Y$-points and $T$-points of $\{u=\de\}$?

\medskip

A second type of problem concerns the precise description  of solutions to \eqref{EL outer intro}. In this direction, the first problem one wants to solve is the construction, as small modifications of some well-prepared {\it Ansatz}, of solutions to \eqref{ace modified} that converge to a $Y$-cone in $\R^2$, or to a $Y$-cone or a $T$-cone in $\R^3$. This kind of result should elucidate several interesting points, like what should be the characteristic length scales of the transition regions of $u_j$ and of the $\{u_j\approx 1\}$-regions depicted in Figure \ref{fig uj}. In turn, once these fundamental examples have been understood, one would like to prove such qualitative properties for generic minimizers of $\Upsilon(v,\e,\de)$.

\subsection{Organization of the paper} In Section \ref{section basic properties of C spanning sets} we gather the main results from \cite{MNR1} that concern measure theoretic homotopic spanning. These results are used in Section \ref{section closure theorem diffused} to prove some closure theorems for densities $u$ satisfying homotopic spanning conditions. In Section \ref{section limsup theorems} we discuss the approximation of ``wet soap films'', meant as competitors in $\Psi_{\rm bk}(v)$, by competitors in $\Upsilon(v,\e,\de)$. Section \ref{section lambdaepsilon} contains one of the more delicate arguments of the paper, where we prove that the Lagrange multipliers $\l_j$ of minimizers $u_j$ of $\Upsilon(v_j,\e_j,\de_j)$ are such that $\e_j\,\l_j\to 0^+$ whenever $v_j\to 0^+$, $\e_j\to 0^+$, $\e_j/v_j\to 0^+$ and $\de_j\to \de_0\in[1/2,1]$. Finally, in Section \ref{section existence and convergence for diffused} and Section \ref{section EL proof} we prove Theorem \ref{theorem upsilon existence convergence} and Theorem \ref{theorem main regularity} (plus Proposition \ref{proposition conditional regularity}), respectively.

\subsection*{Acknowledgements} FM has been supported by NSF Grant DMS-2247544. FM, MN, and DR have been supported by NSF Grant DMS-2000034 and NSF FRG Grant DMS-1854344. MN has been supported by NSF RTG Grant DMS-1840314.

\section{Measure theoretic homotopic spanning}\label{section basic properties of C spanning sets}

\subsection{Sets of finite perimeter, rectifiable sets, and essential disconnection} We generally adopt the (quite common) terminology and notation of \cite{maggiBOOK} for what concerns rectifiable sets and sets of finite perimeter. Given a locally $\H^k$-finite set $S$ in $\R^{n+1}$, we define the {\bf rectifiable part} $\RR(S)$ and the {\bf unrectifiable part} $\P(S)$ of $S$ as in \cite[13.1]{Simon}. Given a Borel set $E\subset\R^{n+1}$, we denote by $E^{{\scriptscriptstyle{(t)}}}$, $t\in[0,1]$, the {\bf points of density $t$} of $E$, by $\pa^*E$ the reduced boundary of $E$ (defined as the largest open set $A$ wherein $E$ is of locally finite perimeter -- it could of course be $A=\varnothing$), and by $\nu_E:\pa^*E\to\SS^n$ the measure theoretic outer unit normal to $E$. We shall repeatedly use that if $E$ is of finite perimeter in $\Om$, then $\Om\cap \pa^*E\subset E^\half\subset\pa^eE$ where $\pa^eE=\R^{n+1}\setminus(E^{\zero}\cup E^{\one})$ is the essential boundary of $E$, as well as the theorem by Federer stating that
\begin{equation}
  \label{federer theorem}
  \mbox{$\Om$ is $\H^n$-contained in $E^{\zero}\cup E^{\one}\cup(\Om\cap\pa^*E)$}\,,
\end{equation}
and, in particular, that $\Om\cap\pa^*E$ is $\H^n$-equivalent to $\Om\cap\pa^eE$.

\medskip

We also recall the following notion of what it means for a Borel set $K$ to disconnect a Borel set $G$, originating in the study of rigidity for symmetrization inequalities \cite{CCDPM17,CCDPMSteiner}, and lying at the heart of the notion of measure theoretic homotopic spanning. Given Borel sets $K$ and $G$, we say that $K$ {\bf essentially disconnects} $G$ if there is a Lebesgue partition $\{G_1,G_2\}$ of $G$ (i.e., $|G\Delta (G_1\cup G_2)|=0$, $|G_1\cap G_2|=0$) which is non-trivial (i.e., $|G_1|\,|G_2|>0$) and such that
\[
\mbox{$G^\one\cap\pa^eG_1\cap\pa^eG_2$ is $\H^n$-contained in $K$}\,.
\]
(Notice that $G^\one\cap\pa^eG_1\cap\pa^eG_2=G^\one\cap\pa^eG_k$ for every $k=1,2$.) For example, if $J\subset(0,1)$ with $\L^1(J)=1$, then $K=J\times\{0\}$ essentially disconnects the open unit disk $B_1^2$ of $\R^2$ (although, evidently, $B_1^2\setminus K$ will be connected in topological terms as soon as $(0,1)\setminus J\ne\varnothing$). We say that $G$ is {\bf essentially connected} when $\varnothing$ does not essentially disconnect $G$. In the special case when $G$ is of finite perimeter, being essentially connected is the same as being {\bf indecomposable} (according to the terminology introduced in \cite{ambrosiocaselles}).

\subsection{Homotopic spanning and induced essential partitions} We now recall the measure theoretic notion of homotopic spanning introduced in \cite{MNR1}; see
\begin{figure}
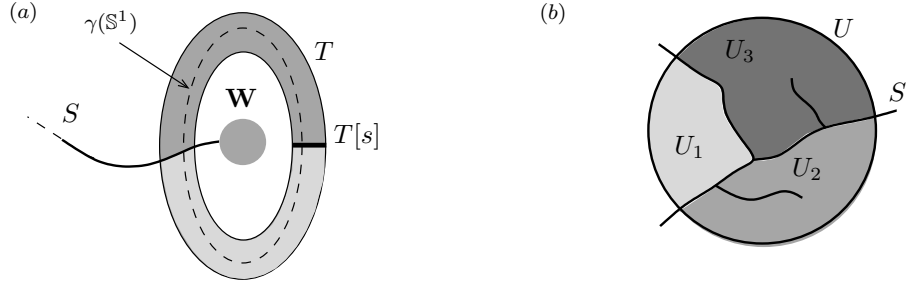
\caption{\small{(a) In the original homotopic spanning condition, $S$ has to intersect $\g(\SS^1)$; in the new measure theoretic version, $S\cup T[s]$ is (roughly speaking) required to essentially disconnect $T$ (for a.e. $s\in\SS^1$); (b) The induced essential partition $\{U_1,U_2,U_3\}$ by $S$ on a disk $U$. Notice that the tendrils of $S$ that do not contribute to bounding some subset of $U$ do not contribute to the boundaries of the the $U_i$'s, and are thus not part of ${\rm UBEP}(S;U)$.}}
\label{fig figurone}
\end{figure}
Figure \ref{fig figurone}-(a) for an illustration. Given a closed set $\wire\subset\R^{n+1}$ and a spanning class $\C$ for $\wire$, the {\bf tubular spanning class} $\mathcal{T}(\C)$  associated to $\C$ is the family of triples $(\gamma, \Phi, T)$ such that $\g\in\C$, $T = \Phi(\mathbb{S}^1 \times B_1^n)\cc\Om$, and (setting $B_1^n=\{x\in\R^n:|x|<1\}$) $\Phi:\SS^1\times\cl(B_1^n)\to\cl(T)$ is a diffeomorphism with $\Phi(s,0)=\g(s)$ for every $s\in\SS^1$. Given $s\in\SS^1$ we set
\[
T[s]=\Phi(\{s\}\times B_1^n)
\]
for the {\bf slice of $T$} corresponding to $s\in\mathbb{S}^1$. Finally, we say that a Borel set $S \subset\Omega$ is {\bf $\C$-spanning} $\wire$ if for each $(\gamma,\Phi, T)\in \mathcal{T}(\C)$, $\H^1$-a.e. $s\in \mathbb{S}^1$ has the following property:
\begin{eqnarray}\nonumber
&&\mbox{for $\H^n$-a.e. $x\in T[s]$}
\\
\label{spanning borel intro}
&&\mbox{$\exists$ a partition $\{T_1,T_2\}$ of $T$ s.t. $x\in\partial^e T_1 \cap \partial^e T_2$}
\\
\nonumber
&&\mbox{and s.t. $S \cup  T[s]$ essentially disconnects $T$ into $\{T_1,T_2\}$}\,.
\end{eqnarray}
As proved in \cite[Theorem A.1]{MNR1}, as soon as $S$ is closed in $\Om$, the notion of $\C$-spanning just introduced is equivalent to the one of Harrison and Pugh. The dependency of the partition $\{T_1,T_2\}$ on $x\in T[s]$ has a subtle reason, see \cite[Figure A.1]{MNR1}.

\medskip

Now, in the study of soap films, condition \eqref{spanning borel intro} is only applied to sets $S$ that are either locally $\H^n$-finite in $\Om$, or that are the bulk $E^\one\cup(\Om\cap\pa^*E)$ of a set $E$ of finite perimeter in $\Om$, or are a combination of these two cases, in the sense that $S=K\cup E^\one$ for some $(K,E)\in\KK_{\rm B}$ (the first two cases are then obtained by taking $S=K\cup E^\one$ with either $E=\varnothing$ or $K=\Om\cap\pa^*E$). In all these cases the geometric meaning of \eqref{spanning borel intro} can be greatly elucidated using the following results concerning partitions into indecomposable sets of finite perimeter.

\medskip

Given Borel sets $S,U\subset\R^{n+1}$, a Lebesgue partition $\{U_i\}_i$ of $U$ (that is, $U_i\subset U$ with $|U\setminus\bigcup_iU_i|=0$ with $|U_i\cap U_j|=\varnothing$ if $i\ne j$) is {\bf induced by $S$} if, for each $i$,
\begin{equation}
\label{induced partition}
\mbox{$U^\one\cap\partial^eU_i$ is $\H^n$-contained in $S$}\,.
\end{equation}
The following theorem is \cite[Theorem 2.1]{MNR1}:

\begin{theorem}[Induced essential partitions \cite{MNR1}]\label{theorem decomposition}
If $U\subset\R^{n+1}$ is a bounded set of finite perimeter and $S\subset\R^{n+1}$ is a Borel set with $\H^n(S \cap U^\one)<\infty$, then there exists a partition $\{U_i\}_i$ of $U$ induced by $S$ such that, for every $i$,
\begin{equation}
\label{essential partition}
\mbox{$S$ does not essentially disconnect $U_i$}\,.
\end{equation}
Moreover, if either $S^*=\RR(S)$ or $S^*$ is $\H^n$-equivalent to $S$, and if $\{U^*_j\}_j$ is a partition of $U$ induced by $S^*$ such that $S^*$ does not essentially disconnect $U_j^*$ for every $j$, then there is a bijection $\s$ such that $|U_i\Delta U^*_{\s(i)}|=0$ for every $i$. For this reason, $\{U_i\}_i$ is called {\bf the essential partition of $U$ induced by $S$}.
\end{theorem}

With $S$ and $U$ as in Theorem \ref{theorem decomposition}, the {\bf union of the (reduced) boundaries of the essential partition induced by $S$ on $U$} is uniquely defined as
\begin{equation}
\label{def of UBEP}
{\rm UBEP}(S;U)=U^\one\cap\bigcup_i\partial^*U_i\,,
\end{equation}
see Figure \ref{fig figurone}-(b), and correspondingly we can formulate the following characterization of measure-theoretic homotopic spanning, cf. with \cite[Theorem 3.1]{MNR1}.

\begin{theorem}[\cite{MNR1}]\label{theorem spanning with partition}
If $\wire\subset\R^{n+1}$ is a closed set in $\R^{n+1}$, $\C$ is a spanning class for $\wire$, and $(K,E)\in\KK_{\rm B}$, then
\begin{eqnarray}
\label{S equal to K union E}
\mbox{$\RR(K)\cup E^{\one}$ is $\C$-spanning $\wire$}\,,
\end{eqnarray}
if and only if, for every $(\gamma,\Phi, T)\in \mathcal{T}(\C)$ and $\H^1$-a.e. $s\in\SS^1$,
\begin{eqnarray}\label{spanning and the S partition equation}
&&\mbox{$T[s]\cap E^\zero$ is $\H^n$-contained in ${\rm UBEP}(K\cup T[s];T)$}\,;
\end{eqnarray}
see
\begin{figure}
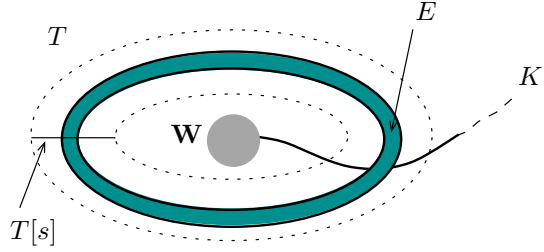
\caption{\small{Condition \eqref{spanning and the S partition equation} in a situation where $T[s]$ is {\it not} $\H^n$-contained in ${\rm UBEP}(K\cup T[s];T)$ -- indeed, $E$ itself is one of the elements of the essential partition of $T$ induced by $K\cup T[s]$ -- while at the same time $T[s]\cap E^\zero$ (the part of $T[s]$ outside of $E$) is $\H^n$-contained in ${\rm UBEP}(K\cup T[s];T)$. The key feature here is that $K$ contains the boundary of $E$, but no points inside $E$, while $E$ contains $\g(\SS^1)$.}}
\label{fig figurone2}
\end{figure}
Figure \ref{fig figurone2}.
\end{theorem}

\begin{remark}\label{remark RRK cup Eone is spanning too}
{\rm An immediate corollary of Theorem \ref{theorem spanning with partition} is that if $K$ is $\H^n$-finite and $(K,E)\in\KK_{\rm B}$ then $K\cup E^\one$ is $\C$-spanning $\wire$ if and only if $\RR(K)\cup E^\one$ is $\C$-spanning $\wire$. Indeed, $\RR(K\cup T[s])=\RR(K)\cup T[s]$, so that, by \eqref{def of UBEP}, ${\rm UBEP}(K\cup T[s])={\rm UBEP}(\RR(K)\cup T[s])$.}
\end{remark}

\subsection{Closure theorems for homotopic spanning} We finally state the two closure theorems for homotopically spanning sets that make the above definitions useful in the study of minimization problems. The first result corresponds to a particular case of \cite[Theorem 1.4]{MNR1}:

\begin{theorem}[\cite{MNR1}]\label{theorem first closure theorem} Let $\wire$ be a closed set in $\mathbb{R}^{n+1}$, $\C$ be a spanning class for $\wire$, and $\{(K_j,E_j)\}_j$ be a sequence in $\K_{\rm B}$ such that
\[
\mbox{$K_j\cup E_j^\one$ is $\C$-spanning $\wire$}\,,\qquad\sup_j\,\H^n(K_j)<\infty\,.
\]
Let $E$ be a Borel set and $\mu_{\rm bk}$ be a Radon measure in $\Om$ such that, as $j\to\infty$, $E_j\toloc E$ and
\[
\H^n\mres (\Om\cap\pa^*E_j) + 2\,\H^n \mres (\mathcal{R}(K_j) \cap E_j^\zero) \weakstar \mu_{\rm bk}\,,
\]
as Radon measures in $\Om$. Then the set
\[
K_{\rm bk}:=\big(\Om\cap\partial^* E\big) \cup \Big\{x\in \Omega \cap E^\zero : \theta^n_*(\mu_{\rm bk})(x)\geq 2 \Big\}\,,
\]
is such that $(K_{\rm bk},E)\in\K_{\rm B}$, $K_{\rm bk}\cup E^\one$ is $\C$-spanning $\wire$, and
\[
\liminf_{j\to\infty}\F_{\rm bk}(K_j,E_j)\ge\F_{\rm bk}(K_{\rm bk},E)\,.
\]
\end{theorem}

The second closure theorem we shall need from \cite{MNR1} requires the introduction of some additional terminology. Given an open set $\Omega$ and an $\H^n$-finite subset $S$ of $\Omega$, we define the {\bf essential spanning part ${\rm ESP}(S)$ of $S$ in $\Om$} as the $\H^n$-rectifiable set defined by
\begin{equation}
  \label{esp def}
  {\rm ESP}(S)=\bigcup_k\,{\rm UBEP}(S;\Om_k)=\bigcup_k\,\Big\{\Om_k\cap\bigcup_i\pa^*U_i[\Om_k]\Big\}\,,
\end{equation}
where $\{\Om_k\}_k$ is the open covering of $\Om$ defined by
\begin{equation}
\label{def of Omega i}
\{\Om_k\}_k=\{B_{r_{mh}}(x_m)\}_{m,h}\,,
\end{equation}
where $\{x_m\}_m=\Q^{n+1}\cap\Om$ and $\{r_{mh}\}_h=\Q\cap(0,\dist(x_m,\pa\Om))$, and where $\{U_i[\Om_k]\}_i$ denotes the essential partition of $\Om_k$ induced by $S$. In light of Theorem \ref{theorem spanning with partition}, the intuition behind this definition is that, by adding up all the unions of boundaries of essential partitions induced by $S$ over smaller and smaller balls we are capturing all the parts of $S$ that may potentially contribute to achieve a spanning condition on $\wire=\R^{n+1}\setminus\Omega$. It is thus natural to expect that ${\rm ESP}(S)$ is $\C$-spanning $\wire$ whenever $S$ is. This is correct, and follows indeed as a particular case of Theorem \ref{theorem second closure theorem} below. The more general situation considered in Theorem \ref{theorem second closure theorem} requires the introduction of a notion of subsequential limit for $\{{\rm ESP}(S_j)\}_j$. More precisely, given a sequence $\{S_j\}_j$ of Borel subsets of $\Om$ such that $\sup_j\H^n(S_j)<\infty$, we say that $S$ is a {\bf subsequential partition limit of $\{S_j\}_j$ in $\Om$} if
\begin{equation}
\label{spl def}
S=\bigcup_k\,\Big\{\Om_k\cap\bigcup_i\pa^*U_i[\Om_k]\Big\}\,,
\end{equation}
where $\{U_i[\Om_k]\}_i$ is a Lebesgue partition of $\Om_k$ such that, denoting by $\{U_i^j[\Om_k]\}_i$ the essential partition of $\Om_k$ induced by $S_j$, and up to extracting a subsequence in $j$, for every $i$ and $k$ we have $|U_i^j[\Om_k]\Delta U_i[\Om_k]|\to 0$ as $j\to\infty$. The natural expectation is of course that if each $S_j$ is $\C$-spanning $\wire$, then every subsequential partition limit $S$ of $\{S_j\}_j$ should be $\C$-spanning $\wire$ too. The next theorem, which corresponds to \cite[Theorem 5.1]{MNR1}, proves this and, actually, an even more general fact:

\begin{theorem}[\cite{MNR1}]
\label{theorem second closure theorem} Let $\wire$ be a closed set in $\mathbb{R}^{n+1}$, $\C$ a spanning class for $\wire$, and $\{(K_j,E_j)\}_j$ a sequence in $\K_{\rm B}$ such that $\sup_j\H^n(K_j)<\infty$ and $K_j\cup E_j^\one$ is $\C$-spanning $\wire$ for every $j$.

\medskip

If $S_0$ and $E_0$ are, respectively, a subsequential partition limit of $\{K_j\}_j$ in $\Om$ and an $L^1$-subsequential limit of $\{E_j\}_j$ (corresponding to a same not relabeled subsequence in $j$), then the set
\[
K_0=(\Om\cap\pa^*E_0)\cup S_0\,,
\]
is such that $(K_0,E_0)\in\KK_{\rm B}$ and $K_0\cup E_0^\one$ is $\C$-spanning $\wire$. In particular:

\medskip

\noindent {\bf (i):} if $S$ is $\C$-spanning $\wire$, then ${\rm ESP}(S)$ is $\C$-spanning $\wire$;

\medskip

\noindent {\bf (ii):} if $S_j$ is $\C$-spanning $\wire$ for each $j$ and $S$ is a subsequential partition limit of $\{S_j\}_j$ in $\Om$, then $S$ is $\C$-spanning $\wire$.
\end{theorem}

\section{Closure theorems for homotopically spanning diffused interfaces}\label{section closure theorem diffused}

\subsection{The precise representative of a Sobolev function} Given an open set $\Om$ and a Lebesgue measurable function $u:\Om\to\R\cup\{\pm\infty\}$ the approximate upper and lower limits of $u$ at $x\in\Om$ are defined by
\[
u^+(x)=\inf\big\{t\in\R:x\in\{u>t\}^\zero\big\}\,,\qquad u^-(x)=\sup\big\{t\in\R:x\in\{u<t\}^\zero\big\}\,,
\]
(where $\{u>t\}=\{x\in\Om:u(x)>t\}$). Both $u^+$ and $u^-$ are Borel functions on $\Om$, with values in $\R\cup\{\pm\infty\}$, and their value {\it at any point in $\Om$} does not depend on the Lebesgue representative of $u$. It is easily seen that
\begin{eqnarray}\nonumber
  \{u^+<t\}\subset\{u<t\}^\one\subset \{u^+\le t\}\,,\qquad \{u^+<t\}^\one=\{u<t\}^\one\,,
  \\\label{utile}
  \{u^->t\}\subset\{u>t\}^\one\subset \{u^-\ge t\}\,,\qquad \{u^->t\}^\one=\{u>t\}^\one\,.
\end{eqnarray}
The approximate jump of $u$ is the Borel function $[u]:\Om\to [0,\infty]$ defined by $[u]=u^+-u^-$. Setting $\Sigma_u=\{x\in\Om:[u](x)>0\}$, we define the precise representative $u^*:\Om\setminus \Sigma_u\to\R\cup\{\pm\infty\}$ by taking
\[
u^*=u^+=u^-\qquad\mbox{on $\Om\setminus \Sigma_u$}\,.
\]
Since it always hold that $|\Sigma_u|=0$, it turns out that $u^+$, $u^-$ and $u^*$ are all Lebesgue representatives of $u$. If $u\in BV_{\rm loc}(\Om)$, then the distributional derivative $Du$ of $u$ can be decomposed as $Du=\nabla u\,d\L^{n+1}+[u]\,d\H^n\mres \Sigma_u+D^cu$ (where $\nabla u\in L^1_\loc(\R^{n+1};\R^{n+1})$ and $D^cu$ is the Cantorian part of $Du$). In particular, if $u\in W^{1,1}_{\rm loc}(\Om)$, then $\H^n(\Sigma_u)=0$. We shall repeatedly use the following fact: if $u\in W^{1,1}_{\rm loc}(\Om)$, then, for a.e. $t$, $\{u>t\}$ is a set of finite perimeter in $\Om$ (this is immediate from the coarea formula), is Lebesgue equivalent to $\{u\ge t\}$ (and thus such that $\Om\cap\pa^*\{u>t\}=\Om\cap\pa^*\{u\ge t\}$), and satisfies
\begin{equation}
  \label{second trace lemma}
  \mbox{$\Om\cap\pa^*\{u>t\}$ is $\H^n$-equivalent to $\{u^*=t\}$}\,.
\end{equation}
To prove \eqref{second trace lemma} we notice that if $x\in\pa^*\{u>t\}\subset\{u>t\}^\half$, then $x\not\in\{u>t\}^\zero$, hence $u^+(x)\ge t$; similarly, if $x\in\pa^*\{u<t\}\subset\{u<t\}^\half$, then $x\not\in\{u<t\}^\zero$, and thus $t\ge u^-(x)$; since $\{u>t\}$ is $\L^{n+1}$-equivalent to $\Om\setminus\{u<t\}$ for a.e. $t$, we find that $\Om\cap\pa^*\{u>t\}=\Om\cap\pa^*\{u<t\}$ for a.e. $t$, and thus
\[
\Om\cap\pa^*\{u>t\}=\{u^+\ge t\}\cap\{u^-\le t\}\,,\qquad\mbox{for a.e. $t$}\,.
\]
In particular, \eqref{second trace lemma} follows from $\H^n(\Sigma_u)=0$. We finally notice that if $u\in W^{1,1}_{\rm loc}(\Om)$ and $E$ is a set of finite perimeter in $\Om$, then $u^*$ is such that
\begin{equation}
  \label{div theorem}
  \int_E u\,\nabla\vphi=-\int_E\,\vphi\,\nabla u+\int_{\Om\cap\pa^*E}\vphi\,u^*\,\nu_E\,d\H^n\,,
\end{equation}
for every $\vphi\in C^\infty_c(\Om)$; see \cite{AFP}.

\subsection{Closure theorems for homotopically spanning densities} In this section we consider the following setting
\begin{eqnarray}\label{setting 1}
  &&\mbox{$\wire\subset\R^{n+1}$ is closed and $\C$ is a spanning class for $\wire$}\,,
  \\\label{setting 2}
  &&\mbox{$\{u_j\}_j\subset W^{1,2}(\Om)$ with $\sup_j\,\acj(u_j)<\infty$ and $u_j\to u$ in $L^1_{\rm loc}(\Om)$}\,,
  \\\label{setting 3}
  &&\mbox{$\{u_j^*\ge t\}$ is $\C$-spanning $\wire$ for all $t\in(1/2,\de_j)$}\,,
  \\\label{setting 4}
  &&\mbox{$\e_j>0$, $\de_j\in(1/2,1]$, $\e_j\to\e_0\ge0$, and $\de_j\to\de_0\in[1/2,1]$}\,,
\end{eqnarray}
where the limits hold as $j\to\infty$ and where $\Om=\R^{n+1}\setminus\wire$. We discuss the problem of showing that the spanning condition \eqref{setting 3} is transferred from $u_j$ to $u$. We consider separately the cases when $\e_0>0$ (Theorem \ref{theorem fixed epsilon compactness}) or $\e_0=0$ (Theorem \ref{theorem epsilon to zero compactness}). In both cases we use the following lemma.

\begin{lemma}\label{lemma I0}
  Let $\wire$, $\C$, $u_j$, $\e_j$, and $\de_j$ satisfy \eqref{setting 1}, \eqref{setting 2}, \eqref{setting 3}, and \eqref{setting 4}.

  \medskip

  If $I_0\subset(0,\de_0)$ is a closed interval of positive length, then there is $\{t_j\}_j\subset I_0$ such that
  \[
  t_j\to t_0\in I_0\,,\qquad E_j:=\{u_j>t_j\}\toloc E_0\,,
  \]
  as $j\to\infty$, where $E_0$  is a set of finite perimeter in $\Om$ and where $\{\Om\cap\pa^*E_j\}_j$ has a subsequential partition limit $S_0$ in $\Om$ such that
  \begin{equation}
    \label{lemma I0 conclusion}
      \mbox{$S_0\cup(\Om\cap\pa^*E_0)\cup E_0^\one$ is $\C$-spanning $\wire$}\,.
  \end{equation}
\end{lemma}

\begin{proof}
Indeed, by the co-area formula and \eqref{modica mortola identity}
\begin{eqnarray*}
\fint_{\Phi(I_0)}P(\{\Phi\circ u_j> s\};\Om)\,ds
= C(I_0)\,\int_\Om|\nabla(\Phi\circ u_j)|\le C(I_0)\,\acj(u_j;\Om)\,,
\end{eqnarray*}
so that $\sup_j\acj(u_j;\Om)<\infty$ and the strict monotonicity of $\Phi(t)=\int_0^t\sqrt{W}$ imply the existence of $\{t_j\}_j\subset I_0$ such that $\{P(E_j;\Om)\}_j$ is bounded, where $E_j=\{u_j>t_j\}$ and, by \eqref{utile},
\begin{equation}
\label{little claim}
\{u_j^*>t\}\subset\{u_j^->t\}\subset\{u_j>t\}^\one\subset E_j^\one\,,\qquad\forall t>t_j\,.
\end{equation}
Since $\sup_j P(E_j,\Om)<\infty$ we can find $\{t_j\}_j\subset I_0$ and limits $t_0$, $E_0$ and $S_0$ as in the statement, and are left to prove \eqref{lemma I0 conclusion}. Indeed, since $\delta_j\to \delta_0$ and $I_0\subset(0,\de_0)$ is closed, we have $t_j<\delta_j$ for $j$ large enough. In particular, by exploiting \eqref{little claim} with $t\in(t_j,\de_j)$ we find that $E_j^\one$ is $\C$-spanning $\wire$, so that \eqref{utile} implies that $E_j^\one$ is $\C$-spanning $\wire$. We can thus apply Theorem \ref{theorem second closure theorem} to $\{(K_j,E_j)\}_j$ with $K_j=\Om\cap\pa^*E_j$ and conclude the proof.
\end{proof}

\begin{theorem}\label{theorem fixed epsilon compactness} Let $\wire$, $\C$, $u_j$, $\e_j$, and $\de_j$ satisfy \eqref{setting 1}, \eqref{setting 2}, \eqref{setting 3}, and \eqref{setting 4}. If
\[
\e_0>0\,,\qquad \de_0>\frac12\,,
\]
then $u\in W^{1,2}_{\rm loc}(\Om)$ and $\{u^*\ge t\}$ is $\C$-spanning $\wire$ for all $t\in(1/2,\delta_0)$.
\end{theorem}

\begin{remark}\label{remark vanishing volume with non vanishing transition is degenerate}
  {\rm Notice that in the situation of Theorem \ref{theorem fixed epsilon compactness} it must be that $\V(u;T)>0$ for every $(\g,\Phi,T)\in\T(\C)$. Indeed,  $\V(u;T)=0$ would imply $\{u^*\ge t\}\cap T=\varnothing$ for every $t>0$, and $\{u^*\ge t\}$ could not be  $\C$-spanning $\wire$ for any $t\in(1/2,\delta_0)$, a contradiction. As a consequence,
  \begin{equation}
    \label{vanishing volume with non vanishing epsilon is degenerate}
      \liminf_{(v,\e,\de)\to(0,\e_0,\de_0)}\Upsilon(v,\e,\de)=+\infty\,,\qquad \e_0>0\,,\de_0\in(1/2,1]\,,
  \end{equation}
  that is to say, {\it the vanishing volume limit of $\Upsilon$ with non-vanishing phase transition length is always degenerate}.}
\end{remark}

\begin{proof}[Proof of Theorem \ref{theorem fixed epsilon compactness}] Since $\e_0>0$, \eqref{setting 2} implies that $u_j\weak u$ in $W^{1,2}_{\rm loc}(\Om)$, and in particular that $u\in W^{1,2}_{\rm loc}(\Om)$. Given $N\in\N$, let us apply Lemma \ref{lemma I0} to the interval $I_0=[\de_0-2/N,\de_0-1/N]$, and correspondingly find $\{t_j\}_j\subset I_0$ such that $t_j\to t_0\in I_0$, $E_j=\{u_j>t_j\}\toloc E_0$, $\{\Om\cap\pa^*E_j\}_j$ has a subsequential partition limit $S_0$ in $\Om$, and \eqref{lemma I0 conclusion} holds. In particular, we can prove the theorem by showing that
\begin{equation}
  \label{first reduction}
S_0\cup(\Om\cap\pa^*E_0)\cup E_0^\one\,\,\shn\,\, \{u^* \geq t_0\}\,,
\end{equation}
and by then applying the fact that $t_0=t_0(N)\to \de_0^-$ as $N\to\infty$. We divide the proof of \eqref{first reduction} in three parts:

\medskip

\noindent {\it To check that $E_0^\one$ is $\H^n$-contained in $\{u^*\ge t_0\}$}: Up to extract a further subsequence in $j$, the set
\[
E_0^*=\big\{x\in E_0:\mbox{$1_{E_j}(x)\to 1$ and $u_j(x)\to u(x)$ as $j\to\infty$}\big\}\,,
\]
is Lebesgue equivalent to $E_0$. If $x\in E_0^*$, then $u_j(x)\ge t_j$ for every $j\ge j(x)$; letting $j\to\infty$ we find $u(x)\ge t_0$, and prove that $E_0^*\subset\{u\ge t_0\}$. In particular, by \eqref{utile},
\[
E_0^\one=(E_0^*)^\one\subset\{u\ge t_0\}^\one\subset\{u^-\ge t_0\}
\]
and then we find the claimed $\H^n$-containment by intersecting with $\Om\setminus \Sigma_u$ (and recalling that $\H^n(\Sigma_u)=0$).

\medskip

\noindent {\it To check that $\Om\cap\pa^*E_0$ is $\H^n$-contained in $\{u^*\ge t_0\}$}: We combine the general fact that
\[
\Om\cap\pa^*A_2\shn\Om\cap(\pa^*A_1\cup A_1^\one)\,,\qquad\forall A_1\subset A_2\subset\Om\,,
\]
with the inclusion $E_0^\one\subset\{u^*\ge t_0\}$ to find that
\begin{eqnarray*}
  \Om\cap\pa^*E_0=  \Om\cap\pa^*E_0^\one\shn \big(\Om\cap\pa^*\{u^*\ge t_0\}\big)\cup\{u^*\ge t_0\}^\one\,,
\end{eqnarray*}
where the latter set is $\H^n$-contained in $\{u^*\ge t_0\}$ thanks to \eqref{utile}, \eqref{second trace lemma}, and $\H^n(\Sigma_u)=0$.

\medskip

\noindent {\it To check that $S_0$ is $\H^n$-contained in $\{u^*\ge t_0\}$}: Let us recall that given the decomposition $\{\Om_k\}_k$ of $\Om$ introduced in \eqref{def of Omega i}, and denoting by $\{U_i^j[\Om_k]\}_i$ the essential partition of $\Om_k$ induced by $\Om\cap\pa^*E_j$, the fact that $S_0$ is a partition limit of $\{\Om\cap\pa^*E_j\}_j$ means that
\[
S_0=\bigcup_k\Big\{\Om_k\cap\bigcup_i\pa^*U_i[\Om_k]\Big\}\,,
\]
where, for each $k$, $\{U_i[\Om_k]\}_i$ is a partition of $\Om_k$ such that $U_i^j[\Om_k]\to U_i[\Om_k]$ as $j\to\infty$ and for every $i$. Thus, if we fix $k$ and consider $i$ such that $\Om_k\cap\pa^*U_i[\Om_k]\ne\varnothing$, then it suffices to show that
\begin{equation}
  \label{ganzo}
  \mbox{$\Om_k\cap\pa^*U_i[\Om_k]$ is $\H^n$-contained in $\{u^*=t_0\}$.}
\end{equation}
Since $\Om_k=B_r(x)$ (for some $x\in\Om$ and $r>0$) if we set $G_1=U_i[\Om_k]$, $G_1^j=U_i^j[\Om_k]$, $G_2=B_r(x)\setminus G_1$ and $G_2^j=B_r(x)\setminus G_2^j$, then we see that $\{G_1,G_2\}$ is a non-trivial Borel partition of $B_r(x)$ (indeed $0<|G_1|<|B_r(x)|$ thanks to $B_r(x)\cap\pa^*G_1\ne\varnothing$), and $\{G_1^j,G_2^j\}$ is a non-trivial Borel partition of $B_r(x)$ for every $j$ large enough (thanks to $G_1^j,G_2^j\to G_1,G_2$ as $j\to\infty$). In particular,
\begin{eqnarray}
  \label{fact 1}
  \nu_{G_1}=-\nu_{G_2}\qquad\mbox{$\H^n$-a.e. on $B_r(x)\cap\pa^*G_1=B_r(x)\cap\pa^*G_2$}\,,
  \\
  \label{fact 1 j}
  \nu_{G_1^j}=-\nu_{G_2^j}\qquad\mbox{$\H^n$-a.e. on $B_r(x)\cap\pa^*G_1^j=B_r(x)\cap\pa^*G_2^j$}\,.
\end{eqnarray}
Define $L_j:[0,1]\to[0,1]$ by taking $L_j(t)=t$ for $t\in[t_j,1]$ and $L_j$ to be affine on $[0,t_j]$ with $L_j(0)=1$ and $L_j(t_j)=t_j$, and similarly $L_0:[0,1]\to[0,1]$ using $t_0$ in place of $t_j$. Since $t_j,t_0\in I_0\cc(1/2,\de_0)$, we have $\Lip(L_0),\Lip(L_j)\le 1$, and thus $L_0\circ u,L_j\circ u_j\in W^{1,2}_{\rm loc}(\Om)$. If we set
\begin{equation}
  \label{formula for z}
z_j=(L_j\circ u_j)\,1_{G_1^j}+t_j\,1_{G_2^j}\,,\qquad z=(L_0\circ u)\,1_{G_1}+t_0\,1_{G_2}\,,
\end{equation}
then we easily see that $z_j\to z$ in $L^1(B_r(x))$. Moreover, by combining \eqref{fact 1 j} and \eqref{fact 1} with the divergence theorem \eqref{div theorem} we see that, for every $\vphi\in C^\infty_c(B_r(x))$,
\begin{eqnarray}\nonumber
Dz_j[\vphi]&=&-t_j\,\int_{B_r(x)\cap G_2^j}\nabla\vphi-\int_{B_r(x)\cap G_1^j}(L_j\circ u_j)\,\nabla\vphi
\\\nonumber
&=&-t_j\,\int_{B_r(x)\cap\pa^*G_1^j}\vphi\,\nu_{G_2^j}\,d\H^n-
\int_{B_r(x)\cap\pa^*G_1^j}\vphi\,(L_j\circ u_j)^*\,\nu_{G_1^j}\,d\H^n
\\\nonumber
&&+\int_{B_r(x)\cap G_1^j}\vphi\,\nabla (L_j\circ u_j)
\\\label{formula for Dzj}
&=&\int_{B_r(x)\cap G_1^j}\vphi\,\nabla (L_j\circ u_j)+\int_{B_r(x)\cap\pa^*G_1^j}\vphi\,\Big\{t_j-(L_j\circ u_j)^*\Big\}\,\nu_{G_1^j}\,d\H^n\,,\hspace{1cm}
\end{eqnarray}
and, similarly, that
\begin{eqnarray}\label{formula for Dz}
Dz[\vphi]=\int_{B_r(x)\cap G_1}\vphi\,\nabla (L_0\circ u)+\int_{B_r(x)\cap\pa^*G_1}\vphi\,\Big\{t_0-(L_0\circ u_0)^*\Big\}\,\nu_{G_1}\,d\H^n\,.
\end{eqnarray}
Now, since, by construction, $K_j=\Om\cap\pa^*E_j$ essentially disconnects $B_r(x)$ into $\{G_1^j,G_2^j\}$, we have that $B_r(x)\cap\pa^*G_1^j$ is $\H^n$-contained in $B_r(x)\cap\pa^*E_j$, which, in turn, by \eqref{second trace lemma}, is $\H^n$-contained in $\{u_j^*=t_j\}$; thus, by the Lipschitz continuity of $L_j$ and by $L_j(t_j)=t_j$, we conclude that
\begin{equation}
  \label{fact 2}
  \mbox{$(L_j\circ u_j)^*=t_j$ $\H^n$-a.e. on $B_r(x)\cap\pa^*G_1^j$}\,.
\end{equation}
Combining \eqref{formula for Dzj} and \eqref{fact 2} we conclude that $z_j\in W^{1,2}(B_r(x))$ with
\[
\nabla z_j=1_{G_1^j}\,\nabla (L_j\circ u_j)=1_{G_1^j}\,(L_j'\circ u_j)\,\nabla u_j\,.
\]
As a consequence, ${\rm Lip}(L_j)\le 1$, \eqref{setting 2}, and the fact that $\e_j\to\e_0>0$ combined give $\sup_j\|\nabla z_j\|_{L^2(B_r(x))}<\infty$, and thus, thanks to $z_j\to z$ in $L^1(B_r(x))$, $z\in W^{1,2}(B_r(x))$. In particular $Dz\ll\L^{n+1}$, so that \eqref{formula for Dz} implies $(L_0\circ u)^*=t_0$ $\H^n$-a.e. on $B_r(x)\cap\pa^*G_1$. Since $L_0(t)=t_0$ if and only if $t=t_0$ and $L_0$ is Lipschitz continuous, this proves that $B_r(x)\cap\pa^*G_1$ is $\H^n$-contained in $\{u^*=t_0\}$. Since this is \eqref{ganzo}, we have concluded the proof of the theorem.
\end{proof}

\begin{theorem}\label{theorem epsilon to zero compactness}
Let $\wire$, $\C$, $u_j$, $\e_j$, and $\de_j$ satisfy \eqref{setting 1}, \eqref{setting 2}, \eqref{setting 3}, and \eqref{setting 4}, and assume (as it can always be done up to extracting a further subsequence) that for some $v_0\ge0$ and $\mu$ a Radon measure in $\Om$, as $j\to\infty$, it holds that
\[
\V(u_j;\Om)\to v_0\,,\qquad |\nabla(\Phi\circ u_j)|\,d\L^{n+1}\mres\Om\weakstar \mu\,.
\]
If
\[
\e_0=0\,,
\]
then there exists $(K,E) \in \mathcal{K}_{\rm B}$ such that
\[
u=1_{E}\,,\qquad |E|\leq v_0\,,\qquad\mbox{$K\cup E^{\one}$ is $\C$-spanning $\wire$}\,,
\]
and such that
\begin{equation}
\label{weak star in diffuse compactness theorem}
\mu \geq 2\,\Phi(\de_0)\,\H^n\mres (K\cap E^\zero)+\H^n\mres(\Om\cap\pa^*E)\,.
\end{equation}
Moreover, in the particular case when $v_0=0$, it must be
\begin{equation}
\label{cool liminf}
\liminf_{j\to\infty}\frac{v_j}{\e_j}>0\,.
\end{equation}
\end{theorem}

\begin{remark}\label{divergence remark 2}
  {\rm As a consequence of \eqref{cool liminf} in Theorem \ref{theorem epsilon to zero compactness} we see that if $v(\e)$ and $\de(\e)$ are functions of $\e>0$ such that
  \[
  \lim_{\e\to 0^+}\frac{v(\e)}\e=0\,,\qquad\lim_{\e\to 0^+}\de(\e)=\de_0\in[1/2,1]\,,
  \]
  then
  \begin{equation}
    \label{justifies SFR}
      \lim_{\e\to 0^+}\Upsilon(v(\e),\e,\de(\e))=+\infty\,.
  \end{equation}}
\end{remark}

\begin{proof}[Proof of Theorem \ref{theorem epsilon to zero compactness}] By $\e_j\to 0^+$, \eqref{modica mortola identity}, and \eqref{setting 2} there is $E\subset\Om$ with $|E|\le v_0$ such that $u=1_{E}$, $\{u_j>t\}\to E$ as $j\to\infty$ for a.e. $t\in(0,1)$, and
\begin{equation}\label{first mu bound}
  \mu\ge\H^n\mres(\Om\cap\pa^*E)\,.
\end{equation}
If we set
\begin{align}\notag
    K = \big(\Om\cap\partial^* E\big) \cup \Big\{x\in \Omega \cap E^\zero: \tnl(\mu)(x) \geq 2\Phi(\de_0)\Big\}\,.
\end{align}
then \cite[Theorem 6.4]{maggiBOOK} implies $\mu\mres E^\zero \geq 2\,\Phi(\de_0)\,\H^n \mres (K\cap E^\zero)$, which combined with \eqref{first mu bound} implies \eqref{weak star in diffuse compactness theorem}. We now want to prove that
\begin{equation}
  \label{diffused closure due fine}
  \mbox{$K\cup E^\one$ is $\C$-spanning $\wire$}\,.
\end{equation}
We divide the proof of \eqref{diffused closure due fine} into three steps.

\medskip

\noindent {\it Step one}: We prove that for every $N\in\N$ the Borel set
\begin{align}\notag
    K_N= \big(\Om\cap\partial^* E\big) \cup\Big\{x\in \Omega \cap E^\zero: \tnl (\mu)(x) \geq 2\Phi(\de_0- 1/N)\Big\}\,,
\end{align}
is such that
\begin{equation}
  \label{diffused closure due fine N}
  \mbox{$K_N\cup E^\one$ is $\C$-spanning $\wire$}\,.
\end{equation}
To this end, we apply Lemma \ref{lemma I0} to the interval $I_0=[\de_0-1/N, \de_0-1/(2N)]$ to find  $\{t_j\}_j\subset I_0$ such that $t_j\to t_0\in I_0$, $E_j=\{u_j>t_j\}\toloc E_0$, $\{\Om\cap\pa^*E_j\}_j$ has a partition limit $S_0$ in $\Om$, and \eqref{lemma I0 conclusion} holds. By monotonicity of $t\mapsto\{u_j>t\}$ and since $\{u_j>t\}\toloc E$ for a.e. $t\in(0,1)$, we easily see that $E_0=E$. Hence, we have proved that
\begin{equation}
  \label{compare with}
  \mbox{$E^\one\cup(\Om\cap\pa^*E)\cup S_0$ is $\C$-spanning $\wire$}\,.
\end{equation}
Thanks to Federer's theorem \eqref{federer theorem} and to $\Om\cap\pa^*E\subset K_N$ we can deduce \eqref{diffused closure due fine N} from \eqref{compare with} once we prove that
\begin{equation}\label{compare with 2}
\mbox{$E^{\zero}\cap S_0$ is $\H^n$-contained in $\big\{\theta_*^n(\mu) \geq 2\,\Phi(\de_0-1/N)\big\}$}\,.
\end{equation}
We begin the proof of \eqref{compare with 2} by recalling that $S_0=\bigcup_k\{\Om_k\cap\bigcup_i\pa^*U_i[\Om_k]\}$, with $\{\Om_k\}_k$ as in \eqref{def of Omega i}, $\{U_i^j[\Om_k]\}_i$ the essential partition of $\Om_k$ induced by $\Om\cap\pa^*E_j$, and with $\{U_i[\Om_k]\}_i$ a Lebesgue partition of $\Om_k$ such that $U_i^j[\Om_k]\to U_i[\Om_k]$ as $j\to\infty$ for every $k$ and $i$. Therefore \eqref{compare with 2} can be further reduced to proving that, for each $k$ and $i$,
\begin{align}\label{suffices showing the delta bound}
\mbox{$\Om_k\cap E^{\zero}\cap \pa^* U_i[\Omega_k]$ is $\H^n$-contained in $\big\{\theta_*^n(\mu)\geq 2\,\Phi(\de_0-1/N)\big\}$}\,.
\end{align}
Since $k$ will be fixed from now on, we just set for brevity $U_i=U_i[\Om_k]$, $U_i^j=U_i^j[\Om_k]$, and consider the sets
\begin{eqnarray*}
&&
X^j_0=\{i:|U_i^j|>0\,,\,(U_i^j)^{\one} \subset E_j^\zero\}\,,\qquad X^j_1=\{i:|U_i^j|>0\,,\,(U_i^j)^{\one} \subset E_j^\one\}\,,
\\
&&
X_0=\{i:|U_i^j|>0\,,\,U_i^{\one} \subset E^\zero\}\,,\qquad X_1=\{i:|U_i|>0\,,\,U_i^{\one} \subset E^\one\}\,.
\end{eqnarray*}
Since $\{U_i^j\}_i$ is the essential partition of $\Om_k$ induced by $\Om\cap\pa^*E_j$, it follows by Federer's theorem and by the $\H^n$-containment of $\Om_k\cap\pa^*U_i^j$ into $\Om\cap\pa^*E_j$ that for each $i$ such that $|U_i^j|>0$ we either have $(U_i^j)^\one\subset E_j^\one$ or $(U_i^j)^\one\subset E_j^\zero$. Therefore, if we set
\[
X^j:=\{i:|U_i^j|>0\}\,,\qquad X:=\{i:|U_i|>0\}\,,
\]
then we have $X^j=X_0^j\cup X_1^j$ (with disjoint union); moreover, by $U_i^j\to U_i$ and $E_j\toloc E$ for each $i\in X_0$ there exists $j(i)$ such that $i\in  X_0^j$ for every $j\geq j(i)$, so that $X=X_0\cup X_1$ (also with disjoint union), and thus
\[
\mbox{$\{U_i\}_{i\in X_0}$ is a Lebesgue partition of $\Omega_k \cap E^\zero$}\,,
\]
from which we deduce
\begin{eqnarray}\nonumber
\Om_k\cap E^\zero\cap\bigcup_i \pa^* U_i&\ehn&\Om_k\cap E^\zero\cap\bigcup_{i,i'\in X\,,i\ne i'} \pa^* U_i\cap\pa^*U_{i'}
\\\label{reduced boundary breakdown}
&\ehn&\Om_k\cap E^\zero\cap\bigcup_{i,i'\in X_0\,,i\ne i'} \pa^* U_i\cap\pa^*U_{i'}\,.
\end{eqnarray}
By \eqref{reduced boundary breakdown}, the proof of \eqref{suffices showing the delta bound} can be further reduced to showing that, for every fixed $(i,i')\in X_0\times X_0$ with $i\neq i'$,
\begin{align}\label{suffices showing the delta bound 2}
\mbox{$\Om_k\cap\pa^* U_i \cap \pa^* U_{i'}$ is $\H^n$-contained in $\big\{\tnl(\mu)\geq 2\,\Phi(\de_0-1/N)\big\}$\,.}
\end{align}
To prove \eqref{suffices showing the delta bound 2}, let us fix $i\ne i'\in X_0$, and set
\[
G_1=U_i\,,\qquad G_2=U_{i'}\,,\qquad G_1^j = U_i^j\,,\qquad G_2^j = U_{i'}^j\,.
\]
By  \eqref{second trace lemma} and the $\H^n$-inclusion of $\Om_k\cap\pa^* G_m^j$ into $\Om_k\cap\partial^* E_j$ (recall indeed that $\{U_i^j\}_i$ is the essential partition of $\Om_k$ induced by $K_j=\Om\cap\pa^*E_j$), it follows that
\begin{equation}
  \label{martedi}
  \mbox{$\Omega_k \cap \partial^* G_m^j $ is $\H^n$-contained in $\{u_j^*=t_j\}$}\,,
\end{equation}
for each $m=1,2$; if, correspondingly, we set
\[
u_m^j=u_j\,1_{\Om_k\cap G_m^j}+t_j\,1_{\Om_k\setminus G_m^j}\,,
\]
then by $t_j\to t_0$, $G_m^j \to G_m$ ($m=1,2$), the inclusion $(G_1^j)^\one \cup (G_2^j)^\one \subset E_j^\zero$ for $j\ge j(i)$, and the fact that $E_j^\zero\to\{u=0\}$ as $j\to\infty$, we see that
\begin{align}\label{L1 conv of pieces}
    \mbox{$u_m^j \to t_0\,1_{\Omega_k\setminus G_m}$ in $L^1_{\rm loc}(\Om)$}\,,\qquad\mbox{as $j\to\infty$}\,.
\end{align}
Now, by \eqref{martedi} and the divergence theorem \eqref{div theorem} we have
\begin{equation}
  \label{endless}
  Du_m^j\mres\Om_k=\nabla(\Phi\circ u_j)\,\L^{n+1}\mres(\Om_k\cap G_m^j);
\end{equation}
indeed, for every $\vphi\in C^\infty_c(\Om_k)$ we have
\begin{eqnarray*}
  &&\int_{\Om_k}u_m^j\,\nabla\vphi=\int_{\Om_k\cap G_m^j}u_j\,\nabla\vphi+t_j\int_{\Om_k\setminus G_m^j}\nabla\vphi
  \\
  &=&-\int_{\Om_k\cap G_m^j}\vphi\,\nabla u_j+\int_{\Om_k\cap\pa^*G_m^j}u_j^*\vphi\nu_{G_m^j}-t_j\int_{\Om_k\cap\pa^* G_m^j}\vphi\nu_{G_m^j}
  =-\int_{\Om_k\cap G_m^j}\vphi\,\nabla u_j\,.
\end{eqnarray*}
By combining \eqref{L1 conv of pieces} and \eqref{endless} with the lower semicontinuity of the total variation we find that, for every open set $A\subset\Om_k$,
\begin{eqnarray}\nonumber
\liminf_{j\to\infty}\int_{A\cap G_m^j}|\nabla(\Phi\circ u_j)|
&=&\liminf_{j\to\infty}|D(\Phi\circ u_j)|(A)
\\\label{see in particular}
&\ge&\big|D\big(\Phi\circ(t_0\,1_{A\setminus G_m})\big)\big|(A)=\Phi(t_0)\,P(G_m;A)\,.
\end{eqnarray}
Adding up over $m=1,2$ with $A=B_s(x)\cc\Om_k$ and recalling that $t_0\ge \de_0-1/N$ we find
\begin{eqnarray*}
  \mu(\cl B_{s}(x))
    \!\!\!&\geq&\!\!\!\liminf_{j\to \infty}\int_{B_{s}(x)}|\nabla (\Phi\circ u_j)|
    \ge\sum_{m=1}^2\,\liminf_{j\to \infty}\int_{B_{s}(x)\cap G_m^j}|\nabla (\Phi\circ u_j)|
    \\
    \!\!\!&\geq&\!\!\!\Phi(\de_0-1/N)\, \sum_{m=1}^2 P(G_m;B_{s}(x))\,.
\end{eqnarray*}
As soon as $x\in\Om_k\cap \pa^* G_1 \cap \pa^* G_2=\Om_k\cap \pa^* U_i \cap \pa^* U_{i'}$, if we divide by $\om_n\,s^n$, and let $s\to 0^+$ then we conclude that $\tnl(\mu)(x) \geq 2\,\Phi(\de_0 - 1/N)$, thus proving \eqref{suffices showing the delta bound 2}.

\medskip

\noindent {\it Step two}: We prove that \eqref{diffused closure due fine N} implies \eqref{diffused closure due fine}. Indeed, let us consider the Radon measures $\l_N=\H^n\mres(\Om\cap\pa^*E)+2\,\H^n\mres(\RR(K_N)\cap E^\zero)$ and $\l=\H^n\mres(\Om\cap\pa^*E)+2\,\H^n\mres(\RR(K)\cap E^\zero)$. Since $\{K_N\}_N$ is decreasing in $N$ and $K=\bigcap_N K_N$ we easily see that $\RR(K)=\bigcap_N\RR(K_N)$, and thus deduce that $\l_N\weakstar\l$ as $N\to\infty$. By \eqref{diffused closure due fine N} and Theorem \ref{theorem first closure theorem} we thus conclude that $K_{\rm bk}^\l\cup E^\one$ is $\C$-spanning $\wire$, where
\begin{eqnarray*}
K_{\rm bk}^\l&=&(\Om\cap\pa^*E)\cup\big\{x\in\Om\cap E^\zero:\theta_*^n(\l)\ge 2\big\}
\\
&\ehn&(\Om\cap\pa^*E)\cup (\RR(K)\cap E^\zero)\ehn \RR(K)\setminus E^\one\,.
\end{eqnarray*}
We have thus proved that $\RR(K)\cup E^\one$ is $\C$-spanning $\wire$, which, by Remark \ref{remark RRK cup Eone is spanning too} implies \eqref{diffused closure due fine}.

\medskip

\noindent {\it Step three}: We finally prove\footnote{This result is not needed in the remaining parts of the paper, and its proof can be omitted on a first reading.} that \eqref{cool liminf} holds if $v_0=0$. We shall actually prove a much stronger property, namely, that for every $(\g,\Phi,T)\in\T(\C)$ it holds
\[
\liminf_{j\to\infty}\frac{\V(u_j;T)}{\e_j}>0\,.
\]
If this is not the case then we can find $(\g,\Phi,T)\in\T(\C)$ and a subsequence in $j$ such that $\V(u_j;T)={\rm o}(\e_j)$. Let $z_j=u_j\circ\Phi$, so that $z_j\in W^{1,2}(Y)$ where $Y=\SS^1\times B_1^n$. By the slicing theory for Sobolev functions, if we set $z_j^y(s)=z_j(s,y)$ for $(s,y)\in\SS^1\times B_1^n$, then we can find a Borel set $F$ which is $\H^n$-equivalent to $B_1^n$ and is such that, for each $y\in F$ and each $j$, $z_j^y\in W^{1,2}(\SS^1)$ with $z_j^y=u_j(\Phi(\cdot,y))$ on a set $S_j(y)$ which is $\H^1$-equivalent to $\SS^1$. Notice, in particular, that $z_j^y$ is absolutely continuous on $\SS^1$ for each $y\in F$.

\medskip

We claim that, given $j$, for $\H^n$-a.e. $y\in F$ it holds
\begin{equation}
  \label{prova}
  \mbox{there is $s\in\SS^1$ such that $z_j^y(s)>1/2$}\,.
\end{equation}
If not, there is $F^*\subset F$ with $\H^n(F^*)>0$ and such that, for every $y\in F^*$, $u_j(\Phi(\cdot,y))\le 1/2$ on a set $S_j[y]$ which is $\H^1$-equivalent to $\SS^1$. We thus find that $\Phi(\SS^1\times F^*)$ is a set of positive volume which is $\L^{n+1}$-contained in $\{u_j\le 1/2\}$. In particular, if $t>1/2$, then $\{u_j^*>t\}$ is $\H^n$-disjoint from $\Phi(\SS^1\times F^*)$. However, since $\{u_j^*>t\}$ is $\C$-spanning $\wire$, for $\H^1$-a.e. $s$, $T[s]\cup \{u_j^*>t\}$ is essentially disconnecting $T=\Phi(\SS^1\times B_1^n)$; in particular, $\Phi(\SS^1\times F^*)$ is essentially disconnected by $(T[s]\cup \{u_j^*>t\})\cap \Phi(\SS^1\times F^*)$ which, in turn, is $\H^n$-equivalent to $T[s]\cap\Phi(\SS^1\times F^*)=\Phi(\{s\}\times F^*)$. We have thus concluded that for $\H^1$-a.e. $s\in\SS^1$, $\Phi(\{s\}\times F^*)$ is essentially disconnecting $\Phi(\SS^1\times F^*)$, a contradiction.

\medskip

To conclude, up to modify $F$ on an $\H^n$-null set we can assume that \eqref{prova} holds for every $j$ and every $y\in F$. Next we set
\begin{equation}
  \label{def of Fj}
  F_j=\Big\{y\in F:\H^1(\{z_j^y>1/4\})\le \frac{\V(u_j;T)}M\Big\}\,,
\end{equation}
and notice that
\[
C(\Phi)\,\V(u_j;T)\ge \int_Y V(z_j)\ge \H^n(B_1^n\setminus F_j)\,V(1/4)\,\frac{\V(u_j;T)}M\,,
\]
so that, setting $M=\H^n(B_1^n)\,V(1/4)/2\,C(\Phi)$, we find
\begin{equation}
  \label{Fj short}
  \H^n(F_j)\ge \frac{\H^n(B_1^n)}2\,,\qquad\forall j\,.
\end{equation}
Now, if $y\in F_j$, then by \eqref{prova} there is $s_j^y\in\SS^1$ such that $z_j^y(s_j^y)\ge 1/2$, and $s_j^y$ must lie in a non-empty connected component $I_j^y$ of $\{z_j^y>1/4\}$; by \eqref{def of Fj}, it must be $\H^1(I_j^y)\le \V(u_j,T)/M$, so that, if $j$ is large enough to ensure $\V(u_j,T)/M<\H^1(\SS^1)$, then there must be $t_j^y\in \pa_{\SS^1}I_j^y$. In particular, $z_j^y(t_j^y)=1/4$ and $\dist_{\SS^1}(s_j^y,t_j^y)\le\H^1(I_j^y)\le\V(u_j;T)/M$, which, combined with $z_j^y(s_j^y)\ge 1/2$, give
\[
\int_{\SS^1}|(z_j^y)'|^2\ge\int_{[s_j(y),t_j(y)]}|(z_j^y)'|^2\ge\frac{|z_j^y(s_j^y)-z_j^y(t_j^y)|^2}{\dist_{\SS^1}(s_j^y,t_j^y)}
\ge\frac{M}{16\,\V(u_j;T)}\,.
\]
for every $y\in F_j$, and thus, recalling \eqref{Fj short},
\[
\int_Y|\nabla z_j|^2\ge \int_{F_j}d\H^n_y\int_{\SS^1}|(z_j^y)'|^2\ge\frac{\H^n(B_1^n)}2\,\frac{M}{16\,\V(u_j;T)}\,.
\]
Finally,
\[
\acj(u_j;\Om)\ge \e_j\,\int_\Om|\nabla u_j|^2\ge\frac{\e_j}{C(\Phi)}\,\int_Y|\nabla z_j|^2\ge c(\Phi,M)\,\frac{\e_j}{\V(u_j;T)}
\]
from which we find $\acj(u_j;\Om)\to\infty$ (a contradiction) as $\V(u_j;T)={\rm o}(\e_j)$ as $j\to\infty$.
\end{proof}

\section{Wet and dry soap films as limits of diffused interface soap films}\label{section limsup theorems} In the section we address the approximation in Allen--Cahn energy of minimizers $(K,E)$ of $\Psi_{\rm bk}(v)$ in both the wet ($v>0$) and dry ($v=0$) cases. We shall actually be able to work with a slightly more general class of pairs $(K,E)$, namely, we shall work in the class of those $(K,E)\in\K$, with
\begin{eqnarray} \label{def of K}
&&\K=\Big\{(K,E):\mbox{$K$ is relatively closed and $\H^n$-rectifiable in $\Om$, $E$ is open,}
\\\nonumber
&&\hspace{2.7cm}\mbox{$E$ has finite perimeter in $\Om$, and $\Om\cap\cl(\pa^*E)=\Om\cap\pa E\subset K$}\Big\}\,,
\end{eqnarray}
which also satisfy condition \eqref{good ones} and \eqref{density reco prop} below. In Theorem \ref{theorem approximation of gen soap} we address the wet case ($|E|>0$), while Theorem \ref{theorem approximation of gen soap vuoto} concerns the dry case ($|E|=0$).

\begin{theorem}[Diffused interface approximation of wet soap films]\label{theorem approximation of gen soap} Let $\wire\subset\R^{n+1}$ be compact and such that $\Om=\R^{n+1}\setminus\wire$ has smooth boundary, and let $\C$ be a spanning class for $\wire$. Let $(K,E)\in\KK$ be such that
\begin{equation}
  \label{good ones}
  |E|>0\,,\qquad \mbox{$K\cup E$ is bounded and $\C$-spanning $\wire$}\,,\qquad K\cap E^\one=\varnothing\,,
\end{equation}
and such that there are $c$ and $r_0$ positive such that
\begin{equation}\label{density reco prop}
\h ( K \cap B_r(x) ) \geq c\,r^n\,,
\end{equation}
for every $x\in\cl(K)$ and $r<r_0$.

\medskip

If $\nu:(0,1)\to(0,\infty)$ and $\delta:(0,1)\to (1/2,1)$ are such that, as $\e\to 0^+$,
\[
\nu(\e)\to |E|\,,\qquad \delta(\e)\to \delta_0\in [1/2,1]\,,
\]
then there are $\e_j\to 0^+$ and $\{\{u^j_\e\}_{\e<\e_j}\}_j\subset {\rm Lip}(\Om;[0,1])$ such that $\{u_\e^j>0\}\cc\R^{n+1}$ and
\begin{eqnarray}\label{ueps C spans}
  &&\mbox{$\{u_\e^j\geq \delta(\e)\}$ is $\C$-spanning $\wire$}\,,\qquad\forall \e<\e_j\,,
  \\\label{ueps has volume veps}
  &&\V(u_\e^j;\Om)=\nu(\e)\,,\qquad\hspace{2.35cm}\forall \e<\e_j\,,
  \\\label{ueps goes to E}
  &&\lim_{j\to\infty}\sup_{\e<\e_j}\|u_\e^j-1_E\|_{L^1(\Om)}=0\,,
  \\
  \label{reco}
  &&\lim_{j\to\infty}\sup_{\e<\e_j}\frac{\ac(u_\e^j;\Om)}2=P(E;\Omega) + 2\Phi(\delta_0)\H^n(K \cap E^\zero)\,.
\end{eqnarray}
Moreover, if $\ell$ is finite, then $\Upsilon(v,\e,\delta)$ is finite for every $v>0$, $\e>0$, and $\de\in(0,1]$, and
\begin{equation}
  \label{thus giving}
  \limsup_{\e\to 0^+}\sup_{\de\in(0,1]}\Upsilon(v,\e,\de)\le\Psi_{\rm bk}(v)\,.
\end{equation}
\end{theorem}

\begin{remark}[Choice of spanning set for small $\e$]\label{spanning choice remark}
  Theorem \ref{theorem approximation of gen soap} shows that given any choice of $\de(\e)$ such that $\de(\e)\to[1/2,\de_0)<1$ as $\e\to 0^+$ we are bound to find
  \[
  \limsup_{\e\to 0^+}\Upsilon(v,\e,\delta(\e))\leq  P(E;\Omega) + 2\,\Phi(\delta_0)\,\H^n(K \cap E^\zero)
  <\F_{\rm bk}(K,E) = \Psi_{\rm bk}(v)\,.
  \]
  where we have applied \eqref{reco} to a minimizer $(K,E)$ of $\Psi_{\rm bk}(v)$. This explains why, for recovering the expected/correct limit surface tension energy along the collapsed region one has to require $\de(\e)\to 1^-$ as $\e\to 0^+$.
\end{remark}

\begin{theorem}[Diffused interface approximation of dry soap films]\label{theorem approximation of gen soap vuoto} Let $\wire\subset\R^{n+1}$ be compact and such that $\Om=\R^{n+1}\setminus\wire$ has smooth boundary, and let $\C$ be a spanning class for $\wire$. Let $K$ relatively closed in $\Om$, $\H^n$-rectifiable, bounded, $\C$-spanning $\wire$, and such that there are $c$ and $r_0$ positive with
\begin{equation}\label{density reco prop vuoto}
\h ( K \cap B_r(x) ) \geq c\,r^n\,,
\end{equation}
for every $x\in\cl(K)$ and $r<r_0$.

\medskip

If $\nu:(0,1)\to(0,\infty)$ and $\delta:(0,1)\to (1/2,1)$ are such that, as $\e\to 0^+$,
\begin{eqnarray}\label{hp on veps 1}
\nu(\e)\to 0^+\,,\qquad \delta(\e) \to \delta_0\in [1/2,1]\,,\qquad\frac{\e}{\nu(\e)}\to 0\,,
\end{eqnarray}
then there are $\e_j\to 0^+$ and $\{\{u^j_\e\}_{\e<\e_j}\}_j\subset (W^{1,2}_{\rm loc}\cap{\rm Lip})(\Om;[0,1])$ such that
\begin{eqnarray}\label{ueps C spans vuoto}
  &&\mbox{$\{u_\e^j\geq \delta(\e)\}$ is $\C$-spanning $\wire$}\,,\qquad\forall \e<\e_j\,,
  \\\label{ueps has volume veps vuoto}
  &&\V(u_\e^j;\Om)=\nu(\e)\,,\qquad\hspace{2.35cm}\forall \e<\e_j\,,
  \\
  \label{reco vuoto}
  &&\limsup_{j\to\infty}\sup_{\e<\e_j}\frac{\ac(u_\e^j;\Om)}2\le 2\Phi(\delta_0)\,\H^n(K)\,.
\end{eqnarray}
\end{theorem}

\begin{remark}
  {\rm For the necessity of the third condition in \eqref{hp on veps 1}, see conclusion \eqref{cool liminf} in Theorem \ref{theorem epsilon to zero compactness}.}
\end{remark}

In the proof of both theorems, as well as in the sequel, we will make use of the following elementary lemma:

\begin{lemma}\label{lemma volume fixing}
  {\bf (i):} If $A\subset\R^{n+1}$ is open, $X\in C^\infty_c(A;\R^{n+1})$, and $f_t(x)=x+t\,X(x)$, then there are positive constants $t_0$ and $C_0$ depending on $X$ only, such that, for every $|t|<t_0$, $f_t:A\to A$ is a diffeomorphism, and for every $w\in W^{1,2}(A;[0,1])$ we have
  \begin{eqnarray}\nonumber
    &&\Big|\ac(w\circ f_t;A)-\ac(w;A)
    \\\nonumber
    &&\hspace{1cm}-t\,\int_A\,\e\,|\nabla w|^2+\frac{W(w)}\e\,\Div\,X-2\,\e\,(\nabla w)\cdot\nabla X[\nabla w]\Big|\le C_0\,\ac(w;A)\,t^2\,,
    \\\label{pushing}
    &&\Big|\V(w\circ f_t;A)-\V(w;A)-t\,\int_A\,V(w)\,\Div\,X\Big|\le C_0\,\V(w;A)\,t^2\,,
  \end{eqnarray}
  {\bf (ii):} If $u\in L^1(A;[0,1])$ and $u$ is not constant on $A$, then there are positive constants $\eta_0$, $\beta_0$ and $C_0$ (depending on $A$ and $u$) such that for every $w\in W^{1,2}(A;[0,1])$ with $\|u-w\|_{L^1(A)}\le\beta_0$ and every $|\eta|<\eta_0$ there is a diffeomorphism $f_\eta^w:A\to A$ with $\{f\ne\id\}\cc A$ such that $w_\eta=w\circ f_\eta^w$ satisfies
\[
\V(w_\eta;A)=\V(w;A)+\eta\,,\qquad |\ac(w_\eta;A)-\ac(w;A)|\le C_0\,\ac(w;A)\,|\eta|\,.
\]
\end{lemma}

\begin{proof}
   Statement (i) is a standard consequence of the area formula. Concerning statement (ii), we notice that since $u$ is not constant in $A$ and $V$ is strictly increasing on $[0,1]$, it follows that $V(u)$ is not constant in $A$. In turn, this implies the existence of $X\in C^\infty_c(A;\R^{n+1})$ such that $\int_A\, V(u)\,\Div X>0$, and then one can argue as in \cite[Lemma 29.13, Theorem 29.14]{maggiBOOK}.
\end{proof}

We now prove Theorem \ref{theorem approximation of gen soap} and Theorem \ref{theorem approximation of gen soap vuoto}.

\begin{proof}[Proof of Theorem \ref{theorem approximation of gen soap}] We start by proving \eqref{thus giving}. By Theorem \ref{theorem from MNR1}, $\ell<\infty$ implies the existence of a minimizer $(K,E)$ of $\Psi_{\rm bk}(v)$ ($|E|=v>0$) which satisfies the assumptions in the first part of the statement. Since $u_\e^j$ (corresponding to $\nu(\e)\equiv v$ and $\delta(\e)\equiv 1$) is admissible in $\Upsilon(v,\e,\de)$ for every $\de\in(0,1]$ and $\e<\e_j$, we easily deduce \eqref{thus giving}. The rest of the proof is divided in four steps, %. Before entering into it, we note that
%the term on the right hand side of \eqref{reco} is $\Phi(\delta)\F_{\rm bk}(K,E) + (1-\Phi(\delta))\F_{\rm bk}(K,\Omega \setminus \cl E)$ (note that $(K,\Omega \setminus \cl E)\in \mathcal{K}$ since $\Omega \cap \partial(\Omega \setminus \cl E)=\Omega \cap \partial E=\Omega \cap \cl( \partial^* E)=\Omega \cap \cl(\partial^* (\Omega \setminus \cl E))$). Also, we state
to which we premise the following result and a preliminary remark related to it:

\medskip

\noindent \cite[Proposition 4.13]{V}: {\it If $F$ is a Borel set in $\R^{n+1}$ such that} (a): {\it $\pa F$ is countably $\H^n$-rectifiable, and} (b): {\it there are $c'$ and $r_0'$ positive such that
\begin{eqnarray}
\label{villa hp}
\H^n(B_r(x)\cap\pa F)\geq c'\, r^n\,,\qquad\forall x\in \partial F\,,r <r_0'\,,
\end{eqnarray}
then for every Borel set $A\subset\R^{n+1}$ with
\begin{equation}\label{villa set A}
\H^n(\partial F \cap \partial A)=0\,,
\end{equation}
it holds
\begin{equation}\label{villa conclusion}
\lim_{r\to 0^+} \frac{|(I_r(F)\setminus F)\cap A|}{r} = P(F;A) + 2\H^n (\partial F \cap F^\zero  \cap A)\,,
\end{equation}
where $I_r(F)=\{x\in\R^{n+1}:\dist(x,F)<r\}$.}

\medskip

\noindent {\it A remark on \cite[Proposition 4.13]{V}}: In this remark, let us assume $F$ is closed and satisfies (a) and (b). We first point out that the open set $F^c$ satisfies these assumptions also. This is immediate from the fact that a set and its complement share the same topological boundary. We now record three facts to be used later (in each of them, $A$ satisfies \eqref{villa set A}). First, by \eqref{villa conclusion} applied to $F$ and the fact that $|F\setminus \mathrm{int}\, F|=|\partial F|=0$,
\begin{eqnarray}\notag
\lim_{r\to 0^+} \frac{|(I_r(F)\setminus  F)\cap A|}{r} &=& P(F;A) + 2\H^n (\partial F \cap F^\zero  \cap A) \\ \label{villa conclusion f}
&=& P(\mathrm{int}\, F;A) + 2\H^n(\partial F \cap (\mathrm{int}\,F)^\zero \cap A)\,.
\end{eqnarray}
Second, applying \eqref{villa conclusion} to $F^c$ and again using $|F\setminus \mathrm{int}\, F|=0$, we find
\begin{eqnarray}\notag
  %&=&\lim_{r\to 0^+} \frac{|(I_r(\partial F)\cap F)\cap A|}{r}\\ \notag
   \lim_{r\to 0^+} \frac{|I_r(F^c)\cap (\mathrm{int}\, F)\cap A|}{r}&=& \lim_{r\to 0^+} \frac{|(I_r(F^c)\setminus F^c )\cap A|}{r} \\ \notag
   &=& P(F^c;A) + 2\H^n (\partial(F^c)  \cap (F^c)^\zero  \cap A) \\ \label{villa conclusion g}
   &=&P(\mathrm{int}\, F;A) + 2\H^n (\partial F \cap (\mathrm{int}\, F)^\one \cap A)\,.
\end{eqnarray}
Third, setting $\sdist_{\pa F}(x) = -\dist(x,\partial F)$ if $x\in F$, $\sdist_{\pa F}(x) = \dist(x,\partial F)$ if $x\in F^c$, and $f_\e(s)=\H^n(A\cap\{\sdist_{\pa F}=\e\,s\})$ for $s\in\R$, we claim that, in the limit as $\e\to 0^+$ in the sense of Radon measures on $\R$,
\begin{eqnarray}\label{remark on villa}
f_\e\,\L^1&\weakstar&
\big\{P(\mathrm{int}\, F;A)+2\,\H^n(\partial F\cap (\mathrm{int}\, F)^\zero\cap A)\big\}
\,\L^1\mres (0,\infty)\\ \notag
&&\qquad+ \big\{P(\mathrm{int}\, F;A)+2\,\H^n(\pa F \cap (\mathrm{int}\, F)^\one\cap A)\big\}
\,\L^1\mres (-\infty,0)\,.%\\
%&=&\F_{\rm bk}(K,E)\mathcal{L}^1\mres (0,\infty) + \F_{\rm bk}(K,\Omega \setminus \cl E)\mathcal{L}^1\mres (-\infty,0)\quad\mbox{as $\e\to 0^+$}\,,\,\,\,\,\,\,\,\,
\end{eqnarray}
Indeed, setting for brevity $\a=\big\{P(\mathrm{int}\, F;A)+2\,\H^n(\partial F\cap (\mathrm{int}\, F)^\zero\cap A)\big\}$, we deduce from \eqref{villa conclusion g} that, that for every $b>0$,
\begin{eqnarray*}
  \int_0^b\,f_\e&=&\int_0^b\H^n(A\cap\{\sdist_{\pa F}=\e\,s\})\,ds
  =\frac1\e\,\int_0^{\e\,b}\,\H^n(A\cap\{\sdist_{\pa F}=t\})\,dt
  \\
  &=&b\,\frac{|(I_{\e\,b}(F)\setminus F)\cap A|}{\e\,b}\to b\,\a\,.
\end{eqnarray*}
In particular, for every $(a,b)\subset (0,\infty)$ we have $\int_a^b\,f_\e\to (b-a)\,\a$ as $\e\to 0^+$, and a similar argument based on \eqref{villa conclusion f} completes the proof of \eqref{remark on villa}.

\medskip

\noindent {\it Step one}: We prove that $F=\cl(E \cup K)=\cl E \cup \cl K$ satisfies assumptions (a) and (b) of \cite[Proposition 4.13]{V}. To prove this, we begin by showing that
\begin{align}\label{compute partial F 2}
    \partial F = \cl K \cup \partial E \subset \cl \Omega\,.
\end{align}
The containment in $\cl \Om$ is trivial by $K\cup E\subset \Omega$, so we compute $\partial F$. Towards this end, by the fact (since $E$ is open) that $E$, $\partial E$, and $(\cl E)^c$ partition $\mathbb{R}^{n+1}$, we decompose
\begin{align}\label{decompose partial F}
    \partial F = (\partial F \cap E) \cup (\partial F \cap \partial E) \cup (\partial F \setminus \cl E)
\end{align}
and evaluate each term individually. First, since $E$ is open and $E\subset F$,
\begin{align}\label{nothing in E}
    \partial F \cap E = \varnothing\,.
\end{align}
Second, we claim that $\partial E \subset \partial F$, so that
\begin{align}\label{part e in part f}
    \partial E \cap \partial F = \partial E\,.
\end{align}
To prove $\partial E \subset \partial F$, we must show that if $x\in \partial E$, then
\begin{align}\label{boundary def}
    B_r(x) \cap F \neq \varnothing\quad\textup{and}\quad B_r(x) \setminus F\neq \varnothing\quad \forall r>0\,.
\end{align}
Indeed, if $x\in \partial E$, then $B_r(x) \cap E\neq \varnothing$ for all $r>0$ by definition of $\partial E$, and so $E\subset F$ gives the first condition in \eqref{boundary def}. For the second, we first claim $x\notin E^\one$. Indeed, since $x\in\partial E \subset \partial \Omega \cup K$ and $\partial \Omega \cap E^\one=\varnothing=K \cap E^\one$ (due to $E\subset \Omega$, the smoothness of $\pa \Om$, and our assumption on $K$), we have $x\notin E^\one$. Therefore, noting that $E^\one = F^\one$ (by $|F \setminus E|=|\cl K \cup \partial E| \leq |K \cup \partial \Omega|=0$), we find that $x\notin F^\one$. In turn, this implies that $B_r(x) \setminus F\neq \varnothing$ for all $r>0$ as desired, finishing the proof of \eqref{boundary def} and thus \eqref{part e in part f}. For the last term in \eqref{decompose partial F}, we claim that
\begin{align}\label{boundary outside e}
    \partial F \setminus \cl E = \cl K \setminus \cl E\,.
\end{align}
To see this, we note that by the definition of $F$, $\partial F \setminus \cl E = \partial (\cl K) \setminus \cl E$. Now since $|\cl K|\leq |K| + |\partial \Omega|=0$, we have $\mathrm{int}\,(\cl K)=\varnothing$, and thus
\begin{align}\notag
    \partial (\cl K) = \cl K \setminus \mathrm{int}\,(\cl K) = \cl K\,.
\end{align}
Thus $\partial F \setminus \cl E = \cl K \setminus \cl E$, which is \eqref{boundary outside e}. We may now conclude the proof of \eqref{compute partial F 2}. By \eqref{decompose partial F}, \eqref{nothing in E}, \eqref{part e in part f}, and \eqref{boundary outside e}, $\partial F = \partial E \cup (\cl K \setminus \cl E)=\pa E \cup (\cl K \setminus \mathrm{int}\, E)$, so we would be done with \eqref{compute partial F 2} if $\cl K \cap \mathrm{int}\,E = \varnothing$. Now since $E$ is open, $\cl K \cap \mathrm{int}\, E= \varnothing$ if and only if $K \cap E=\varnothing$. The latter condition holds since, by the openness of $E$ and our assumption $E^\one \cap K = \varnothing$, $E \cap K \subset E^\one\cap K= \varnothing$.
\par
\medskip
As a first consequence of \eqref{compute partial F 2} and $E \cap K=\varnothing$, we have
\begin{align}\notag
    \mathrm{int}\, F = \cl F \setminus \partial F &= (\cl K \cup \cl E) \setminus (\cl K \cup \partial E) \\ \label{computation of int F}
    &= \cl E \setminus (\cl K \cup \partial E) = \mathrm{int}\,E \setminus \cl K=E \setminus K=E\,.
\end{align}
Also \eqref{compute partial F 2}, the relative closedness of $K$ in $\Omega$, and the containment $\Om \cap \partial E \subset K$ give
\begin{equation}
\label{bordi}
\Om\cap\pa F=\Om \cap (\cl K \cup \partial E)=\Om\cap (K\cup \partial E) = \Om \cap K\,.
\end{equation}
Since $\pa\Om$ is $\H^n$-rectifiable, we deduce from \eqref{bordi} and the fact that $K$ is $\H^n$-rectifiable that $\partial F=\cl K\cup \partial E$ is $\H^n$-rectifiable, and thus satisfies (a). The validity of \eqref{villa hp} from (b) at every $x\in \pa F\cap\cl(K)$ is a consequence of assumption \eqref{density reco prop}. The validity of \eqref{villa hp} at $x\in \pa F\setminus\cl(K)\subset \pa E\cap\pa \Om$ can be deduced as follows: with $c$ and $r_0$ as in \eqref{density reco prop}, if $r<r_0$ and there is $y\in B_{r/2}(x)\cap\cl(K)\ne\varnothing$, then by \eqref{density reco prop} and $K\subset \partial F$
\[
\H^n(B_r(x)\cap\pa F)\ge\H^n(B_{r/2}(y)\cap K)\ge c\,(r/2)^n\,;
\]
if, instead, $B_{r/2}(x)\cap\cl(K)=\varnothing$, then $\Om\cap\pa E\subset K$ and $E$ open imply that $B_{r/2}(x)\cap\pa F=B_{r/2}(x)\cap\pa\Om$, and we conclude by the fact that $\H^n(B_r(x)\cap\pa\Om)\ge c_\Om\,r^n$ for every $r<r_\Om$, provided $r_\Om$ and $c_\Om$ are suitable positive constants.

\medskip

\noindent {\it Step two}: We would like to apply \cite[Proposition 4.13]{V} to $F$ and $A=\Om$, although doing so would require checking that $\H^n(\pa\Om\cap\pa F)=0$, something that is potentially false (e.g., if $\H^n(\pa\Om\cap\pa E)>0$). To avoid this difficulty, we ``slightly stretch'' $K$ and $E$ as follows. Since $\Om$ has a (bounded) smooth boundary, there is $t_0>0$ such that if we define $g_t:\Om\to \Om_t:=I_t(\Om)$, $t\in(0,t_0)$, by setting $g_t(x)=x$ for $x\in\Om\cap\{\dist_{\pa\Om}>t\}$ and
\[
g_t(x)=x+\big(\dist_{\pa\Om}(x)-t\big)\,\nabla\dist_{\pa\Om}(x)\,,\qquad x\in\Om\cap\{\dist_{\pa\Om}<t\}\,,
\]
then $g_t$ is diffeomorphism with $g_t\to\id$ and $g_t^{-1}\to\id$ as $t\to 0^+$ (in every $C^k$-norm). Setting $K_t=g_t(K)$ and $E_t=g_t(E)$ we see that $F_t=\cl (K_t \cup E_t)$ satisfies assumptions (a) and (b) of \cite[Proposition 4.13]{V}. Also by \eqref{compute partial F 2} and the fact that $g_t^{-1}(\Omega) \subset \Omega$,
\begin{align}\label{computation of pa Ft}
    \partial (F_t) \cap \Omega = g_t((\cl K \cup \partial E) \cap g_t^{-1}(\Omega)) = g_t(K \cap g_t^{-1}(\Omega))= K_t \cap \Omega\,,
\end{align}
and by \eqref{computation of int F},
\begin{align}\label{computation of int Ft}
    \mathrm{int}\, F_t = E_t\,.
\end{align}
Moreover, since $\dist(g_t(x),\pa\Om)=2\,\dist(x,\pa\Om)-t$ for every $x\in\Om\cap\{\dist_{\pa\Om}<t\}$, we see that
\[
g_t^{-1}\big(\pa \Om\cap\pa (F_t)\big)=g_t^{-1}(\pa\Om)\cap\pa F\subset \big\{\dist_{\pa\Om}=t/2\big\}\cap \pa K \,.
\]
Since $\H^n(K\cap\{\dist_{\pa\Om}=t/2\})=0$ for a.e. $t\in(0,t_0)$ and $g_t^{-1}$ is a Lipschitz map we conclude that
\begin{equation}
  \label{applica villa}
  \H^n\big(\pa\Om\cap\pa (F_t)\big)=0\,,\qquad\mbox{for a.e. $t\in(0,t_0)$}\,.
\end{equation}
As a consequence, we may apply \eqref{villa conclusion f} to $F_t$ with $A=\Om$: by using \eqref{computation of pa Ft} and \eqref{computation of int Ft} to rewrite $\partial F_t \cap \Omega$ and $\mathrm{int}\, F_t$, respectively, for a.e. $t\in(0,t_0)$ and as $r\to 0^+$, we obtain
\begin{eqnarray}
    \nonumber
    \big|\big(I_r(F_t)\setminus F_t\big)\cap \Om\big|
    &=& r\,\big\{P(E_t;\Om) + 2\,\H^n (\Om\cap K_t \cap E_t^\zero)\big\}+r\,{\rm o}_t(1)
    \\\nonumber
    &=&r\,\F_{\rm bk}(K_t,E_t)+r\,{\rm o}_t(1)
    \\
    \label{thanks to villa}
    &=&
    r\,\big(1+\om(t)\big)\F_{\rm bk}(K,E)+r\,\om_t(r)\,,
\end{eqnarray}
where $\om(t)\to 0$ as $t\to 0^+$ and $\om_t(r)\to 0$ as $r\to 0^+$, and where in the last identity we have used the area formula and $\|g_t-\id\|_{C^1}\le C(\Om)\,t$. By the same logic applied to $F_t^c$ with \eqref{villa conclusion g} replacing \eqref{villa conclusion f}, we also have
\begin{eqnarray}\notag
    \big|I_r(F_t^c)\cap F_t \cap \Om\big|&=&
    r\,\big(1+\om(t)\big)[P(E;\Omega) + 2\,\H^n(K \cap E^\one\cap \Omega)]+r\,\om_t(r) \\ \label{thanks to villa 2}
    &=& r\,\big(1+\om(t)\big)P(E;\Omega) +r\,\om_t(r)\,,
\end{eqnarray}
where in the second line we have used our assumption that $K \cap E^\one=\varnothing$. Finally, we notice that
\begin{equation}
  \label{KtEt spans}
  \mbox{$\Omega \cap F_t = K_t\cup E_t$ is $\C$-spanning $\wire$}\,;
\end{equation}
indeed, for every $\g\in\C$, $\g(\SS^1)\cap(K_t\cup E_t)=g_t((g_t^{-1}\circ\g)(\SS^1)\cap(K\cup E))$ where this last set is non-empty since $g_t^{-1}\circ\g$ is homotopic to $\g$ relatively to $\Om$.

\medskip

\noindent {\it Step three}: We prove that for a.e. $t<t_0$ ($t_0$ depending on $(K,E)$), every $M>0$, and every $\e>0$ we can define
\[
u_\e^{M,t} \in{\rm Lip}(\Om;[0,1])\,,\qquad \{u_\e^{M,t}>0\}\cc\R^{n+1}\,,
\]
(see \eqref{def of ueps M t}) in such a way that $\{u_\e^{M,t}\geq \delta(\e)\}$ is $\C$-spanning $\wire$ and
\begin{eqnarray}
\label{uepsMs vol}
  \big|\V(u_\e^{M,t};\Om)-|E|\big|&\le& C\,\big(t\,P(E)+\e\,M\big)\,,
\\\label{uepsMs L1}
  \int_\Om|u_\e^{M,t}-1_E|&\le&  C\,\big(t\,P(E)+\e\,M\big)\,,
\\\label{uepsMs aceps}
 \Big|\frac{\ac(u_\e^{M,t};\Om)}2-P(E;\Omega) - 2\,\Phi(\delta_0)\,\H^n(K \cap E^\zero)\Big|&\le&\om(1/M)+\om(t)+\om_{t,M}(\e)\,,
\end{eqnarray}
where $C$ depends on the data of the theorem, and where $\om(r)$ ($\om_a(r)$) denotes a generic non-negative increasing function such that the limit $\om(r)\to 0$ ($\om_a(r)\to 0$) as $r\to 0^+$ holds at a rate that depends on the data of the theorem (and on the parameter $a$). The construction goes as follows. By the normalization \eqref{W normalization} of $W$, the  Allen--Cahn profile $\eta\in C^\infty(\R;(0,1))$ defined by $-\eta'= \sqrt{W(\eta)}$ on $\R$, $\eta(0)=1/2$, $\eta(-\infty)=1$ and $\eta(+\infty)=0$, is such that
\[
\int_\R (\eta')^2 + W(\eta)= 2\, \int_\R\,\sqrt{W(\eta)}\,|\eta'| = 2\,\int_0^1 \sqrt{W} = 2\,.
\]
Starting from $\eta$, for every $M>0$ we can easily construct $\eta_M\in C^\infty(\R;[0,1])$ with $\{\eta_M=1\}=(-\infty,0]$, $\{\eta_M=0\}=[M,\infty)$ and such that
\begin{eqnarray}
  \label{etaM energy 1}
  \int_{\eta_M^{-1}(\delta_0)}^\infty (\eta_M')^2+W(\eta_M)&=& 2\Phi(\delta_0)+\om(1/M) \\ \label{etaM energy2}
  \int^{\eta_M^{-1}(\delta_0)}_{-\infty} (\eta_M')^2+W(\eta_M)&=& 2(1-\Phi(\delta_0))+\om(1/M)\,,
\end{eqnarray}
where $\om(1/M)\to 0$ as $M\to\infty$. Let $\eta_M^{\delta(\e)}$ be the translation of $\eta_M$ such that $\eta_M^{\delta(\e)}(0)=\delta(\e)$, and similarly for $\delta_0$. Corresponding to $\e$, $M$, and $t$ positive, with $t<t_0$, we now set
\begin{equation}
  \label{def of ueps M t}
  u_\e^{M,t} (x)=\eta_M^{\delta(\e)}\,\Big(\frac{\sdist_{F_t}(x)}\e\Big)\,,\qquad x\in\Om\,.
\end{equation}
In this way, $u_\e^{M,t}\in {\rm Lip}(\Om;[0,1])$ with compact support on $\R^{n+1}$. Since $\{u_\e^{M,t}\geq \delta(\e)\}=\Om\cap F_t=K_t \cup E_t$, by \eqref{KtEt spans} we deduce that $\{u_\e^{M,t}\geq \delta(\e)\}$ is $\C$-spanning $\wire$. Next we notice that since $0=V(0)\le V(t)\le V(1)=1$ for every $t\in[0,1]$, by combining the area formula (to deduce $|E\Delta E_t|\le C\,P(E)\,t$) with \eqref{thanks to villa} we find
\begin{eqnarray}\nonumber
\V(u_\e^{M,t};\Omega \setminus E_t)
  &\le& \big|\big(I_{\e\,M}(F_t)\setminus F_t\big)\cap\Om\big|
  \\ \label{6.17 1}
  &\le&\e\,M\,\big(1+\om(t)\big)\F_{\rm bk}(K,E)+\e\,M\,\om_t(\e\,M)\,.\,
\end{eqnarray}
Similarly, using instead \eqref{thanks to villa 2}, we find
\begin{eqnarray}\nonumber
\hspace{-1cm}\V(u_\e^{M,t};\Omega \cap E_t)
  &=& |E_t|-\mathrm{O}\big(\big|I_{\e\,M}(F_t^c)\cap F_t\cap\Om\big|\big)
  \\ \label{6.17 2}
  &=&|E|-\mathrm{O}\big(C\,P(E)\,t-\e\,M\,\big(1+\om(t)\big)P(E;\Omega) -\e\,M\,\om_t(\e\,M)\big)\,.
\end{eqnarray}
%\begin{eqnarray}\nonumber
%  &&|E|-C\,P(E)\,t\le|E_t|\le \V(u_\e^{M,t})
%  \\ \notag
%  &&\le |E_t|+\big|\big(I_{\e\,M}(K_t\cup E_t)\setminus(K_t\cup E_t)\big)\cap\Om\big|
%  \\ \notag
%  &&\le|E|+C\,P(E)\,t+\e\,M\,\big(1+\om(t)\big)\F_{\rm bk}(K,E)+\e\,M\,\om_t(\e\,M)
%\end{eqnarray}
Together, \eqref{6.17 1} and \eqref{6.17 2} give \eqref{uepsMs vol}; we deduce \eqref{uepsMs L1} similarly. Finally, by the coarea formula,
\begin{eqnarray*}
  \frac{\ac(u_\e^{M,t};\Om)}{2}&=&
  \frac{1}{2}\int_\R\,
  \frac{\big[(\eta_M^{\delta(\e)})'(z/\e)\big]^2+W(\eta_M^{\delta(\e)}(z/\e))}\e\,\,\,\H^n\big(\Om\cap\{\sdist_{F_t}=z\}\big)\,dz
   \\
  &=&
  \frac{1}{2}\int_\R\,
  \big[\big((\eta_M^{\delta(\e)})'\big)^2+W(\eta_M^{\delta(\e)})\big]\,  f_\e^t\,
\end{eqnarray*}
where we have set for brevity
\[
f_\e^t(s)=\H^n\big(\Om\cap\{\sdist_{F_t}=\e\,s\}\big)\,,\qquad s>0\,.
\]
Since, by \eqref{remark on villa} and \eqref{computation of pa Ft}-\eqref{computation of int Ft}, for a.e. $t\in(0,t_0)$, as $\e\to 0^+$,
\begin{eqnarray}\notag
f_\e^t\,\L^1\mres \R&\weakstar&
\big\{P(\mathrm{int}\,F_t;\Omega)+2\,\H^n(\partial F_t\cap (\mathrm{int}\,F_t)^\zero\cap \Omega)\big\}
\,\L^1\mres (0,\infty)\\ \notag
&&\qquad+ \big\{P(\mathrm{int}\,F_t;\Omega)+2\,\H^n(\pa F_t \cap (\mathrm{int}\,F_t)^\one\cap \Omega)\big\}
\,\L^1\mres (-\infty,0)\\ \notag
&=&\big\{P(E_t;\Omega)+2\,\H^n(K_t\cap E_t^\zero\cap \Omega)\big\}
\,\L^1\mres (0,\infty)\\ \notag
&&\qquad+ \big\{P(E_t;\Omega)+2\,\H^n(K_t \cap E_t^\one\cap \Omega)\big\}
\,\L^1\mres (-\infty,0)\quad\mbox{as $\e\to 0^+$}%\\ \notag
%&=:& C_t^+ \mathcal{L}^1\mres (0,\infty) + C_t^- \mathcal{L}^1\mres (-\infty,0) %\\ \notag
%&=&\F_{\rm bk}(K_t,E_t)\mathcal{L}^1\mres (0,\infty) + \F_{\rm bk}(K_t,\Omega \setminus \cl E_t)\mathcal{L}^1\mres (-\infty,0)\,,\,\,\,\,\,\,\,\,
\end{eqnarray}
%\[
%f_\e^t\,\L^1\mres \R\weakstar
%\F_{\rm bk}(K_t,E_t)
%\,\L^1\mres (0,\infty)\,,
%\]
and $\big((\eta_M^{\delta(\e)})'\big)^2+W(\eta_M^{\delta(\e)})$ converges uniformly to $\big((\eta_M^{\delta_0})'\big)^2+W(\eta_M^{\delta_0})$ in $C_c^\infty(\R;[0,1])$ as $\e\to 0^+$, we find in particular that
\begin{eqnarray} \notag
 &&\frac{1}{2}\int_\R\,
\big[\big((\eta_M^{\delta(\e)})'\big)^2+W(\eta_M^{\delta(\e)})\big]\,  f_\e^t\,ds \\ \notag
  &&\qquad=\frac{P(E_t;\Omega)+2\,\H^n(K_t\cap E_t^\zero\cap \Omega)}{2}\int_0^\infty\big((\eta_M^{\delta_0})'\big)^2+W(\eta_M^{\delta_0})\,ds+\om_{t,M}(\e) \\ \notag
  &&\qquad \qquad +\frac{P(E_t;\Omega)+2\,\H^n(K_t \cap E_t^\one\cap \Omega)}{2}\int_{-\infty}^0\big((\eta_M^{\delta_0})'\big)^2+W(\eta_M^{\delta_0})\,ds+\om_{t,M}(\e)\\ \notag
  &&\qquad = \F_{\rm bk}(K_t,E_t)[\Phi(\delta_0)+\om(1/M)+\om_{t,M}(\e)] \\ \notag
  &&\qquad \qquad +\big\{P(E_t;\Omega)+2\,\H^n(K_t\cap E_t^\one\cap \Omega)\big\}[1-\Phi(\delta_0)+\om(1/M)+\om_{t,M}(\e)]\,,
\end{eqnarray}
where $\om_{t,M}(\e)\to 0$ as $\e\to 0^+$ and we have used \eqref{etaM energy 1}-\eqref{etaM energy2}. Since, as noticed in \eqref{thanks to villa} and \eqref{thanks to villa 2},
\begin{eqnarray}\notag
 \F_{\rm bk}(K_t,E_t)&=&(1+\om(t))\,\F_{\rm bk}(K,E)\quad\textup{and}   \\ \notag
 P(E_t;\Omega)+2\,\H^n(K_t\cap E_t^\one\cap \Omega)&=& (1+\om(t))P(E;\Omega)
\end{eqnarray}
we conclude that \eqref{uepsMs aceps} holds.

\medskip

\noindent {\it Step four}: We conclude the proof. Given $j\in\N$, we can find $t_j\to 0^+$ and $M_j\to\infty$ (as $j\to\infty$) depending on the data of the problem, and then $\e_j$ depending on $t_j$, $M_j$ and the data of the problem, such that, $\e_j\to 0^+$ as $j\to\infty$ and, for every $\e<\e_j$,
\[
w_\e^j=u_\e^{M_j,t_j}\in{\rm Lip}(\Om;[0,1])
\]
with $\{w_\e^j>0\}\cc\R^{n+1}$, $\{w_\e^j=1\}$ $\C$-spanning $\wire$, and
\begin{eqnarray}\notag
&&\max\Big\{\big|\V(w_\e^j;\Om)-|E|\big|\,,\int_\Om|w_\e^j-1_E|\,,\Big|\frac{\ac(w_\e^j;\Om)}2- P(E;\Omega) - 2\,\Phi(\delta)\,
\H^n(K \cap E^\zero)\Big|\Big\} \\ \label{end j}
&&\qquad \qquad \le\frac1j\,.
\end{eqnarray}
Next, {\it we use $|E|>0$} to deduce that $u=1_E$ is non-constant in the open set $A=\Om$: by taking $j$ large enough we can thus apply Lemma \ref{lemma volume fixing} to $w=w_\e^j$ and with $\eta=\nu(\e)-\V(w_\e^j;\Om)$ for every $\e<\e_j$; denoting by $f_j$ the resulting diffeomorphism, we set $u_\e^j=w_\e^j\circ f_j$, and notice that, by \eqref{end j}, $u_\e^j$ satisfies \eqref{ueps C spans}, \eqref{ueps has volume veps}, \eqref{ueps goes to E}, and \eqref{reco}.
\end{proof}

Next we turn to the proof of Theorem \ref{theorem approximation of gen soap vuoto}. This will be the first situation where we make use of the properties of the minimizers in the diffused Euclidean isoperimetric problem $\Theta(v,\e)$, consisting of the minimization of $\ac(u):=\ac(u;\R^{n+1})$ under the volume constraint $\V(u):=\V(u;\R^{n+1})=v$; see Appendix \ref{section diffuse isop} for more details.

\begin{proof}[Proof of Theorem \ref{theorem approximation of gen soap vuoto}] Since the assumption ``$|E|>0$'' was not used in the proof of Theorem \ref{theorem approximation of gen soap} until the definition of $u_\e^j$ in step four, we notice that, for a.e. $t\in(0,t_0)$ and for every $M$ and $\e$ positive, if we define $u_\e^{M,t}$ as in \eqref{def of ueps M t}, then \eqref{uepsMs vol} and \eqref{uepsMs aceps} combined with $E=\varnothing$ give
\begin{eqnarray}
\label{uepsMs vol zero}
  \V(u_\e^{M,t};\Om)&\le& C\,\e\,M\,,
\\\label{uepsMs aceps zero}
 \Big|\frac{\ac(u_\e^{M,t};\Om)}2-2\,\Phi(\delta_0)\,\H^n(K)\Big|&\le&\om(1/M)+\om(t)+\om_{t,M}(\e)\,,
\end{eqnarray}
Given $M>0$, there is $\e_*=\e_*(M)>0$ be such that $C\,\e\,M<\nu(\e)$ for every $\e<\e_*$. In particular, for a.e. $t\in(0,t_0)$ and every $M>0$ and $\e<\e_*$ we have
\begin{eqnarray}
\label{uepsMs vol zero j}
  \V(u_\e^{M,t};\Om)\le \nu(\e)%\,,
%\\\label{uepsMs aceps zero}
% \Big|\frac{\ac(u_\e^{M,t})}2-2\Phi(\delta_0)\H^n(K)\Big|&\le&\om(1/M)+\om(t)+\om_{M,t}(\e)\,.
\end{eqnarray}
and \eqref{uepsMs aceps zero}. Now, given $w>0$, let us denote by $\zeta_{\e,w}$ the unique minimizer of the diffused isoperimetric problem
$\Theta(w,\e)$ (see Appendix \ref{section diffuse isop}), and let us consider for a.e. $t\in(0,t_0)$, $M>0$, $\e<\e_*(M)$, $\a>0$, and $\b\in\R$
\[
u_\e^{M,t,\a,\b}=\max\big\{u_\e^{M,t},\zeta_{\e,w_*}\circ \l_\a\circ \tau_\beta\big\}\,,\quad x\in\Om\,,
\]
where $\tau$ is a fixed unit vector in $\R^{n+1}$, $\tau_\beta(x)=x-\beta\,\tau$ and $\l_\a(x)=x/\a$ ($x\in\R^{n+1}$), and where we have set $w_*=w_*(\e,M,t)=\e+\nu(\e)-\V(u_\e^{M,t};\Om)$; we immediately see that
\[
u_\e^{M,t,\a,\b}\in (W_{\rm loc}^{1,2}\cap {\rm Lip})(\Om;[0,1])\,,
\]
and that
\[
\mbox{$\{u_\e^{M,t,\a,\b}\geq\delta(\e)\}$ is $\C$-spanning $\wire$}\,,
\]
since it contains $\{u_\e^{M,t}\geq \delta(\e)\}$, as well as that
\begin{eqnarray}\nonumber
\ac\big(u_\e^{M,t,\a,\b};\Om\big)&\le& \ac(u_\e^{M,t};\Om)+\ac(\zeta_{\e,w_*}\circ\l_\a;\Om)
\\\label{energy1}
&\le&
\ac(u_\e^{M,t};\Om)+\big\{1+C(n,W)\,|\a-1|\big\}\, \Theta(\e,w_*)\,.
\end{eqnarray}
Now, since $\{u_\e^{M,t}>0\}$ is compactly contained in $\R^{n+1}$, we have, uniformly on $|\a-1|<1/2$,
\begin{eqnarray*}
\lim_{\b\to\infty}\V(u_\e^{M,t,\a,\b};\Om)
&=&\V(u_\e^{M,t};\Om)+\V(\zeta_{\e,w_*}\circ\l_\a)
\\
&=&\V(u_\e^{M,t};\Om)+\a^n\,\V(\zeta_{\e,w_*})
\\
&=&\V(u_\e^{M,t};\Om)+\a^n\,\big(\e+\nu(\e)-\V(u_\e^{M,t};\Om)\big)\,.
\end{eqnarray*}
This last expression, evaluated at $\a=1$, is an $\e$ above $\nu(\e)$. In summary, for for a.e. $t\in(0,t_0)$, for every $M>0$, $\e<\e_*(M)$, $\beta>\beta_*(M,t,\e)$ there is $\a(\b,\e)\le 1$, with $\a(\b,\e)\to 1$ as $\b\to\infty$ (uniformly in $\e<\e_*$), such that
\[
\V(u_\e^{M,t,\a(\b,\e),\b};\Om)=\nu(\e)\,.
\]
We can now pick $t_j\to 0^+$, $M_j\to\infty$, $\e_j=\min\{\e_*(M_j),\e^*_j\}\to 0^+$ (where $\e^*_j$ is such that the error $\om_{t,M}$ appearing in \eqref{uepsMs aceps zero} satisfies $\om_{t_j,M_j}(\e_j^*)\to 0^+$), $\beta_j=\beta_*(M_j,t_j,\e_j)$, $\a_j(\e)=\a(\beta_j,\e)$, and define, for every $\e<\e_j$,
\[
u_\e^j=u_\e^{M_j,t_j,\a_j(\e),\b_j}\in (W^{1,2}_{\rm loc}\cap{\rm Lip})(\Om;[0,1])
\]
so that $\{u_\e^j\geq \delta(\e)\}$ is $\C$-spanning $\wire$, $\V(u_\e^j;\Om)=\nu(\e)$, and combining \eqref{uepsMs aceps zero} with \eqref{energy1} such that
\begin{eqnarray*}
\ac(u_\e^j;\Om)&\le&2\,\H^n(K)+\om(1/M_j)+\om(t_j)+\om_{t_j,M_j}(\e_j^*)
\\
&&+\big\{1+C(n,W)\,|\a_j(\e)-1|\big\}\,\Theta(\e,w_j(\e))\,,
\end{eqnarray*}
where $w_j(\e)=w_*(\e,M_j,t_j)=\e+\nu(\e)-\V(u_\e^{M_j,t_j};\Om)$. Now, since \cite{maggirestrepo} implies
\[
\sup_{0<w\le 1}\Theta(\e,w)\le C(n,W)\,\e^{n/(n+1)}\,,\qquad\forall \e<\e_0(n,W)\,,
\]
we conclude from $\om_{t_j,M_j}(\e_j^*)\to 0^+$ and $\a_j(\e)\to 1$ (uniformly in $\e<\e_*$) that \eqref{reco vuoto} holds.
\end{proof}

\section{Lagrange multipliers of diffused interface soap films}\label{section lambdaepsilon} The following theorem is one of the key results of our analysis, as it provides an upper bound on the size of the Lagrange multipliers $\l_j$ appearing in \eqref{ace modified} -- precisely, we show that $\e_j\,\l_j\to 0$ as $j\to\infty$. This information, which is of course interesting in itself, is also useful in the proof of the existence of minimizers of $\Upsilon(v,\e,\de)$ (and the inclusion of the possibility that $\vv_j<v_j$ in the statement is needed in that proof). We notice that our analysis does not touch the very interesting problem of understanding the validity of positive lower bounds on the $|\l_j|$'s. Intuitively, one would indeed expect that they cannot be too small: indeed, since \eqref{ace modified} is compatible with convergence to Plateau-type singularities, it should not be possible to identify it as a too close approximation of the standard Allen--Cahn equation (for which convergence to Plateau-type singularities is indeed impossible).

\begin{theorem}[Lagrange multipliers estimate]\label{theorem lambdaepsj to zero}
If $\wire\subset\R^{n+1}$ is compact, $\C$ is a spanning class for $\wire$, $v_j$, $\e_j$, and $\de_j$ are sequences with
\[
v_j\to 0^+\,,\qquad \e_j\to 0^+\,,\qquad \frac{\e_j}{v_j}\to 0^+\,,\qquad \de_j\to \de_0\in[1/2,1]\,,
\]
as $j\to\infty$, and $u_j$ are minimizers of $\Upsilon(\vv_j,\e_j,\de_j)$ for some $\vv_j\in(0,v_j]$ such that
\begin{equation}
  \label{le hp}
  \acj(u_j;\Om)\le \Upsilon(v_j,\e_j,\de_j)\,,
\end{equation}
then
\begin{equation}
    \label{le tesi}
    \lim_{j\to\infty}\e_j\,\l_j=0\,,
\end{equation}
where $\l_j$ is the Lagrange multiplier in the inner variation Euler--Lagrange equation for $u_j$.
\end{theorem}

\begin{proof} Without loss of generality we can assume that $\e_j\,\l_j$ admits a limit, and then prove \eqref{le tesi} by finding a subsequence $j_N\to\infty$ as $N\to\infty$ such that $\e_{j_N}\,\l_{j_N}\to 0$ as $N\to\infty$. Setting for the sake of brevity ${\rm ac}_\e(u)=\e\,|\nabla u|^2+W(u)/\e$, we are going to achieve this by making a suitable choice of $X\in C_c^\infty(\Om;\mathbb{R}^{n+1})$ in the Euler--Lagrange equation
\begin{equation}
\label{EL inner uj}
\int_\Omega{\rm ac}_{\e_j}(u_j)\,\Div X - 2\,\e_j\,\nabla u_j\,\cdot\, \nabla X[\nabla u_j]=\lambda_j\, \int_\Omega V(u_j)\,\Div X
\end{equation}
satisfied by $u_j$. Since the argument is long, it is convenient to first give an overview of it. In step one we prove the convergence of $\mathrm{ac}_{\e_j}j(u_j)\,\L^{n+1}\mres\Om$ to $2\,\Phi(\de_0)\,\H^n\mres K$, where $K$ is a minimizer of $\ell$, and characterize $K$ as the limit of super/sub-level sets of the $u_j$'s. In step two we blow-up near a regular point of $K$, say, $0$, and identify $r_N\to 0^+$ and $j_N\to\infty$ as $N\to\infty$ such that $K$, at a scale $r_N$ near $0$, is approximately flat and is ``sandwiched'' between two regions $F_N^+$ and $F_N^-$ on each of which $u_{j_N}$ concentrates a $\Phi(\de_0)$-amount of Allen--Cahn energy per unit area (compare with \eqref{acN 1}). In step three we consider the blow-ups $\tilde{u}_N$ of $u_{\e_{j_N}}$ near $0$ at scale $r_N$, and prove that they converge to a double transition from $0$ to $\de_0$ on each side of $K$ near $0$, along a length scale $\s_N\to 0$ (compare with \eqref{limit of tilde UN}). In preparation to further blow-up $\tilde{u}_N$ at the scale $\s_N$ we need to identify first a suitable thin cylinder for performing such blow-up (see the definition of $J_N$ in step four), and then suitable heights that locate the two transitions from $0$ to $\de_0$ (see the definition of $s_N$ and $t_N$ in step five). The resulting blow-ups $\uu_N$ of $\tilde{u}_N$ are defined in step five, where their local convergence to on one-dimensional Allen--Cahn profile is proved (see \eqref{theclaim}), and used to readily infer that $\e_{j_N}\,\l_{j_N}\to 0$ by testing the rescaled Euler--Lagrange equation \eqref{EL inner uj} for $\uu_N$ on a carefully selected vector field.

\medskip

\noindent {\it Step one}: By the assumptions on $\{(v_j,\e_j,\de_j)\}_j$, Theorem \ref{theorem approximation of gen soap vuoto} gives
\begin{equation}
    \label{le 1}
    2\,\Phi(\de_0)\,\ell\ge\limsup_{j\to\infty}\Upsilon(v_j,\e_j,\de_j)
    \ge
    \limsup_{j\to\infty} \acj(u_j;\Om)\big/2\,.
\end{equation}
In particular, $\sup_j\acj(u_j;\Om)<\infty$, so that if we extract a subsequence and denote by $\mu$ the weak-star limit of $|\nabla(\Phi\circ u_j)|\,\L^{n+1}\mres\Om$ as $j\to\infty$, then, by Theorem \ref{theorem epsilon to zero compactness}, the Borel subset of $\Om$ defined by
\[
K=\big\{x\in\Om:\theta_*^n(\mu)(x)\ge2\,\Phi(\de_0)\big\}
\]
is such that $K$ is $\C$-spanning $\wire$ and $\mu\ge2\,\Phi(\delta_0)\,\mathcal{H}^n\mres K$. Combining this last inequality with \eqref{le 1} and \eqref{modica mortola identity} we find that $K$ is a minimizer of $\ell$ (thus $\H^n$-equivalent to $S=\spt\,\H^n\mres K$) and
\begin{eqnarray}
\label{conv to min}
&&2\,\Phi(\delta_0)\,\ell
=\lim_{j\to\infty}\Upsilon(v_j,\e_j,\delta_j)
=\lim_{j\to\infty}\Upsilon(\vv_j,\e_j,\delta_j)
=\lim_{j\to\infty}\acj(u_j;\Om)\big/2\,;
\\
\label{mu limit}
&&\mu=2\,\Phi(\de_0)\,\H^n\mres K
={\rm w^*}\!\!\lim_{j\to\infty}|\nabla(\Phi\circ u_j)|\,\L^{n+1}\mres\Om
={\rm w^*}\!\!\lim_{j\to\infty}\frac{{\rm ac}_{\e_j}(u_j)}2\,\L^{n+1}\mres\Om\,.\hspace{1cm}
\end{eqnarray}
We now notice that, thanks to the minimality property in $\ell$, $K$ can actually be characterized as a partition limit. To see this, let us recall the construction used in the proof of Theorem \ref{theorem epsilon to zero compactness}. There, given $N\in\N$ we applied Lemma \ref{lemma I0} on the interval $I_0^N=[\de_0-(1/2N),\de_0-(1/N)]$ to find $\{t_j^N\}_j\subset I_0^N$ such that, setting $E_j^N=\{u_j>t_j^N\}$, then $\{\Om\cap\pa^*E_j^N\}_j$ admitted a partition limit $S_0^N$ with the property that
\begin{equation}
  \label{compare with 3}
  \mbox{$S_0^N$ is $\C$-spanning $\wire$ and is $\H^n$-contained in $K_N=\{\theta_*^n(\mu)\ge 2\,\Phi(\de_0-1/N)\}$}\,,
\end{equation}
compare with \eqref{compare with} and \eqref{compare with 2}. Having proved that in the present case $\mu=2\,\Phi(\de_0)\,\H^n\mres K$, we see that $K_N=K$, and therefore, by \eqref{compare with 3}, that
\[
\ell=\H^n(K)\ge\H^n(S_0^N)\ge\ell\,.
\]
Hence $S_0^N\shn K$ implies $S_0^N\ehn K$, that is, for every $N\in\N$,
\begin{equation}
  \label{decomposition of K}
  K\ehn\bigcup_k\{\Om_k\cap\bigcup_i\pa^*U_{N,i}[\Om_k]\}\,,
\end{equation}
where\footnote{Should $\{\Om_k\}_k$ be a disjoint family -- something it is definitely not! -- \eqref{decomposition of K} would imply, for each $N\ne N'$ and each $k$, the existence of a bijection $\s$ so that $U_{N,i}[\Om_k]$ is Lebesgue equivalent to $U_{N',\s(i)}[\Om_k]$ for every $i$, and we could thus drop the $N$-dependency from the following arguments. Quite the opposite happens though, since each $\Om_k$ intersects countably many different $\Om_{k'}$'s, and it seems there is no obvious way to trivialize the interaction between $k$ and $N$ in the building up of $K$.} $\{\Om_k\}_k$ is as in \eqref{def of Omega i}, and $\{U_{N,i}[\Om_k]\}_i$ is the limit of the essential partitions $\{U_{N,i}^j[\Om_k]\}_i$ of $\Om_k$ induced by $\Om\cap\pa^*E_j^N$ in the sense that for every $k$, $i$, and $N$, we have  $U_{N,i}^j[\Om_k]\to U_{N,i}[\Om_k]$ as $j\to\infty$. In particular, since, by \eqref{second trace lemma}, $\Om\cap\pa^*E_j^N$ is $\H^n$-equivalent to $\{u_j^*=t_j^N\}$, we find that
\begin{equation}
  \label{better to see that}
  \mbox{$\Om_k\cap\pa^*U_{N,i}^j[\Om_k]$ is $\H^n$-contained in $\{u_j^*=t_j^N\}$}\,,
\end{equation}
where $t_j^N\to t_0^N\in I_0^N$ as $j\to\infty$.

\medskip

\noindent {\it Step two}: From this step onward, we focus our analysis near a regular point of $K$. More precisely, since minimizers of $\ell$ are Almgren minimizing sets in $\Om$, by \cite{Almgren76} $K$ is a smoothly embedded minimal surface in a neighborhood of $\H^n$-a.e. of its points. In particular, setting
\begin{eqnarray*}
&&\QQ_r=\big\{x\in\R^{n+1}:\mbox{$|x_i|<r$ for $i=1,...,n+1$}\big\}=(-r,r)^{n+1}\,,
\\
&&\QQ^n_r=\QQ_r\cap\{x_{n+1}=0\}=(-r,r)^n\times\{0\}\,,
\end{eqnarray*}
denoting by $\pp$ the projection of $\R^{n+1}$ onto $\{x_{n+1}=0\}$, and up to a rigid motion, we can assume that $0\in K$ and that there are $r_0>0$ and a smooth solution to the minimal surfaces equation $f\in C^\infty(\QQ^n_{r_0};(-r_0,r_0))$ with $f(0)=0$, $\nabla f(0)=0$, and
\[
\QQ_r\cap K=\big\{x\in\QQ_r: x_{n+1}=f(\pp(x))\big\}\,,\qquad\forall r\in(0,r_0)\,.
\]
Let us now consider the epigraph and subgraph of $f$ in $\QQ_r$, that is, let us consider
\[
{\rm Epi}(f;r)=\big\{x\in\QQ_r:x_{n+1}>f(\pp(x))\big\}\,,\qquad {\rm Sub}(f;r)=\big\{x\in\QQ_r:x_{n+1}<f(\pp(x))\big\}\,,
\]
so that $\{{\rm Epi}(f;r_0),{\rm Sub}(f;r_0)\}$ is the essential partition of $\QQ_{r_0}$ induced by $K$. By \eqref{decomposition of K} and by the smoothness of $f$, for each $N$ there are $k_N$ and $i_N^+\ne i_N^-$ such that
\[
0\in \Om_{k_N}\cap\pa^*U_{N,i_N^\pm}[\Om_{k_N}]\,.
\]
Proceeding inductively in $N$, we can pick $r_N<\min\{r_0,r_1,...,r_{N-1}\}$ so that
\begin{equation}
  \label{def of rN}
  \lim_{N\to\infty}r_N=0\,,\qquad
\QQ_{2\,r_N}\cc\Om_{k_N}\,,\qquad \H^n(K\cap\pa\QQ_{r_N})=0\,,
\end{equation}
and thus
\[
\QQ_{r_N}\cap U_{N,i_N^+}[\Om_{k_N}]={\rm Epi}(f;r_N)\,,\qquad \QQ_{r_N}\cap U_{N,i_N^-}[\Om_{k_N}]={\rm Sub}(f;r_N)\,.
\]
In particular, for every $N$, as $j\to\infty$ we have
\begin{equation}
  \label{close to epi and sub}
  \QQ_{r_N}\cap U_{N,i_N^+}^j[\Om_{k_N}]\to {\rm Epi}(f;r_N)\,,\qquad \QQ_{r_N}\cap U_{N,i_N^-}^j[\Om_{k_N}]\to {\rm Sub}(f;r_N)\,.
\end{equation}
We now prove two additional properties of $U_{N,i_N^\pm}^j$, see \eqref{diagonal arg} and \eqref{they are lower} below, which will be used to suitably select $\{j_N\}_N$ such that $j_N\to\infty$ as $N\to\infty$:

\medskip

\noindent {\it First}, setting for brevity
\[
\a_{N,j}^\pm=\int_{\QQ_{r_N}\cap U_{N,i_N^\pm}^j[\Om_{k_N}]}\frac{{\rm ac}_{\e_j}(u_j)}2\,,
\]
and noticing that, thanks to \eqref{better to see that}, we can argue as in the proof of \eqref{suffices showing the delta bound 2} in Theorem \ref{theorem epsilon to zero compactness} -- see, in particular, \eqref{see in particular} -- we prove that
\begin{eqnarray}\label{lb dd}
&&\liminf_{j\to\infty}\a_{N,j}^\pm\ge
\liminf_{j\to\infty}\int_{\QQ_{r_N}\cap U_{N,i_N^\pm}^j[\Om_{k_N}]}|\nabla(\Phi\circ u_j)|
\\\nonumber
&&\hspace{2cm}\ge\Phi(t_0^N)\,P(U_{N,i_N^\pm}[\Om_{k_N}];\QQ_{r_N})
\ge\Phi\big(\de_0-(1/2\,N)\big)\,\H^n(K\cap\QQ_{r_N})\,.
\end{eqnarray}
Since \eqref{mu limit} and \eqref{def of rN} give
\begin{eqnarray}\label{ub dd}
\limsup_{j\to\infty}\a_{N,j}^++\a_{N,j}^-\le 2\,\Phi(\de_0)\,\H^n(K\cap\QQ_{r_N})\,,
\end{eqnarray}
by applying \eqref{lb dd} with $m=2$ and by \eqref{ub dd} we find
\begin{eqnarray}\label{ub new}
\limsup_{j\to\infty}\a_{N,j}^+&\le&
2\,\Phi(\de_0)\,\H^n(K\cap\QQ_{r_N})-\liminf_{j\to\infty}\a_{N,j}^-
\\\nonumber
&\le&
\Big\{2\,\Phi(\de_0)-\Phi\Big(\de_0-\frac1{2\,N}\Big)\Big\}\,\H^n(K\cap\QQ_{r_N})
\end{eqnarray}
By similarly applying \eqref{lb dd} with $m=1$ in combination with \eqref{ub dd}, we conclude that,
\begin{eqnarray}
  \label{diagonal arg}
  \Phi\Big(\de_0-\frac1{2\,N}\Big)\,\H^n(K\cap\QQ_{r_N})&\le&\liminf_{j\to\infty}\a_{N,j}^\pm\le\limsup_{j\to\infty}\a_{N,j}^\pm
  \\\nonumber
  &\le&\Big\{2\,\Phi(\de_0)-\Phi\Big(\de_0-\frac1{2\,N}\Big)\Big\}\,\H^n(K\cap\QQ_{r_N})\,.
\end{eqnarray}

\noindent {\it Second}, by \eqref{better to see that}, since each $U_{N,i}^j[\Om_{k_N}]$ is essentially connected, for each $N$, $j$ and $i$,
\[
\mbox{$U_{N,i}^j[\Om_{k_N}]$ is $\L^{n+1}$-contained either in $\{u_j>t_j^N\}$ or in $\{u_j<t_j^N\}$}\,.
\]
Setting
\[
\QQ_r^+=\QQ_r\cap\{x_{n+1}>0\}\,,\qquad \QQ_r^-=\QQ_r\cap\{x_{n+1}<0\}\,,
\]
since $f(0)=0$, $\nabla f(0)=0$, and the smoothness of $f$ imply that $|{\rm Epi}(f;r)\Delta\QQ_r^+|\le C\,r^{n+2}$ for every $r<r_0$,  we see that if $U_{N,i_N^+}^j[\Om_{k_N}]$ is $\L^{n+1}$-contained in $\{u_j>t_j^N\}$, then, by $\V(u_j;\Om)=v_j$ and \eqref{close to epi and sub},
\begin{eqnarray*}
|\QQ_1^+|\,r_N^{n+1}-C\,r_N^{n+2}&\le&|\QQ_{r_N}\cap {\rm Epi}(f;r_N)|\le |\QQ_{r_N}\cap U_{N,i_N^+}^j[\Om_{k_N}]|
+{\rm o}_j^N
\\
&\le&|\{u_j>t_j^N\}|+{\rm o}_j^N\le\frac{\V(u_j;\Om)}{V(1/4)}+{\rm o}_j^N={\rm o}_j^N\,,
\end{eqnarray*}
where ${\rm o}_j^N\to 0$ as $j\to\infty$ at a rate depending on $N$. In particular, up to further decrease the value of $r_0$ (so to have $C\,r_N\le C\,r_0<|\QQ_1^+|/2$ for each $N$), and by repeating the same considerations with ${\rm Sub}(f;r_N)$ in place of ${\rm Epi}(f;r_N)$, we have proved that for each $N$, if $j$ is large enough depending on $N$, then
\begin{eqnarray}
  \label{they are lower}
  \mbox{$U_{N,i_N^\pm}^j[\Om_{k_N}]$ is $\L^{n+1}$-contained in $\{u_j<t_j^N\}$}\,.
\end{eqnarray}

\medskip

\noindent {\it Step three (selection of $\{j_N\}_N$ and first blow-up)}: Using the estimates proved in step two, we can diagonally extract a subsequence $j_N\to\infty$ as $N\to\infty$ such that, if we set
\[
F_N^\pm=\QQ_{r_N}\cap U_{N,i_N^\pm}^{j_N}[\Om_{k_N}]\,,\qquad \s_N=\frac{\e_{j_N}}{r_N}\,,\qquad w_N=\frac{v_{j_N}}{r_N^{n+1}}\,,
\]
then the following holds: first, $\s_N\to0$ and $w_N\to 0$ as $N\to\infty$; second, by \eqref{close to epi and sub},
\begin{eqnarray}\label{FN to epi}
&&\lim_{N\to\infty}\max\Big\{\frac{|F_N^+\Delta {\rm Epi}(f;r_N)|}{r_N^{N+1}},\frac{|F_N^-\Delta {\rm Sub}(f;r_N)|}{r_N^{N+1}}\Big\}=0\,;
\end{eqnarray}
third, by \eqref{mu limit} and \eqref{def of rN} (which yield $\acj(u_j;\QQ_{r_N})\to 2\,\Phi(\delta_0)\,\mathcal{H}^n(K \cap \QQ_{r_N}$)),
\begin{eqnarray}
\label{acN 2}
&&\lim_{N\to\infty}\frac{1}{\H^n(K\cap\QQ_{r_N})}\int_{\QQ_{r_N}}\frac{{\rm ac}_{\e_{j_N}}(u_{j_N})}2=2\,\Phi(\de_0)\,;
\end{eqnarray}
fourth, by \eqref{diagonal arg},
\begin{eqnarray}
\label{acN 1}
&&\lim_{N\to\infty}\frac{1}{\H^n(K\cap\QQ_{r_N})}\int_{\QQ_{r_N}\cap F_N^\pm}\frac{{\rm ac}_{\e_{j_N}}(u_{j_N})}2=\Phi(\de_0)\,;
\end{eqnarray}
and, finally, by \eqref{they are lower},
\begin{eqnarray}
\nonumber
&&\mbox{$\QQ_{r_N}\cap\pa^*F_N^\pm$ is $\H^n$-contained in $\{u_{j_N}^*=t_{j_N}^N\}$}\,,
\\\nonumber
&&\mbox{$F_N^\pm$ is $\L^{n+1}$-contained in $\{u_{j_N}^*<t_{j_N}^N\}$}\,.
\end{eqnarray}
If we now set $\eta_r(y)=r\,y$ ($y\in\QQ_1$) and
\[
G_N^\pm=\frac{F_N^\pm}{r_N}\subset\QQ_1\,,\qquad \de_N=t_{j_N}^N\in I_0^N\,,\qquad \tilde{u}_N=u_{j_N}\circ\eta_{r_N}\,,
\]
then $\de_N\to\de_0$ as $N\to\infty$, while $|{\rm Epi}(f;r)\Delta\QQ_r^+|\le C\,r^{n+2}$ and \eqref{FN to epi} give
\begin{eqnarray*}
  |G_N^+\Delta\QQ^+_1|\le \frac{|F_N^+\Delta{\rm Epi}(f;r_N)|}{r_N^{n+1}}+\frac{|{\rm Epi}(f;r_N)\Delta\QQ_{r_N}^+|}{r_N^{n+1}}\le
   \frac{|F_N^+\Delta{\rm Epi}(f;r_N)|}{r_N^{n+1}}+C\,r_N\,,
\end{eqnarray*}
and an analogous estimate for $G_N^-$, so that,
\begin{eqnarray}
\label{GN to CC1plus}
&&\lim_{N\to\infty}|G_N^\pm\Delta\QQ^\pm_1|=0\,,
\\\label{good trace}
&&\mbox{$\QQ_1\cap\pa^*G_N^\pm$ is $\H^n$-contained in $\{\tilde{u}_N^*=\de_N\}$}\,,
\\\label{parisatlanta1}
&&\mbox{$G_N^\pm$ is $\L^{n+1}$-contained in $\{\tilde{u}_N<\de_N\}$}\,.
\end{eqnarray}
Similarly, taking into account that $\H^n(K\cap\QQ_r)/(2r)^n\to1$ as $r\to 0^+$, that $r_N\to0$ as $N\to\infty$, and that
\begin{eqnarray*}
  \int_{\QQ_{r_N}\cap F_N^\pm}{\rm ac}_{\e_{j_N}}(u_{j_N})
  =r_N^n\,\int_{\QQ_1\cap G_N^\pm}{\rm ac}_{\e_{j_N}/r_N}(\tilde{u}_N)
\end{eqnarray*}
we deduce from \eqref{acN 1} and \eqref{acN 2} that
\begin{eqnarray}
\label{frommm}
\lim_{N\to\infty}\Big|\Phi(\de_0)- 2^{-n}\,\frac{\mathcal{AC}_{\s_N}(\tilde{u}_N;G_N^\pm)}2\Big|=\lim_{N\to\infty}\Big|2\,\Phi(\de_0)- 2^{-n}\,\frac{\mathcal{AC}_{\s_N}(\tilde{u}_N;\QQ_1)}2\Big|=0\,.
\end{eqnarray}
Since $V(t)\ge c\,t^{2(n+1)/n}$ for some $c=c(W)>0$, we have
\[
v_{j_N}\ge\int_{\QQ_{r_N}}u_{j_N}^{2(n+1)/n}=r_N^{n+1}\,\int_{\QQ_1}\tilde{u}_N^{2(n+1)/n}\,,
\]
so that $w_N\to 0$ implies $\tilde{u}_N\to 0$ in $L^1(\QQ_1)$ as $N\to\infty$; hence, taking also \eqref{GN to CC1plus} and \eqref{good trace} into account, we find that, if we set
\begin{equation}\label{utwiddle pm def}
\tilde{u}^\pm_N:=\tilde{u}_N\,1_{G_N^\pm}+\de_N\,1_{\QQ_1\setminus G_N^\pm}\,,
\end{equation}
then $\tilde{u}_N^\pm\in W^{1,2}(\QQ_1)$, and, moreover,
\begin{equation}
  \label{limit of tilde UN}
\lim_{N\to\infty}\int_{\QQ_1}|\tilde{u}_N^+-\de_0\,1_{\QQ_1^-}|=\lim_{N\to\infty}\int_{\QQ_1}|\tilde{u}_N^--\de_0\,1_{\QQ_1^+}|=0\,.
\end{equation}
We also record for future use that, thanks to \eqref{EL inner uj}, it holds
\begin{equation}
  \label{EL inner utildeN}
  \int_{\QQ_1}{\rm ac}_{\s_N}(\tilde{u}_N)\,\Div Y
  - 2\,\s_N\,\nabla\tilde{u}_N\,\cdot\, \nabla Y[\nabla \tilde{u}_N]
  =\frac{\lambda_{j_N}\,\e_{j_N}}{\s_N}\, \int_{\QQ_1} V(\tilde{u}_N)\,\Div Y\,,
\end{equation}
for every $Y\in C^\infty_c(\QQ_1;\R^{n+1})$.

\medskip

\noindent {\it Step four (identification of the second blow-up)}: In this step, we show that for every $N$ large enough, there is $\xi_N\in\QQ_1^n$ such that, setting
\begin{equation}
  \label{QN and JN}
  Q_N=\xi_N+\QQ_{\s_N/2}^n\subset\QQ_1^n\,,\qquad J_N=Q_N\times(-1,1)
\end{equation}
we find
\begin{eqnarray}\label{final L1 close}
&&
\lim_{N\to\infty}\max\Big\{\frac1{\s_N^n}
\int_{J_N}|\tilde{u}_N^+-\de_0\,1_{\QQ_1^-}|\,,
\frac1{\s_N^n}\int_{J_N}|\tilde{u}_N^--\de_0\,1_{\QQ_1^+}|\Big\}=0\,,
\\
\label{final energy concentration}
&&\lim_{N\to\infty}\max\Big\{
\Big|2\Phi(\de_0)-\frac{\mathcal{AC}_{\sigma_N}(\tilde{u}_N;J_N)}{2\,\sigma_N^n}\Big|,
\Big|\Phi(\de_0)-\frac{\mathcal{AC}_{\sigma_N}(\tilde{u}_N;J_N\cap G_N^\pm)}{2\,\sigma_N^n}\Big|\Big\}=0\,.\hspace{1cm}
\end{eqnarray}
(We shall actually prove the existence of {\it many} $\xi_N$ with these properties). To begin with, let us define a sequence $\beta_N\to 0$ as $N\to\infty$ by setting
\[
\beta_N=\max\Big\{\int_{\QQ_1}|\tilde{u}_N^+-\de_0\,1_{\QQ_1^-}|\,,\int_{\QQ_1}|\tilde{u}_N^--\de_0\,1_{\QQ_1^+}|\Big\}\,,
\]
and, denoting by $M_N$ the integer part of $2/\s_N$, we consider a collection
\[
\{Q_{N,i}\}_{i=1}^{M_N^n}\,,\qquad Q_{N,i}=\xi_{N,i}+\QQ_{\s_N/2}^n\,,\qquad \xi_{N,i}\in\QQ_1^n\,,
\]
of disjoint open cubes contained in $\QQ^n_1$, with side length $\s_N$, and such that $\H^n(\QQ^n_1\setminus\bigcup_i Q_{N,i})\le C(n)\,\s_N$ (of course, if $2/\s_N$ itself is an integer, then such cubes can be chosen so that $\H^n(\QQ^n_1\setminus\bigcup_i Q_{N,i})=0$). Noticing that
\begin{equation}
  \label{MNeoN to 1}
  \Big|1-(M_N\,\s_N)^n/\H^n(\QQ_1^n)\Big|=\Big|1-(M_N\,\s_N)^n/2^n\Big|\le C(n)\,\s_N\,,
\end{equation}
we consider the open cylinders $\{J_{N,i}\}_{i=1}^{M_N^n}$, $J_{N,i}:=Q_{N,i}\times(-1,1)\subset\QQ_1$, so that
\[
|J_{N,i}|=2\,\s_N^n\,,\qquad \Big|\QQ_1\setminus\bigcup_i J_{N,i}\Big|\le C(n)\,\s_N\,,
\]
and let
\begin{eqnarray*}
&&\G_N^1=\Big\{1\le i\le M_N^n:\int_{J_{N,i}}|\tilde{u}_N^+-\de_0\,1_{\QQ_1^-}|\le\,\sqrt{\beta_N}\,\s_N^n\Big\}\,,
\\
&&\G_N^2=\Big\{1\le i\le M_N^n:\int_{J_{N,i}}|\tilde{u}_N^--\de_0\,1_{\QQ_1^+}|\le\,\sqrt{\beta_N}\,\s_N^n\Big\}\,.
\end{eqnarray*}
On combining \eqref{MNeoN to 1} with
\begin{eqnarray*}
(M_N^n-\#\,\G_N^1)\,\s_N^n\,\sqrt{\b_N}
\le
\sum_{i\not\in\G_N^1}\int_{J_{N,i}}|\tilde{u}_N^+-\de_0\,1_{\QQ_1^-}|\le\beta_N\,,
\end{eqnarray*}
we find
\begin{equation}\label{GNM exhaustion}
0\le 1-\frac{\#\G_N^m}{M_N^n}\le \frac{\sqrt{\beta_N}}{M_N^n\s_N^n}\le C(n)\,\sqrt{\beta_N}\,,\qquad m=1,2\,,
\end{equation}
provided $N$ is large enough. In particular,
\begin{equation}
  \label{fine a1}
  \lim_{N\to\infty}\frac{\#(\mathcal{G}_N^1 \cap \mathcal{G}_N^2)}{M_N^n}=1\,,
\end{equation}
so that, for $N$ large, the vast majority of the cubes $\{Q_{N,i}\}_{i=1}^{M_N^n}$ satisfies \eqref{final L1 close}. This suggests to consider a class $\G_{N,{\rm good}}$ of cubes in $\G_N^1\cap\G_N^2$ such that \eqref{final energy concentration} holds: more precisely, we set $\mathcal{G}_{N,{\rm good}}=(\mathcal{G}_N^1 \cap \mathcal{G}_N^2)\setminus \mathcal{G}_{N,{\rm bad}}$, where
\begin{equation}\notag
\mathcal{G}_{N,{\rm bad}}= \Big\{i\in \mathcal{G}_N^1 \cap \mathcal{G}_N^2:\frac{\mathcal{AC}_{\sigma_N}(\tilde{u}_N;J_{N,i})}{2\,\sigma_N^n}\ge 2\,\Phi(\delta_0)+\sqrt{\alpha_N}\Big\}\,,
\end{equation}
is defined in dependence of the quantity
\begin{eqnarray}\notag
    \alpha_N&:=&
    \max\Big\{
    \Big|\# (\mathcal{G}_N^1\cap \mathcal{G}_N^2)\,\sigma_N^n\, 2\,\Phi(\delta_0)-\frac{\mathcal{AC}_{\sigma_N}(\tilde{u}_N,\QQ_1)}2 \Big|\,,
    \\
    \notag
    &&\quad\qquad\qquad\qquad \Big[\Phi(\delta_0)-\inf_{i\in \mathcal{G}_N^1 \cap \mathcal{G}_N^2} \frac{\mathcal{AC}_{\sigma_N}(\tilde{u}_N^\pm;J_{N,i} \cap G_N^\pm)}{2\,\sigma_N^n}\Big]_+ \Big\}\,,
\end{eqnarray}
(where $t_+:=\max\{t,0\}$, $t\in\R$). We now claim that
\begin{eqnarray}
  \label{fine fine 1}
  &&\lim_{N\to\infty}\a_N=0\,,
  \\
  \label{fine fine 2}
  &&\lim_{N\to\infty}\frac{\#\mathcal{G}_{N,{\rm good}}}{M_N^n}=1\,,
\end{eqnarray}
and that {\it any choice} of $Q_N\in\G_{N,{\rm good}}$ satisfies \eqref{final L1 close} and \eqref{final energy concentration}.

\medskip

\noindent {\it To prove \eqref{fine fine 1}}: By \eqref{MNeoN to 1}, \eqref{GNM exhaustion} we have $\# (\mathcal{G}_N^1\cap \mathcal{G}_N^2)\,\sigma_N^n\to \H^n(\QQ_1^n)=2^n$ as $N\to\infty$, so that the first quantity in the definition of $\a_N$ is vanishing as $N\to\infty$ thanks to \eqref{frommm}. We are thus left to prove that, if we set
\begin{equation}\label{skinny rectangle lower bound}
   L:=\liminf_{N\to \infty} \inf_{i\in \mathcal{G}_N^1 \cap \mathcal{G}_N^2}
   \frac{\mathcal{AC}_{\sigma_N}(\tilde{u}_N^\pm;J_{N,i} \cap G_N^\pm)}{2\,\sigma_N^n}\,,
\end{equation}
then $L\ge\Phi(\de_0)$. To prove this, let us consider $N_k\to \infty$, $i_k\in \mathcal{G}_{N_k}^1\cap \mathcal{G}_{N_k}^2$, and $m_k\in\{+,-\}$ such that
\[
L= \lim_{k\to \infty} \frac{\mathcal{AC}_{\sigma_{N_k}}(\tilde{u}_{N_k}^{m_k};J_{{N_k},{i_k}} \cap G_{N_k}^{m_k})}{2\,\sigma_{N_k}^n}\,.
\]
Up to extracting a subsequence we can assume that either $m_k=+$ or $m_k=-$ for every $k$, and we can assume without loss of generality to be in the first case. Recalling that $Q_{N,i}=\xi_{N,i}+\QQ_{\s_N/2}^n$ and $J_{N,i}=Q_{N,i}\times(-1,1)$, if we set
\[
J_k^*:=[(J_{N_k,i_k}-\xi_{N_k,i_k})/\s_{N_k}]=\QQ_{1/2}^n\times\Big(-1\big/\s_{N_k},1\big/\s_{N_k}\Big)\,,
\]
and define
\[
 G_k^*=[(G_{N_k}^+-\xi_{N_k,i_k})/\s_{N_k}]\,,\qquad \uu_k(z)=\tilde{u}_{N_k}^+(\s_{N_k}\,z+\xi_{N_k,i_k})\,,\qquad z\in J_k^*\,,
\]
then we have
\begin{equation}\label{skinny reduction}
\frac{\mathcal{AC}_{\sigma_{N_k}}(\tilde{u}_{N_k}^{m_k};J_{{N_k},{i_k}} \cap G_{N_k}^{m_k})}{2\,\sigma_{N_k}^n}=
\frac{\mathcal{AC}_{1}(\uu_k;J_k^*\cap G_k^*)}2\,.
\end{equation}
Now, since $\tilde{u}_N^+\in W^{1,2}(\QQ_1)$ implies $\uu_k\in W^{1,2}(J_k^*)$, and, by definition of $\mathcal{G}_N^1$,
\begin{eqnarray*}
C(n)\,\sqrt{\beta_{N_k}}&\ge&  \frac1{\s_{N_k}^n}\,\int_{J_{N_k,i_k}}|\tilde{u}_{N_k}^+-\de_0\,1_{\QQ_1^-}|
=\s_{N_k}\,\int_{J_k^*}|\uu_k - \delta_0 1_{\QQ_{1/2}^n\times(-1/\sigma_{N_k},0)}|
\\
&=&2\,\int_{\QQ_{1/2}^n}\,d\H^n_y\,\fint_{-1/\s_{N_k}}^{1/\s_{N_k}}|\uu_k(y,t) - \delta_0 1_{(-1/\sigma_{N_k},0)}(t)|\,dt\,,
\end{eqnarray*}
we find that, for $\mathcal{H}^n$-a.e.~ $y\in \QQ_{1/2}^n$, $t\mapsto \uu_k(y,t)$ is absolutely continuous on $(-1/\sigma_{N_k},1/\sigma_{N_k})$ and there are $a_k^y $ and $b_k^y$ in $(-1/\sigma_{N_k},1/\sigma_{N_k})$ such that
\[
\begin{split}
  &\lim_{k\to\infty}a_k^y=-\infty\,,\quad\lim_{k\to\infty}\uu_k(y,a_k^y)=\de_0\,,
  \\
  &\lim_{k\to\infty}b_k^y=+\infty\,,\quad\lim_{k\to\infty}\uu_k(y,b_k^y)=0\,;
\end{split}
\]
as a consequence, by \eqref{modica mortola identity}, by the fact that $\uu_k$ is constant on $J_k^*\setminus G_k^*$, by Fubini's theorem, and by Fatou's lemma, we find that
\begin{eqnarray}\notag
    \lim_{k\to \infty}\frac{\mathcal{AC}_1(\uu_k;J_k^* \cap G_k^*)}2&\geq&
    \liminf_{k\to \infty}\int_{J_k^*\cap G_k^*}|\pa_{x_{n+1}} (\Phi \circ \uu_k)|
    =    \liminf_{k\to \infty}\int_{J_k^*}|\pa_{x_{n+1}} (\Phi \circ \uu_k)|
    \\ \notag
    &\geq& \int_{\QQ_{1/2}^n}\liminf_{k\to \infty} \int_{(a_k^y,b_k^y)} |\pa_{x_{n+1}} (\Phi \circ \uu_k)|\geq \Phi(\delta_0)\,,
\end{eqnarray}
which, combined with \eqref{skinny reduction}, proves that $L\ge \Phi(\de_0)$, and thus that \eqref{fine fine 1} holds.

\medskip

\noindent {\it To prove \eqref{fine fine 2}}: We can estimate that
\[
\frac{\# \mathcal{G}_{N,{\rm bad}}}{M_N^n}
\leq
\frac{1}{M_N^n\,\sqrt{\alpha_N}}
\sum_{i\in \mathcal{G}_{N,{\rm bad}}}
\frac{\mathcal{AC}_{\sigma_N}(\tilde{u}_N;J_{N,i})}{\sigma_N^n}
- 2\,\Phi(\delta_0)
=A_N+B_N\,,
\]
where, by definition of $\a_N$ and since $(M_N\,\sigma_N)^n\to \H^n(\QQ_1^n)$ as $N\to\infty$,
\begin{eqnarray*}
A_N&:=&\frac{1}{M_N^n \sqrt{\alpha_N}}\,
\sum_{i\in \mathcal{G}_N^1 \cap \mathcal{G}_N^2}\Big\{\frac{\mathcal{AC}_{\sigma_N}(\tilde{u}_N;J_{N,i})}{\sigma_N^n}
- 2\,\Phi(\delta_0)\Big\}
\\
&\le&\frac{\mathcal{AC}_{\sigma_N}(\tilde{u}_N;\QQ_1)-\#(\mathcal{G}_N^1 \cap \mathcal{G}_N^2)\,\sigma_N^n\,2\Phi(\de_0)}{M_N^n\,\sigma_N^n\, \sqrt{\alpha_N}}\le C(n)\,\sqrt{\a_N}\,,
\end{eqnarray*}
and, again by definition of $\a_N$,
\begin{eqnarray*}
&&B_N:=-\frac{1}{M_N^n \sqrt{\alpha_N}}\,\sum_{i\in \mathcal{G}_{N,{\rm good}}}\Big\{\frac{\mathcal{AC}_{\sigma_N}(\tilde{u}_N;J_{N,i})}{\sigma_N^n}- 2\,\Phi(\delta_0)\Big\}
\\
&&\le\frac{2}{M_N^n \sqrt{\alpha_N}}\,\sum_{i\in \mathcal{G}_{N,{\rm good}}}\Big\{\Phi(\de_0)-\sum_{m\in\{+,-\}}\frac{\mathcal{AC}_{\sigma_N}(\tilde{u}_N;J_{N,i}\cap G_N^m)}{2\,\sigma_N^n}\Big\}\le C(n)\,\sqrt{\a_N}\,.
\end{eqnarray*}
This shows that $\# \mathcal{G}_{N,{\rm bad}}/M_N^n\to 0$ as $N\to\infty$, and since $\#(\mathcal{G}_N^1\cap \mathcal{G}_N^2)/M_N^n\to 1$ as $N\to\infty$ by \eqref{GNM exhaustion}, we conclude  the proof of \eqref{fine fine 2}.

\medskip

\noindent {\it To conclude the proof of \eqref{final L1 close} and \eqref{final energy concentration}}: For an arbitrary choice of $i(N)\in\G_{N,{\rm good}}$, let $Q_N:=Q_{N,i(N)}$ and $J_N:=Q_N\times(-1,1)$. Since $\G_{N,{\rm good}}\subset\G_N^1\cap\G_N^2$ we deduce the validity of \eqref{final L1 close}. At the same time, by $L\ge\Phi(\de_0)$, by $G_N^+\cap G_N^-=\varnothing$, by $\tilde{u}_N^m=\tilde{u}_N$ on $G_N^m$ for $m\in\{+,-\}$, and by the very definition of $\mathcal{G}_{N,{\rm good}}$, we see that
\begin{eqnarray*}
  \Phi(\de_0)&\le&\liminf_{N\to\infty}\min_{m\in\{+,-\}}\frac{\mathcal{AC}_{\s_N}(\tilde{u}^m_N;J_N\cap G_N^m)}{2\,\s_N^n}
  \le\liminf_{N\to\infty}\frac12\sum_{m\in\{+,-\}}2\,\frac{\mathcal{AC}_{\s_N}(\tilde{u}^m_N;J_N\cap G_N^m)}{2\,\s_N^n}
  \\
  &=&\frac12\,\liminf_{N\to\infty}\frac{\mathcal{AC}_{\s_N}(\tilde{u}_N;J_N\cap (G_N^+\cup G_N^-))}{2\,\s_N^n}\le
  \frac12\,\liminf_{N\to\infty}\frac{\mathcal{AC}_{\s_N}(\tilde{u}_N;J_N)}{2\,\s_N^n}
  \\
  &\le& \Phi(\de_0)+\liminf_{N\to\infty}\frac{\sqrt{\a_N}}2=\Phi(\de_0)\,,
\end{eqnarray*}
which readily implies \eqref{final energy concentration}.

\medskip

\noindent {\it Step five (analysis of the second blow-up)}: With $Q_N=\xi_N+\QQ_{\s_N/2}^n$ and $J_N=Q_N\times(-1,1)$ as in step five, if we now set
\begin{eqnarray*}
\jj_N&=&(J_N-\xi_N)/\s_N=\QQ_{1/2}^n\times(-1/\s_N,1/\s_N)\,,
\\
\jj_N^\pm&=&\jj_N\cap\{x_{n+1}\gtrless 0\}\,,
\\
\mathbf{G}_N^\pm&=&[(J_N\cap G_N^\pm)-\xi_N]/\s_N\,,
\\
\uu_N(z)&=&\tilde{u}_N(\xi_N+\s_N\,z)\,,
\\
\uu_N^\pm(z)&=&\tilde{u}_N^\pm(\xi_N+\s_N\,z)
\\
&=&\uu_N(z)\,1_{\mathbf{G}_N^\pm}(z)+\de_N\,1_{\jj_N\setminus\mathbf{G}_N^\pm}(z)\,,
\end{eqnarray*}
(where $m\in\{+,-\}$ and $z\in\jj_N$), then by \eqref{final L1 close}, \eqref{final energy concentration}, \eqref{good trace}, and \eqref{parisatlanta1}, we find that
\begin{eqnarray}\label{final L1 close rescaled}
&&
\lim_{N\to\infty}\max\Big\{\sigma_N\,\int_{\mathbf{J}_N}|\uu_N^+-\de_0\,1_{\jj_N^-}|\,,
\sigma_N\,\int_{\mathbf{J}_N}|\uu_N^--\de_0\,1_{\jj_N^+}|\Big\}=0\,,
\\ \label{final energy concentration rescaled}
&&\lim_{N\to\infty}\max\Big\{\Big|2\,\Phi(\de_0)-\frac{\mathcal{AC}_{1}(\uu_N;\mathbf{J}_N)}2\Big|,
\Big|\Phi(\de_0)-\frac{\mathcal{AC}_1(\uu_N;\GGG_N^\pm)}2\Big|\Big\}=0\,,
\\ \label{good trace blown}
&&\mbox{$\jj_N\cap\pa^*\GGG_N^\pm$ is $\H^n$-contained in $\{\uu_N^*=\de_N\}$}\,,
\\\label{parisatlanta1 blown}
&&\mbox{$\GGG_N^\pm$ is $\L^{n+1}$-contained in $\{\uu_N<\de_N\}$}\,,
\end{eqnarray}
where $\sigma_N\to 0$ and $\de_N\to\de_0$ as $N\to\infty$. We now formulate a clam, where, for $a\in\R$, we use the notation,
\begin{equation}
  \label{traslazione e coordinata}
  \TT_a(x)=x-a\,e_{n+1}\,,\qquad \ttt(x)=x_{n+1}\,,\qquad\forall x\in\R^{n+1}\,.
\end{equation}

\medskip

\noindent {\it Claim}: there exists a sequence $\{a_N\}_N$ with $\sigma_N\,|a_N|\le 1/2$ such that
\begin{equation}
  \label{theclaim}
\lim_{N\to\infty}
\int_{\QQ_{1/2}^n\times Z}
|\uu_N\circ\TT_{a_N}-q_0\circ\ttt|^2
+|(\nabla\uu_N)\circ\TT_{a_N}-(q_0'\circ\ttt)\,e_{n+1}|^2\,dx=0\,,
\end{equation}
for every $Z\cc(-\infty,Q_0]$. Here we denote by $q:\R\to(0,1)$ the unique solution to
\begin{equation}
  \label{def of q}
  q'=\sqrt{W(q)}\quad\mbox{on $\R$}\,,\qquad q(0)=\frac14\,,\qquad \lim_{t\to-\infty}q(t)=0\,,
\end{equation}
and we define $q_0:\R\to(0,\de_0]$ as
\begin{eqnarray}
  \label{def of q0}
  q_0=1_{(-\infty,Q_0]}\,q+1_{(Q_0,+\infty)}\,\de_0\,,\qquad Q_0=q^{-1}(\de_0)\,.
\end{eqnarray}
(In particular, $Q_0>0$ by $\de_0\ge 1/2$, and $Q_0=+\infty$ if and only if $\de_0=1$.)

\medskip

\noindent {\it Conclusion of the theorem from the claim}: By \eqref{EL inner utildeN}, if we set $\mathbf{\Lambda}_N:=\lambda_{j_N}\,\e_{j_N}$, then
\[
\int_{\jj_N}{\rm ac}_{1}(\uu_N)\,\Div Z
  - 2\,\nabla\uu_N\,\cdot\, \nabla Z[\nabla \uu_N]
  =\mathbf{\Lambda}_N\, \int_{\jj_N} V(\uu_N)\,\Div Z\,,
\]
for every $Z\in C^\infty_c(\jj_N;\R^{n+1})$. In fact, setting $\UU_N(x)=\uu_N(x-a_N\,e_{n+1})$ and noticing that by $\sigma_N\,|a_N|\le1/2$ it holds that
\[
B_1(z_0)\subset\QQ_{1/2}^n\times(-1/2\,\sigma_N,1/2\,\sigma_N)\subset\jj_N-a_N\,e_{n+1}
\]
where $z_0=(Q_0-1)\,e_{n+1}$, we conclude that
\begin{equation}
  \label{EL inner UUN}
  \int_{B_1(z_0)}{\rm ac}_{1}(\UU_N)\,\Div Z
  - 2\,\nabla\UU_N\,\cdot\, \nabla Z[\nabla \UU_N]
  =\mathbf{\Lambda}_N\, \int_{B_1(z_0)} V(\UU_N)\,\Div Z\,,
\end{equation}
for every $Z\in C^\infty_c(B_1(z_0);\R^{n+1})$. Let $\vphi\in C^\infty_c(B_1(z_0))$ be radially symmetric decreasing with respect to $z_0$, and set $Z(x)=\vphi(x)\,e_{n+1}$. In this way, denoting by $\rho$ the reflection of $\R^{n+1}$ with respect to $\{x_{n+1}=Q_0-1\}$ and noticing that $\pa_{n+1}\vphi$ is odd with respect to such reflection, we deduce by \eqref{theclaim} that
\begin{eqnarray*}
&&\lim_{N\to\infty}\int_{B_1(z_0)} V(\UU_N)\,\Div Z=\int_{B_1(z_0)} V(q_0(x_{n+1}))\,\pa_{n+1}\vphi
\\
&&=\int_{B_1(z_0)\cap\{x_{n+1}>Q_0-1\}} \big\{V(q_0(x_{n+1}))-V\big(q_0(\rho(x)_{n+1})\big)\big\}\,\pa_{n+1}\vphi\,,
\end{eqnarray*}
Now, since $q_0$ is strictly increasing on $(-\infty,Q_0]$, and $V$ is strictly increasing on $[0,1]$, we have $V(q_0(x_{n+1}))>V(q_0(\rho(x)_{n+1}))$ for every $x\in B_1(z_0)\cap\{x_{n+1}>Q_0-1\}$. Since $\vphi$ being radially symmetric decreasing with respect to $z_0$ implies that $\pa_{n+1}\vphi\le 0$ on $\{x_{n+1}>Q_0-1\}$, the choice of $\vphi$ can thus be arranged so that
\begin{equation}
  \label{negative}
\lim_{N\to\infty}\int_{B_1(z_0)} V(\UU_N)\,\Div Z<0\,.
\end{equation}
At the same time, by \eqref{theclaim} and $q_0'=\sqrt{W(q_0)}$ we find that
\begin{eqnarray*}
&&\lim_{N\to\infty}\int_{B_1(z_0)}{\rm ac}_{1}(\UU_N)\,\Div Z-2\,\nabla\UU_N\,\cdot\, \nabla Z[\nabla \UU_N]
\\
&&=\int_{B_1(z_0)}\big[W(q_0)-(q_0')^2\big](x_{n+1})\,\pa_{n+1}\vphi=0\,,
\end{eqnarray*}
which combined with \eqref{EL inner UUN} and \eqref{negative} implies that $\mathbf{\Lambda}_N\to 0$ as $N\to\infty$, and completes the proof of the theorem.

\medskip

\noindent {\it Proof of the claim}: We begin by reducing the proof of \eqref{theclaim} to showing the existence of $\{b_N\}_N$ and $\{c_N\}_N$ with $\max\{|b_N|,|c_N|\}\,\sigma_N\le 1/2$ such that for every $W\cc[P_0,\infty)$ and $Z\cc(-\infty,Q_0]$ we have
\begin{eqnarray}
\label{theclaim plus}
  \lim_{N\to\infty}
  \int_{\QQ_{1/2}^n\times W}
  |\uu_N^+\circ\TT_{b_N}-p_0\circ\ttt|^2+|(\nabla\uu_N^+\circ\TT_{b_N})-(p_0'\circ\ttt)\,e_{n+1}|^2\,dx=0\,,
  \\
\label{theclaim minus}
  \lim_{N\to\infty}
  \int_{\QQ_{1/2}^n\times Z}
  |\uu_N^-\circ\TT_{c_N}-q_0\circ\ttt|^2+|(\nabla\uu_N^-\circ\TT_{c_N})-(q_0'\circ\ttt)\,e_{n+1}|^2\,dx=0\,,
\end{eqnarray}
where we set denote by $p:\R\to(0,1)$ the unique solution to
\begin{equation}
  \label{def of p}
  p'=-\sqrt{W(p)}\quad\mbox{on $\R$}\,,\qquad p(0)=\frac14\,,\qquad \lim_{t\to+\infty}p(t)=0\,,
\end{equation}
and we define $p_0:\R\to(0,\de_0]$ as
\begin{eqnarray}
  \label{def of p0}
  p_0=1_{[P_0,+\infty)}\,p+1_{(-\infty,P_0)}\,\de_0\,,\qquad P_0=p^{-1}(\de_0)\,.
\end{eqnarray}
To deduce \eqref{theclaim} from \eqref{theclaim minus} and \eqref{theclaim plus}, we first note that for any $Z \cc (-\infty,Q_0]$ and $W\cc [P_0,\infty)$, since $Z\subset\{q_0<\delta_0\}$ and $\{\uu_N^-<\de_N\}=\GGG_N^-$ (which follows from the definition \eqref{utwiddle pm def} of $\tilde{u}_N^-$), \eqref{theclaim minus} and the analogous statements for $\uu_N^+$ imply that
\begin{equation}
  \label{theclaim Gs}
  \lim_{N\to\infty}|(\QQ_{1/2}^n\times Z)\setminus \TT_{c_N}(\GGG_N^-)|=0=\lim_{N\to\infty}|(\QQ_{1/2}^n\times W)\setminus \TT_{b_N}(\GGG_N^+)|\,.
\end{equation}
As a consequence, the moving rectangles in $\jj_N$ on which $\uu_N^-$ and $\uu_N^+$ are locally converging to translations of $q_0 \circ \ttt$ and $p_0\circ \ttt$, respectively, cannot overlap too much: more precisely, given $C_1,C_2>0$ with $-2Q_0+C_1<C_2$, there exists $N_0(C_1,C_2)$ such that
\begin{align}\label{cNbN relation}
    c_N - b_N \notin [-2Q_0+C_1,C_2] \quad \forall N\geq N_0\,.
\end{align}
Indeed, if \eqref{cNbN relation} did not hold for some $C_1$ and $C_2$, then, up to a subsequence which we do not notate, we would have $c_{N} - b_{N}\to C'\in [-2Q_0+C_1,C_2]$. Then testing \eqref{theclaim Gs} with $W=[P_0+C_1/3,C_2]$ and $Z=[-C_2,Q_0-C_1/3]$ would give
\begin{eqnarray}\notag
0&=&\lim_{N\to\infty}|(\QQ_{1/2}^n\times [P_0+C_1/3,C_2])\setminus\TT_{b_N}(\GGG_N^+)| \\ \notag
&=&\lim_{N\to\infty}|(\QQ_{1/2}^n\times [b_N+P_0+C_1/3,b_N+C_2])\setminus \GGG_N^+|= \lim_{N\to\infty}|(\QQ_{1/2}^n\times (b_N+W))\setminus \GGG_N^+|\,,\\ \notag
0&=&\lim_{N\to\infty}|(\QQ_{1/2}^n\times [-C_2,Q_0-C_1/3])\setminus\TT_{b_N}(\GGG_N^-)|\\ \notag
&=&\lim_{N\to\infty}|(\QQ_{1/2}^n\times [c_N-C_2,c_N+Q_0-C_1/3])\setminus \GGG_N^-|=\lim_{N\to\infty}|(\QQ_{1/2}^n\times (c_N+Z))\setminus \GGG_N^-|\,.
\end{eqnarray}
Now since $b_N+W$ and $c_N+Z$ are intervals of equal length (by $P_0=-Q_0$), we may bound the length of their intersection from below by subtracting endpoints as follows:
\begin{eqnarray}\notag
  &&\liminf_{N\to \infty} \mathcal{L}^1\big((b_N+W)\cap (c_N+Z)\big)\\ \notag
  &&\quad\geq \liminf_{N\to \infty}  \min\{c_N+Q_0-C_1/3 - (b_N+P_0+C_1/3), b_N+C_2 - (c_N-C_2)\} \\ \notag
  &&\quad= \liminf_{N\to \infty} \min \{c_N - b_N + 2Q_0-2C_1/3, b_N - c_N + 2C_2\} \\ \notag
&&\quad=\min\{ C' - (-2Q_0+2C_1/3), 2C_2-C')\}\geq\min\{C_1/3,C_2\}>0\,.
\end{eqnarray}
But this contradicts $|\GGG_N^+ \cap \GGG_N^-| = 0$, since the previous two estimates yield
\begin{eqnarray*}
    0=\lim_{N\to \infty}|\GGG_N^+ \cap \GGG_N^-| \geq \liminf_{N\to \infty}\mathcal{H}^n(\QQ_{1/2}^n)\times\mathcal{L}^1\big((b_N+W)\cap (c_N+Z) \big)>0\,.
\end{eqnarray*}
%\begin{eqnarray*}
%&&\hspace{-.75cm}\liminf_{N\to \infty} \big|\big(\QQ_{1/2}^n\times ([b_N+P_0+C_1/3,b_N+C_2]\cap  [c_N-C_2,c_N+Q_0-C_1/3])\big)\big| \\
%&&\quad\geq
%\liminf_{N\to \infty}\min\{c_N+Q_0-C_1/3 - (b_N+P_0+C_1/3), b_N+C_2 - (c_N-C_2)\} \\
%&&\quad = \liminf_{N\to \infty} \min \{c_N - b_N + 2Q_0-2C_1/3, b_N - c_N + 2C_2\} \\
%&&\quad =\min\{ C' - (-2Q_0+2C_1/3), 2C_2-C')\}\geq\min\{C_1/3,C_2\}>0\,.
%\end{eqnarray*}
Moving on to proving \eqref{theclaim}, since $\uu_N = \uu_N^- 1_{\GGG_N^-} + \uu_N 1_{\jj_N \setminus \GGG_N^-}$, \eqref{theclaim Gs} and \eqref{theclaim minus} imply that for any $Z'\cc(-\infty,Q_0]$,
\begin{align}\label{func conv}
    \lim_{N\to \infty}\int_{\QQ_{1/2}^n\times Z'}|\uu_N \circ \TT_{c_N}- q_0 \circ \ttt|^2 dx =0\,.
\end{align}
To finish proving \eqref{theclaim}, we fix $Z\cc(-\infty,Q_0]$. By \eqref{func conv}, it suffices to show that, for $Z\cc Z'\cc(-\infty,Q_0]$ to be chosen shortly,
\begin{align}\label{grad conv}
    \lim_{N\to \infty}\int_{\QQ_{1/2}^n\times Z'}|\nabla(\uu_N \circ \TT_{c_N})- q_0' \circ \ttt \,e_{n+1}|^2 dx =0\,.
\end{align}
We observe that as a consequence of the uniform energy bound \eqref{final energy concentration rescaled}, any subsequence $\{\uu_{N_k} \circ \TT_{c_{N_k}}\}_k$ of $\{\uu_N \circ \TT_{c_N}\}_N$ has a further subsequence with a weak $W^{1,2}(\QQ_{1/2}^n\times Z')$ limit, which by \eqref{func conv}, must be $q_0 \circ \ttt$. Therefore, the entire sequence $\{\uu_N \circ \TT_{c_N}\}_N$ weakly converges in $W^{1,2}(\QQ_{1/2}^n\times Z')$ to $q_0 \circ \ttt$; in particular $\nabla(\uu_N \circ \TT_{c_N}) \weak q_0' \circ \ttt\,e_{n+1}$ in $L^2(\QQ_{1/2}^n\times Z';\mathbb{R}^{n+1})$. To upgrade this weak convergence to \eqref{grad conv}, it is enough to show
\begin{equation}\label{conv of norms}
     \lim_{N\to \infty}\int_{\QQ_{1/2}^n\times Z'}|\nabla(\uu_N \circ \TT_{c_N}) |^2 dx =  \int_{\QQ_{1/2}^n\times Z'}|q_0' \circ \ttt\,e_{n+1} |^2 dx\,.
\end{equation}
Assuming for contradiction that \eqref{conv of norms} were false, then by the lower-semicontinuity of norms under weak convergence, we would have
\begin{align}\label{ac zw pre}
    \liminf_{N\to \infty}\int_{\QQ_{1/2}^n\times Z'}|\nabla(\uu_N \circ \TT_{c_N}) |^2 dx \geq \tau+ \int_{\QQ_{1/2}^n\times Z'}|q_0' \circ \ttt\,e_{n+1} |^2 dx\,,\quad \tau>0\,.
\end{align}
Since $q$ and $p$ are optimal Allen-Cahn profiles, we may now choose $0<C_1<2Q_0$ small enough and $C_2>0$ large enough such that the set $Z'=[-C_2/2,Q_0-C_1/2]\cc(-\infty,Q_0]$ compactly contains $Z$ and such that $Z'$ and $W'=[P_0+C_1/2,C_2/2]$ satisfy
\begin{align}\label{ac zw}
    \mathcal{AC}_1(q_0 \circ \ttt;\QQ_{1/2}^n\times Z')/2 + \mathcal{AC}_1(p_0\circ \ttt;\QQ_{1/2}^n\times W')/2 \geq 2\,\Phi(\delta_0) - \tau/2\,.
\end{align}
By \eqref{cNbN relation}, for each large $N$, either $c_N-b_N>C_2$ or $b_N-c_N > 2Q_0-C_1$. This implies that the intervals $c_N+Z'$ and $b_N+W'$ are disjoint for large $N$: indeed, two closed intervals $[\alpha_1,\alpha_2]$ and $[\alpha_3,\alpha_4]$ are disjoint if and only if $\max\{\alpha_3-\alpha_2,\alpha_1-\alpha_4\}>0$, and
\begin{eqnarray*}
    && \max\{c_N-C_2/2-(b_N+C_2/2),b_N+P_0+C_1/2 - (c_N+Q_0-C_1/2 )\}\\ \notag
    &&\quad = \max\{ c_N - b_N - C_2, b_N - c_N -2Q_0 + C_1\}>0\quad\textup{for large }N\,.
\end{eqnarray*}
Then using in order \eqref{final energy concentration rescaled} and the disjointness of $c_N+Z'$ and $b_N+W'$; \eqref{ac zw pre}, the lower semicontinuity of norms under weak convergence, and Fatou's lemma; and \eqref{ac zw}, we may compute
\begin{eqnarray}\notag
 2\,\Phi(\delta_0) %&\geq&  \liminf_{N\to \infty}  \mathcal{AC}_1(\uu_N;\QQ_{1/2}\times (c_N+Z'))/2 + \mathcal{AC}_1(\uu_N;\QQ_{1/2}\times (b_N+W'))/2 \\ \notag
 &\geq& \liminf_{N\to \infty}  \mathcal{AC}_1(\uu_N;\QQ_{1/2}\times (c_N+Z'))/2 + \mathcal{AC}_1(\uu_N;\QQ_{1/2}\times (b_N+W'))/2 \\ \notag
 &\geq& \tau+\mathcal{AC}_1(q_0 \circ \ttt;\QQ_{1/2}\times Z')/2 + \mathcal{AC}_1(p_0\circ \ttt;\QQ_{1/2}\times W')/2 \\ \notag
 &\geq& 2\,\Phi(\de_0) + \tau/2\,.
\end{eqnarray}
This is a contradiction since $\tau>0$. Thus \eqref{conv of norms} holds and gives \eqref{theclaim}.

\medskip

The rest of the proof is thus devoted to showing the validity of \eqref{theclaim minus}; the proof of \eqref{theclaim plus} is the same. To begin with, we show that, setting $\nabla'=(\pa_1,...,\pa_n)$, we have
\begin{equation}
  \label{one dim limits}
  \lim_{N\to\infty}\int_{\jj_N}|\nabla' \uu_N^-|^2=0\,,\qquad\limsup_{N\to\infty}\int_{\jj_N}|\nabla(\Phi\circ\uu_N^-)|\le\Phi(\de_0)\,.
\end{equation}
(In particular, local limits of $\uu_N^\pm$ will depend only on the $x_{n+1}$-variable.) Indeed, by \eqref{final L1 close rescaled}, for $\mathcal{H}^n$-a.e.~ $y\in\QQ_{1/2}^n$ we can find
\begin{eqnarray}\label{aNy bNy}
s_N^y\in\Big(-\frac1{\sigma_N},-\frac2{3\,\sigma_N}\Big)\,,\quad\mbox{s.t.}\quad \lim_{N\to\infty}\uu_N^-(y,s_N^y)=0\,,
\\
\nonumber
t_N^y\in\Big(\frac2{3\,\sigma_N},\frac1{\sigma_N}\Big)\,,\quad\mbox{s.t.}\quad \lim_{N\to\infty}\uu_N^-(y,t_N^y)=\de_0\,.
\end{eqnarray}
Since, for $\mathcal{H}^n$-a.e.~ $y\in\QQ_{1/2}^n$,
\begin{equation}
  \label{is abs cont}
  \mbox{$t\mapsto\uu_N^-(y,t)$ is absolutely continuous on $(-1/\sigma_N,1/\sigma_N)$ for every $N$}\,,
\end{equation}
by Fubini's theorem and Fatou's lemma,
\begin{eqnarray}\label{ac strip lower bound}
    \liminf_{N\to \infty}\int_{\jj_N}|\partial_{n+1} (\Phi \circ \uu_N^-)|\ge\int_{\QQ^n_{1/2}}\liminf_{N\to \infty} \int_{(s_N^y,t_N^y)} |\partial_{n+1} (\Phi \circ \uu_N^-)|
    \ge \Phi(\delta_0)\,.
\end{eqnarray}
Combining \eqref{ac strip lower bound} with \eqref{final energy concentration rescaled} and with $\nabla (\Phi \circ \uu_N^-) = 0$ a.e.~ on $\GGG_N^-$, we find
\begin{align}\label{string}
    \Phi(\delta_0)&\leq \liminf_{N\to \infty}\int_{\jj_N}|\partial_{n+1} (\Phi \circ \uu_N^-)|=\liminf_{N\to \infty}\int_{\GGG_N^-}|\partial_{n+1} (\Phi \circ \uu_N^-)|
    \\\nonumber
    &\leq \liminf_{N\to \infty} \int_{\GGG_N^-} \big(|\partial_{n+1}\uu_N^-|^2 + W(\uu_N^-)\big)/2\leq \liminf_{N\to \infty}\int_{\GGG_N^-} \mathrm{ac}_1(\uu_N^-)/2  \leq \Phi(\delta_0)\,.
\end{align}
from which we immediately deduce the first conclusion in \eqref{one dim limits}; the limsup inequality in \eqref{one dim limits} is of course derived along similar lines.

\medskip

Next, we claim that for every $\eta\in(0,\de_0)$, we can find sequences $\{c_N^-\}_N$ and $\{c_N^+\}_N$ with $-1/2\sigma_N\le c_N^-<c_N^+< 1/2\sigma_N$ such that
\begin{equation}
  \label{sntn def}
  \int_{\QQ_{1/2}^n\times(c_N^-,c_N^-+1)}\uu_N^-=\eta\,,\qquad \int_{\QQ_{1/2}^n\times(c_N^+,c_N^++1)}\uu_N^-=\de_0-\eta\,.
\end{equation}
To prove this, let us consider the continuous function $f_N$
defined by
\[
f_N(t)=\int_{\QQ_{1/2}^n\times(t,t+1)}\uu_N^-\,,\qquad |t|<\frac1{\sigma_N}-1\,.
\]
Denoting by $k_N$ the integer part of $1/(2\,\sigma_N)$, we have that
\begin{eqnarray*}
\sigma_N\,\int_{\jj_N}|\uu_N^--\de_0\,1_{\jj_N^+}|
&\ge&\sigma_N\,\Big\{\sum_{k=0}^{k_N}\int_{\QQ_{1/2}^n\times(k,k+1)}|\uu_N^--\de_0|
+\sum_{h=0}^{k_N}\int_{\QQ_{1/2}^n\times(-h-1,-h)}\uu_N^-\Big\}
\\
&\ge&\sigma_N\,\Big\{\sum_{k=0}^{k_N}|f_N(k)-\de_0|+\sum_{h=0}^{k_N}f_N(-h)\Big\}\,,
\end{eqnarray*}
so that, by $\sigma_N\,k_N\ge (1/2)-\sigma_N\ge 1/4$ ($N$ large), we see that
\[
\lim_{N\to\infty}\Big\{\inf_{0\le k\le k_N}|f_N(k)-\de_0|+\inf_{0\le h\le k_N}f_N(-h)\Big\}=0\,.
\]
In particular, for every $N$ large enough depending on $\eta$, we can find integers $k,h\in(0,1/(2\sigma_N))$ such that $f_N(k)> \de_0-\eta$ and $f_N(-h)<\eta$, and $c_N^+$ and $c_N^-$ satisfying \eqref{sntn def} are then found by the intermediate value theorem.

\medskip

By \eqref{final energy concentration rescaled}, both $\{\uu_N^-\circ\TT_{c_N^-}\}_N$ and $\{\uu_N^-\circ\TT_{c_N^+}\}_N$ are bounded in $W^{1,2}(\QQ_{1/2}^n\times Z)$ for every $Z\cc\R$. Thus, up to extracting a not relabeled subsequence and taking into account \eqref{one dim limits}, we can find $q_0^-,q_0^+\in (W^{1,2}_{\rm loc}\cap C^{0,1/2}_{\rm loc})(\R;[0,\de_0])$ such that, as $N\to\infty$,
\begin{eqnarray}\label{p0 conv}
&&
\left\{
\begin{split}
&\uu_N^-\circ\TT_{c_N^-}\weak q_0^-\circ\ttt\,,
\\
&\uu_N^-\circ\TT_{c_N^+}\weak q_0^+\circ\ttt\,,\qquad\mbox{weakly in $W^{1,2}(\QQ_{1/2}^n\times Z)$, $\forall\,Z\cc\R$}\,,
\end{split}
\right .
\\\label{ae convergence}
&&
\left\{
\begin{split}
  &\uu_N^-\circ\TT_{c_N^-}\to q_0^-\circ\ttt\,,
  \\
  &\uu_N^-\circ\TT_{c_N^+}\to q_0^+\circ\ttt\,,\qquad \mbox{a.e. on $\QQ_{1/2}^n\times\R$}\,.
\end{split}
\right .
\end{eqnarray}
In particular, by \eqref{sntn def} and \eqref{p0 conv}, we have that
\begin{equation}
  \label{limit average}
  \int_{(0,1)}q_0^-=\eta\,,\quad\int_{(0,1)}q_0^+=\de_0-\eta\,,\quad\mbox{$\exists\, s_*,t_*\in(0,1)$ s.t.}
  \left\{\begin{split}
  &q_0^-(s_*)=\eta\,,
  \\
  &q_0^+(t_*)=\de_0-\eta\,.
  \end{split}\right .
\end{equation}
On combining \eqref{limit average}, \eqref{ae convergence} and \eqref{is abs cont} we conclude that, for $\H^n$-a.e. $y\in\QQ_{1/2}^n$,
\begin{equation}
  \label{good limits}
  \eta=\lim_{N\to\infty}\uu_N^-(y,c_N^-+s_*)\,,\qquad\de_0-\eta=\lim_{N\to\infty}\uu_N^-(y,c_N^++t_*)\,.
\end{equation}
Since $(q_0^\pm)'\in L^2(\R)$, the limits
\[
q_0^\pm(+\infty):=\lim_{t\to+\infty}q_0^\pm(t)\,,\qquad q_0^\pm(-\infty):=\lim_{t\to-\infty}q_0^\pm(t)\,,
\]
exist, and we can prove that we always have
\begin{equation}
  \label{llll}
  \{q_0^+(\pm\infty),q_0^-(\pm\infty)\}\subset\{0,\de_0\}\,.
\end{equation}
Indeed, by \eqref{final energy concentration rescaled}, by $\GGG_N^-=\{\uu_N^-<\de_N\}$ and $\uu_N=\uu_N^-$ on $\GGG_N^-$, by Fatou's lemma, and by \eqref{ae convergence},
\begin{eqnarray*}
\Phi(\delta_0)&=&\lim_{N\to\infty}\frac{\mathcal{AC}_1(\uu_N;\GGG_N^-)}2
\ge\limsup_{N\to\infty}\frac12\int_{\{\uu_N^-<\de_N\}}W(\uu_N^-)
\\
&=&\limsup_{N\to\infty}\frac12\int_{\{\uu_N^-\circ\TT_{c_N^-}<\de_N\}}W\big(\uu_N^-\circ\TT_{c_N^-}\big)
\\
&\ge&\frac12\int_{\{\uu_N^-\circ\TT_{c_N^-}<\de_N\}}\liminf_{N\to\infty}W\big(\uu_N^-\circ\TT_{c_N^-}\big)\ge
\int_{\{q_0^-<\de_0\}}W(q_0^-)\,,
\end{eqnarray*}
and, similarly, we prove that $\int_{\{q_0^+<\de_0\}}W(q_0^+)<\infty$; since $W>0$ on $(0,\de_0)$, we deduce \eqref{llll}. We now need to split the proof of \eqref{theclaim minus} depending on the value of the limit inferior of $c_N^+-c_N^-\ge0$.

\medskip

\noindent {\it Proof of \eqref{theclaim minus} in the case when}
\begin{equation}
  \label{case one hp}
  \liminf_{N\to\infty}c_N^+-c_N^-\ge 1\,.
\end{equation}
In this case we must have $c_N^-+s_*<c_N^++t_*$ for every $N$ large enough. Therefore, by \eqref{good limits}, we thus find
\begin{align}
\label{in between bound}
\liminf_{N\to \infty}\int_{\QQ_{1/2}^n\times [c_N^- + s_*,c_N^++t_*]}|\nabla (\Phi \circ \uu_N^-)| \geq \Phi(\delta_0-\eta) - \Phi(\eta)\,.
\end{align}
We can now improve \eqref{llll} by showing that, if $\eta$ is sufficiently small in terms of $\de_0$ and $W$, then
\begin{equation}\label{computation of Lpm}
q_0^+(+\infty)=\delta_0\,,\qquad q_0^-(-\infty)=0\,.
\end{equation}
Indeed, by using, in the order, \eqref{limit average} and the absolute continuity of $\Phi\circ q_0^\pm$, \eqref{p0 conv} and the lower semicontinuity of the total variation, the second conclusion in \eqref{one dim limits}, and \eqref{in between bound}, we find
\begin{eqnarray}\nonumber
&&|\Phi(q_0^-(-\infty))-\Phi(\eta)|+|\Phi(q_0^+(+\infty))-\Phi(\de_0-\eta)|\le\int_{-\infty}^{s_*}|(\Phi\circ q_0^-)'|+\int_{t_*}^\infty|(\Phi\circ q_0^+)'|
\\\nonumber
&&\hspace{3cm}\le\liminf_{N\to\infty}\int_{\QQ_{1/2}^n\times[(-1/\sigma_N,c_N^-+s_*)\cup(c_N^++t_*,1/\sigma_N)]}|D(\Phi\circ\uu^-_N)|
\\\label{the chain}
&&\hspace{3cm}\le\Phi(\de_0)-\liminf_{N\to\infty}\int_{\QQ_{1/2}^n\times(c_N^-+s_*,c_N^++t_*)}|D(\Phi\circ\uu^-_N)|
\\\nonumber
&&\hspace{3cm}\le\Phi(\de_0)-\big(\Phi(\de_0-\eta)-\Phi(\eta)\big)\,.
\end{eqnarray}
This inequality, combined with \eqref{llll} and considered with $\eta$ small enough in terms of $\de_0$ and $W$, implies \eqref{computation of Lpm}.

\medskip

Armed with \eqref{computation of Lpm}, we prove that $q_0^+$ is strictly increasing on $\{0<q_0^+<\de_0\}$. Indeed, if this were not the case, then
\begin{equation}
  \label{look at t1t2}
  \mbox{$\exists\, t_1<t_2$ s.t. $q_0^+(t_1)\ge q_0^+(t_2)\in (0,\delta_0)$}\,.
\end{equation}
Setting for brevity $I_N[t_1,t_2]=\QQ_{1/2}^n\times[t_1+c_N^+,t_2+c_N^+]$ we could then estimate
\begin{eqnarray}\nonumber
\mathcal{A}_1(\uu_N^-;\GGG_N^-)&\ge&\mathcal{A}_1(\uu_N^-;\GGG_N^-\setminus I_N[t_1,t_2])+\mathcal{A}_1(\uu_N^-;\GGG_N^-\cap I_N[t_1,t_2])
\\\label{tab}
&\ge&2\,\int_{\jj_N\setminus I_N[t_1,t_2]}|D(\Phi\circ\uu_N^-)|+\int_{\GGG_N^-\cap I_N[t_1,t_2]}W(\uu_N^-)\,.
\end{eqnarray}
Now, with $s_N^y$ and $t_N^y$ as in \eqref{aNy bNy}, we have
\[
s_N^y<-\frac2{3\,\sigma_N}<-\frac1{2\sigma_N}<c_N^-<c_N^+<\frac1{2\,\sigma_N}<\frac2{3\,\sigma_N}<t_N^y\,,
\]
so that, thanks to $\s_N\to0^+$ as $N\to\infty$, for $N$ large enough we find $s_N^y<c_N^++t_1$ and $t_N^y>c_N^++t_2$ $\H^n$-a.e. on $\QQ_{1/2}^n$. In particular,
\begin{eqnarray}\nonumber
&&\liminf_{N\to\infty}\int_{\jj_N\setminus I_N[t_1,t_2]}|D(\Phi\circ\uu_N^-)|
\\\nonumber
&&\ge\liminf_{N\to\infty}
 \int_{\QQ_{1/2}^n}d\H^n_y\int_{c_N^++t_2}^{t_N^y}|D(\Phi\circ\uu_N^-)|+
 \int_{\QQ_{1/2}^n}d\H^n_y\int_{s_N^y}^{c_N^++t_1}|D(\Phi\circ\uu_N^-)|
 \\\nonumber
 &&\ge\int_{\QQ_{1/2}^n}\liminf_{N\to\infty}|\Phi(\uu_N^-(y,t_N^y))-\Phi(\uu_N^-(y,c_N^++t_2))|
 \\\nonumber
 &&+\int_{\QQ_{1/2}^n}\liminf_{N\to\infty}|\Phi(\uu_N^-(y,c_N^++t_1))-\Phi(\uu_N^-(y,s_N^y))|\,d\H^n_y
 \\\nonumber
 &&=|\Phi(\de_0)-\Phi(q_0^+(t_2))|+|\Phi(q_0^+(t_1))-\Phi(0)|
 \\\label{primo liminf}
 &&=\Phi(\de_0)-\Phi(q_0^+(t_2))+\Phi(q_0^+(t_1))\ge\Phi(\de_0)\,,
\end{eqnarray}
where we have used $q_0^+(t_2)\le q_0^+(t_1)$ and the fact that $\Phi$ is increasing on $[0,1]$. Concerning the second term in \eqref{tab} we notice that, thanks to $\GGG_N^-=\{\uu_N^-<\de_N\}$ we find
\[
\int_{\GGG_N^-\cap I_N[t_1,t_2]}W(\uu_N^-)=\int_{\{\uu_N^-\circ\TT_{c_N^+}<\de_N\}\cap[\QQ_{1/2}^n\times(t_1,t_2)]} W(\uu_N^-\circ\TT_{c_N^+})\,,
\]
so that
\begin{equation}
  \label{altro liminf}
  \liminf_{N\to\infty}\int_{\GGG_N^-\cap I_N[t_1,t_2]}W(\uu_N^-)\ge\int_{\{q_0^+<\de_0\}\cap(t_1,t_2)}W(q_0^+)=:c\,,
\end{equation}
where $c>0$ since $q_0^+(t_1),q_0^+(t_2)\in(0,\de_0)$, $W>0$ on $(0,\de_0)$, and $q_0^+$ is continuous. By combining \eqref{final energy concentration rescaled} with \eqref{tab}, \eqref{primo liminf} and \eqref{altro liminf} we conclude that
\[
\Phi(\de_0)=\lim_{N\to\infty}\mathcal{A}_1(\uu_N^-;\GGG_N^-)\ge\Phi(\de_0)+c>\Phi(\de_0)\,,
\]
thus obtaining a contradiction with \eqref{look at t1t2}.

\medskip

Having proved that $q_0^+$ is strictly increasing on the (possibly unbounded) interval $(a,b)=\{0<q_0^+<\de_0\}$, we see that it must be $q_0^+(t)\to 0^+$ as $t\to a^+$. Since $q_0^+(+\infty)=\de_0$ implies $q_0^+(t)\to\de_0$ as $t\to b^-$, we find
\begin{eqnarray*}
  \Phi(\de_0)=\int_{\{0<q_0^+<\de_0\}}|(\Phi\circ q_0^+)'|\le\int_{\{0<q_0^+<\de_0\}}\mathcal{AC}_1(q_0^+\circ\ttt)
  \le\Phi(\de_0)\,,
\end{eqnarray*}
so that $(q_0^+)'=\sqrt{W(q_0)}$ on $\{0<q_0^+<\de_0\}$. By the Cauchy uniqueness theorem and by definition \eqref{def of q} of the Allen--Cahn one-dimensional profile $q$, since $q_0^+$ is not identically equal to $0$, there exists $c_0\in\R$ such that
\[
q_0^+(t)=q(t-c_0)\,,\qquad\forall t\in\{0<q_0^+<\de_0\}=\big(-\infty,c_0+Q_0\big)\,,
\]
where, we recall, $Q_0:=q^{-1}(\de_0)$. Therefore, if we set $c_N=c_N^++c_0$, then we conclude that
\begin{equation}
  \label{the weak}
  \mbox{$\uu_N^-\circ\TT_{c_N}\weak q_0\circ\ttt$ weakly in $W^{1,2}(\QQ_{1/2}^n\times Z)$, $\forall\,Z\cc\R$}\,,
\end{equation}
where $q_0=1_{(-\infty,Q_0]}\,q+1_{(Q_0,\infty)}\,\de_0$ is defined as in \eqref{def of q0}. The fact that
\begin{eqnarray*}
\Phi(\de_0)&=&\int_{\QQ_{1/2}^n\times(-\infty,Q_0)}|\nabla(q_0\circ\ttt)|^2=\frac{\mathcal{AC}_1(q_0\circ\ttt;\QQ_{1/2}^n\times(-\infty,Q_0))}2
\\
&\le&\liminf_{N\to\infty}\frac{\mathcal{AC}_1(\uu_N^-\circ\TT_{c_N};\QQ_{1/2}^n\times(-\infty,Q_0))}2
\le\Phi(\de_0)\,,
\end{eqnarray*}
implies that
\[
\lim_{N\to\infty}\int_{\QQ_{1/2}^n\times(-\infty,Q_0)}|\nabla(\uu_N^-\circ\TT_{c_N})|^2=\int_{\QQ_{1/2}^n\times(-\infty,Q_0)}|\nabla(q_0\circ\ttt)|^2\,,
\]
and thus to improve the weak convergence stated in \eqref{the weak} to the strong convergence claimed in \eqref{theclaim minus}. This complete the proof of \eqref{theclaim minus} in the case when \eqref{case one hp} holds.

\medskip

\noindent {\it Proof of \eqref{theclaim minus} in the case when}
\begin{equation}
  \label{case two hp}
  \liminf_{N\to\infty}c_N^+-c_N^-< 1\,.
\end{equation}
In this case, up to extracting a subsequence in $N$, we can assume that $c_N^+-c_N^-\to d_0\in[0,1)$. In particular, we find
\[
q_0^+(t)=q_0^-(t-d_0)\,,\qquad\forall t\in\R\,,
\]
so that, by \eqref{limit average},
\begin{equation}
  \label{limit average case one 1}
  \int_{(-d_0,-d_0+1)}q_0^+=\eta\,,\qquad\int_{(0,1)}q_0^+=\de_0-\eta\,,
\end{equation}
and, in particular, $d_0\in(0,1)$ by taking $\eta<\de_0/2$. With respect to the case when \eqref{case one hp} holds, we now define $s_*$ and $t_*$ differently, by claiming that
\begin{equation}
  \label{s star and t star case one}
  \mbox{$\exists\,s_*<t_*$ s.t. $q_0^+(s_*)=\eta$ and $q_0^+(t_*)=\de_0-2\,\eta$}\,.
\end{equation}
Indeed, by \eqref{limit average case one 1} and the continuity of $q_0^+$, there is $s_*\in(-d_0,-d_0+1)$ such that $q_0^+(s_*)=\eta$, while again by \eqref{s star and t star case one},
\[
\int_{-d_0+1}^1 q_0^+=\int_0^1 q_0^+-\int_0^{-d_0+1}q_0^+=\de_0-\eta-\int_0^{-d_0+1} q_0^+\ge\de_0-2\,\eta\,,
\]
so that there is also $r_*\in(-d_0+1,1)$ (and thus, such that $r_*>s_*$) with the property that
\[
q_0^+(r_*)=\frac1{1-(-d_0+1)}\int_{-d_0+1}^1 q_0^+\ge\frac{\de_0-2\,\eta}{d_0}\ge\de_0-2\,\eta\,.
\]
Since $q_0^+(s_*)=\eta$, we can find $t_*\in(s_*,r_*)$ such that $q_0^+(t_*)=\de_0-2\,\eta$ and conclude the proof of \eqref{s star and t star case one}.

\medskip

Having proved \eqref{s star and t star case one}, we find
\begin{align}
\label{in between bound 2}
\liminf_{N\to \infty}\int_{\QQ_{1/2}^n\times [c_N^+ + s_*,c_N^++t_*]}|\nabla (\Phi \circ \uu_N^-)| \geq \Phi(\delta_0-2\,\eta) - \Phi(\eta)\,,
\end{align}
which we can use  in place of \eqref{in between bound} to argue as in \eqref{the chain} to prove \eqref{computation of Lpm} (that is, $q_0^+(+\infty)=\delta_0$ and $q_0^-(-\infty)=0$). Without modifications one repeats the rest of the proof, from establishing that $q_0^+$ is strictly increasing on $\{0<q_0^+\}$ to showing that $q_0^+$ is a translation of $q$ and a strong $W^{1,2}$-limit of $\TT_{c_N}\circ\uu_N^-$ for some finite translation $c_N$ of $c_N^+$. The proof of the theorem is thus complete.
\end{proof}

\section{A hierarchy of Plateau problems (Proof of Theorem \ref{theorem upsilon existence convergence})}\label{section existence and convergence for diffused} This section is devoted the proof of Theorem \ref{theorem upsilon existence convergence}. We premise a simple lemma that will be used in the course of the proof.

\begin{lemma}\label{lemma isoperimetry at infinity} If $\wire\subset\R^{n+1}$ is compact, $\Om=\R^{n+1}\setminus \wire$, $\e_j\to 0^+$ as $j\to\infty$, $u_j\in W^{1,2}_{\rm loc}(\Om;[0,1])$, $u_j\to u$ in $L^1_{\rm loc}(\Om)$ with $\sup_j\,\acj(u_j;\Om)<\infty$, $\V(u_j;\Om)\to v>0$, and
\begin{eqnarray*}
{\rm ac}_{\e_j}(u_j)\,\L^{n+1}\mres\Om\weakstar\mu\,,\qquad\mbox{as Radon measures in $\Om$}\,,
\end{eqnarray*}
as $j\to\infty$, then
\[
\liminf_{j\to\infty}\acj(u_j;\Om)\ge\mu(\Om)+\Theta(v,\e)\,,
\]
where $\Theta(v,\e)$ is the diffused Euclidean isoperimetric profile if $\e>0$ (see Appendix \ref{section diffuse isop}) and $\Theta(v,0)=(n+1)\,\om_{n+1}^{1/(n+1)}\,v^{n/(n+1)}$.
\end{lemma}

\begin{proof}[Proof of Lemma \ref{lemma isoperimetry at infinity}]
This can be proved by combining a localization argument which is quite common in the theory of concentration-compactness (and whose details are therefore omitted) with \eqref{modica mortola identity} and the Euclidean isoperimetric inequality in its sharp and diffused versions.
\end{proof}

\begin{proof}
  [Proof of Theorem \ref{theorem upsilon existence convergence}] Momentarily assuming the validity of conclusion (i), we prove conclusions (ii), (iii), and (iv). To this end, let $\{u_j\}_j$  be a sequence of minimizers for $\Upsilon(v_j,\e_j,\delta_j)$, where, after achieving $\diam\wire=1$ by scaling, we can assume that, as $j\to\infty$,
  \begin{eqnarray}
   \label{e/v vanishes}
   &&v_j \to v_0 \in [0,\tau_0]\,,\qquad\frac{\e_j}{v_j} \to 0^+\,,
   \\
   \label{delta to 1}
   && \delta_j\to \delta_0\in\Big[\frac12,1\Big]\,,\qquad \min\{1-\delta_j, v_j\}\to 0\,.
   \end{eqnarray}
   In this way, the situation of conclusion (ii) is met when $v_0>0$ (in which case \eqref{delta to 1} forces $\de_0=1$) and the situation of conclusion (iii) is met when $v_0=0$; and, in both cases, \eqref{e/v vanishes} implies that $\e_j\to 0^+$. Since $\ell<\infty$, $\Om$ has smooth boundary, and $\e_j/v_j \to 0^+$, we can apply Theorem \ref{theorem approximation of gen soap} (if $v_0>0$) and Theorem \ref{theorem approximation of gen soap vuoto} (if $v_0=0$), to find
\begin{align}\label{we did the construction}
    \limsup_{j\to \infty} \Upsilon(v_j,\e_j,\delta_j)
    \leq
    \begin{cases}
    \Psi_{\rm bk}(v_0)\,, & \mbox{if  $v_0>0$}\,,
    \\
    2\,\Phi(\delta_0)\,\ell\,, & \mbox{if $v_0=0$}\,.
    \end{cases}
\end{align}
Therefore, by Theorem \ref{theorem epsilon to zero compactness}, up to extracting subsequences, there exist $\mu$ a Radon measure on $\Om$ and $(K,E)\in  \mathcal{K}_{\rm B}$ with $|E|\leq v_0$ and $K\cup E^{\one}$ is $\C$-spanning, such that $\mu$ is the weak-star limit of $\{ |D(\Phi\circ u_j)|\}_j$, $1_E$ is $L^1_{\rm loc}(\Om)$-limit of $\{u_j\}_j$, and
\begin{align}\label{ineq 1}
  \mu
  \geq
  \begin{cases}
  2\H^n\mres(K \cap E^\zero) + \H^n\mres (\partial^* E \cap \Omega)\,, &\mbox{if $v_0>0$}\,,
  \\
  2\,\Phi(\delta_0)\,\H^n\mres K\,, &\mbox{if $v_0=0$}\,.
  \end{cases}
\end{align}
We claim that $|E|=v_0$. If $v_0=0$ this is trivial, while if $v_0>0$ we can combine \eqref{ineq 1} with Lemma \ref{lemma isoperimetry at infinity} to conclude that
\begin{align}\notag
    \Psi_{\rm bk}(v_0) \geq \F_{\rm bk}(K,E) +(n+1)\,\omega_{n+1}\,(v-|E|)^{n/(n+1)}\,.
\end{align}
In particular, if $|E|<v_0$, then we can construct a minimizing sequence for $\Psi_{\rm bk}(v_0)$ with positive volume loss at infinity, thus contradicting Theorem \ref{theorem from MNR1}-(iii). Having proved $|E|=v_0$, and thus that $(K,E)$ is a minimizer of $\Psi_{\rm bk}(v_0)$, we can combine \eqref{we did the construction}, \eqref{modica mortola identity} and \eqref{ineq 1} to find, in the case $v_0>0$,
\begin{eqnarray}
  \label{lookback1}
  \Psi_{\rm bk}(v_0)&\ge&\limsup_{j\to \infty} \frac{\acj(u_j;\Om)}2\ge\limsup_{j\to\infty}|D(\Phi\circ u_j)|(\Om)
  \\\nonumber
  &\ge&\mu(\Om)\ge\F_{\rm bk}(K,E)=\Psi_{\rm bk}(v_0)\,;
\end{eqnarray}
and, in the case $v_0=0$,
\begin{equation}
    \label{lookback2}
    2\,\Phi(\delta_0)\,\ell\ge\limsup_{j\to \infty}\frac{\acj(u_j;\Om)}2\ge\limsup_{j\to\infty}|D(\Phi\circ u_j)|(\Om)\ge\mu(\Om)\ge 2\,\Phi(\delta_0)\,\ell\,.
\end{equation}
From \eqref{lookback1} and \eqref{lookback2} we find
\[
\lim_{j\to \infty} \Upsilon(v_j,\e_j,\delta_j)
    =
    \begin{cases}
    \Psi_{\rm bk}(v_0)\,, & \mbox{if  $v_0>0$}\,,
    \\
    2\,\Phi(\delta_0)\,\ell\,, & \mbox{if $v_0=0$}\,,
    \end{cases}
\]
thus proving conclusions (ii) and (iii), as well as, looking back at \eqref{modica mortola identity},
\[
\lim_{j\to\infty}\int_\Om\Big(\sqrt{\e_j}\,|\nabla u_j|-\sqrt{W(u_j)/\e_j}\Big)^2=0\,,
\]
from which conclusion (iv) follows immediately.

\medskip

\noindent {\it Proof of conclusion (i)}: Since the validity of the Euler--Lagrange equation in inner variation form is immediate from Lemma \ref{lemma volume fixing}-(i), it is really a matter of proving that for every positive $\tau_0$ there is a positive $\tau_1$ (depending on the data of the problem) such that $\Upsilon(v,\e,\de)$ admits a minimizer for every $(v,\e,\de)\in{\rm SFR}(\tau_0,\tau_1)$. We shall actually prove that for every such $(v,\e,\de)$, every minimizing sequence of $\Upsilon(v,\e,\de)$ converge (modulo extracting subsequences) to a minimizer. To do this, after a rescaling that sets $\diam\wire=1$, we argue by contradiction. This means assuming that there are a sequence
$\{(v_j,\e_j,\de_j)\}_j$ satisfying \eqref{e/v vanishes} and \eqref{delta to 1}, and, for each $j$, a minimizing sequence $\{u_j^k\}_k$ for $\Upsilon(v_j,\e_j,\delta_j)$, such that, up to extracting a diagonal subsequence, for each $j$ there is $u_j^0\in W^{1,2}_{\rm loc}(\Om;[0,1])$ such that $u_j^k\to u_j^0$ in $L^1_{\rm loc}(\Om)$ as $k\to\infty$, {\it but, for no index $j$, $u_j^0$ is a minimizer of $\Upsilon(v_j,\e_j,\de_j)$}.

\medskip

Now, by  Theorem \ref{theorem fixed epsilon compactness} and Theorem \ref{theorem epsilon to zero compactness}, we know that
\begin{equation}
  \label{wj is spanning}
  \mbox{$\{(u_j^0)^*\ge t\}$ is $\C$-spanning $\wire$ for every $t\in(1/2,\de_j)$}\,,
\end{equation}
for every $j$, while by lower semicontinuity of the Allen--Cahn energy we have
\begin{equation}
\label{we did the construction2}
\Upsilon(v_j,\e_j,\delta_j)\ge\frac{\acj(u_j^0;\Om)}2\,,\qquad\forall j\,.
\end{equation}
Therefore the only possibility for $u_j^0$ not to be a minimizer of $\Upsilon(v_j,\e_j,\delta_j)$ is that
\begin{eqnarray}
\label{converges to the wrong thing}
\vv_j^\infty:=v_j-\V(u_j^0;\Om)>0\,,\qquad\forall j\,.
\end{eqnarray}
We shall conclude the proof of the theorem by exploiting \eqref{converges to the wrong thing} to identify a subsequence $\{j(i)\}_{i\in \mathbb{N}}$ such that, for every $i$ large enough,
\begin{align}\label{final contradiction}
\lim_{i\to\infty}\frac{\mathcal{AC}_{\e_{j(i)}}(u^k_{j(i)})}2>\Upsilon(v_{j(i)},\e_{j(i)},\delta_{j(i)})\,,
\end{align}
thus obtaining the desired contradiction.

\medskip

The idea behind the proof of \eqref{final contradiction} under \eqref{converges to the wrong thing} is to ``bring back from infinity'' the lost volume $\vv_j^\infty$ in the form of a ``half-bubble'' that touches the wire frame, and then to exploit the fact that the isoperimetric profile of half-spaces is strictly less than the isoperimetric profile of $\mathbb{R}^{n+1}$. There are several issues that must be addressed to make this approach work. First is the fact that we do not yet know that the sharp interface limit is a good approximation for the behavior of the escaping volume at infinity, since the sharp interface problem is, roughly speaking, closely describing the escaping volume only if $\e_j / (\vv_j^\infty)^{1/(n+1)}\to 0$. Second, unlike sharp interface problems, in which escaping volumes can be ``brought back" and placed somewhere that does not disturb the rest of the configuration, Allen-Cahn minimizers will always have interacting tails. Without the regularity of minimizers (or even of limits of minimizing sequences), controlling these tail interactions will require care. Last but not least is the fact that the energy corresponding to a volume escaping at infinity corresponds to a lower order energy term whose presence does not contradict the leading order convergence of $\Upsilon(v_j,\e_j,\delta_j)$ to, respectively, $\Psi_{\rm bk}(v_0)$ or to $2\,\Phi(\delta_0)\,\ell$. Any sort of analysis leading to the existence of minimizers must therefore be fine enough to detect vanishingly small inefficiencies in minimizing sequences with escaping volumes.

\medskip

We divide the argument into steps. In step one, we sharply improve \eqref{we did the construction2} to \eqref{ball running} and establish minimality and criticality properties of $u_j^0$. In step two, we conclude the contradiction in the case when $v_0>0$ by using Theorem \ref{theorem from MNR1}-(iii). Step three deals with the case $v_0=0$, and it is in this case that the difficulties described in the previous paragraph are carefully addressed.
\medskip

\noindent {\it Step one}: Denoting by $\zeta_j$ a radial minimizer with maximum at the origin for the diffused Euclidean isoperimetric problem $\Theta(\vv_j^\infty,\e_j)$ and by $\Lambda_j$ the Lagrange multiplier of $\zeta_j$, so that $2\,\e_j^2\,\Delta\zeta_j=W'(\zeta_j)-\e_j\,\Lambda_j\,V'(\zeta_j)$ on $\R^{n+1}$ (see Appendix \ref{section diffuse isop}), we claim that,
\begin{eqnarray}
  \label{ball running}
    &&\Upsilon(v_j,\e_j,\delta_j) = \big(\ac(u_j^0;\Om)/2\big)+ \Theta(\vv_j^\infty,\e_j)\,,
    \\
  \label{wj minimizer of Upsilon vvjinf}
    &&\mbox{$u_j^0$ is a minimizer of $\Upsilon(v_j-\vv_j^\infty,\e_j,\de_j)$}\,,
\end{eqnarray}
for every $j$, where $u_j^0$ satisfies \eqref{EL inner} with $\l=\Lambda_j$, that is
\begin{equation}
  \label{EL inner wj}
      \int_\Omega\Big(\e_j\,|\nabla u_j^0|^2 + \frac{W(u_j^0)}{\e_j} \Big) \Div X - 2\,\e_j\,
    \nabla u_j^0\,\cdot\, \nabla X[\nabla u_j^0]=\Lambda_j\, \int_\Omega V(u_j^0)\,\Div X\,,
\end{equation}
whenever $X\in C_c^\infty(\mathbb{R}^{n+1};\mathbb{R}^{n+1})$ with $X \cdot \nu_\Omega= 0$ on $\partial \Omega$.

\medskip

Indeed, we can prove the $\le$-part of \eqref{ball running} by constructing a competitor for $\Upsilon(v_j,\e_j,\delta_j)$ obtained as a slight modification via Lemma \ref{lemma volume fixing}-(ii) of the {\it Ansatz} $x\mapsto u_j^0(x)+\zeta_j(x-k\,e)$ ($e\in\SS^n$, $k$ large). The matching lower bound is obtained by applying Lemma \ref{lemma isoperimetry at infinity} to the minimizing sequence $\{u_j^k\}_k$ of $\Upsilon(v_j,\e_j,\delta_j)$ (notice that the lemma is applied here ``at fixed $\e$''). Having proved \eqref{ball running}, we notice that, again by the construction of Theorem \ref{theorem approximation of gen soap vuoto},
\begin{equation}
  \label{abc}
  \Upsilon(v_j,\e_j,\de_j)\le\Upsilon(v_j-\vv_j^\infty,\e_j,\de_j)+\Theta(\vv_j^\infty,\e_j)\,.
\end{equation}
This inequality, combined with \eqref{ball running}, implies \eqref{wj minimizer of Upsilon vvjinf}. By \eqref{wj minimizer of Upsilon vvjinf}, \eqref{EL inner wj} holds indeed with some Lagrange multiplier $\l_j$: the fact that $\l_j=\Lambda_j$ thus follows by a standard first variation argument (it they were different, we could violate \eqref{ball running}).

\medskip

\noindent {\it Step two}: We conclude the proof in the case when $v_0>0$. Indeed, up to extracting a further subsequence there is $\vv^\infty_0\in[0,v_0]$ such that $\vv_j^\infty\to \vv^\infty_0$ as $j\to\infty$. Since $\V(u_j^0;\Om)=v_j-\vv_j^\infty\to v_0-\vv_0^\infty$ and (thanks to \eqref{we did the construction}) $\sup_j\acj(u_j^0;\Om)<\infty$, up to extracting subsequences we can apply Theorem \ref{theorem epsilon to zero compactness} to conclude the existence of  $(K,E)\in \mathcal{K}_{\rm B}$ such that $u_j^0\to 1_E$ in $L^1_{\rm loc}(\Om)$, $K \cup E^{\one}$ is $\C$-spanning $\wire$, $|E|\leq v_0-\vv_0^\infty$, and $\liminf_j\acj(u_j^0;\Om)/2\ge\mathcal{F}_{\rm B}(K,E)$ (when exploiting \eqref{weak star in diffuse compactness theorem}, recall that in the present argument $v_0>0$ implies $\de_j\to 1^-$). We can actually improve on this lower bound by using Lemma \ref{lemma isoperimetry at infinity}, thus concluding that
\begin{align}\label{this ineq}
    \liminf_{j\to \infty} \frac{\acj(u_j^0;\Om)}2 \geq \mathcal{F}_{\rm B}(K,E) + (n+1)\,\omega_{n+1}\,\big(v_0-\vv_0^\infty-|E|\big)^{n/(n+1)}\,.
\end{align}
By \eqref{we did the construction}, \eqref{ball running}, \eqref{this ineq}, the continuity of $\Theta$ (see Theorem \ref{theorem diffused isop}-(ii)), and the concavity of the Euclidean isoperimetric profile, we find that
\begin{align}\notag
    \Psi_{\rm B}(v_0) &\geq \limsup_{j\to \infty}\Upsilon(v_j,\e_j,\delta_j)
    = \liminf_{j\to \infty} \big(\acj(u_j^0;\Om)/2\big)+\Theta(\vv_j^\infty,\e_j)
    \\ \notag
    &\geq \mathcal{F}_{\rm B}(K,E) + (n+1)\,\omega_{n+1}^{1/(n+1)}\,\Big\{(\vv_0^\infty)^{n/(n+1)}+ \big(v_0-\vv_0^\infty-|E|\big)^{n/(n+1)}\big\}\\
     \label{expression}
    &\geq \mathcal{F}_{\rm B}(K,E) + (n+1)\,\omega_{n+1}^{1/(n+1)}\,2\,(v_0 - |E|)^{n/(n+1)}\,.
\end{align}
Now, thanks to Theorem \ref{theorem from MNR1}-(iii), no minimizing sequence in $\Psi_{\rm B}(v_0)$ can lose volume at infinity. Therefore \eqref{expression} implies that $|E|=v_0$ and $\vv_0^\infty=0$.

\medskip

We can use the latter information to obtain a contradiction by arguing as follows. Since $|E|=v_0>0$ implies that $1_E$ is not constant in $\Om$ and $u_j^0 \to 1_E$ in $L^1_{\rm loc}(\Om)$, by Lemma \ref{lemma volume fixing}-(ii) and $\V(u_j^0;\Om)=v_j\to v_0$ we can find diffeomorphisms $f_j:\Om\to\Om$ such that $\V(u_j^0\circ f_j;\Om)=v_j$ and
\[
|\acj(u_j^0\circ f_j;\Om)-\acj(u_j^0;\Om)|\le C_0\,\acj(u_j^0;\Om)\,|\V(u_j^0;\Om)-v_j|\le C\,\vv_j^\infty\,,
\]
where $C_0$ depends on $\Om$ and $E$ as in Lemma \ref{lemma volume fixing}-(ii), and $C=C_0\,\sup_j\,\acj(u_j^0;\Om)$. Since the homotopic spanning constraint is preserved under composition with a diffeomorphism of $\Om$, we find that $u_j^0\circ f_j$ is a competitor in $\Upsilon(v_j,\e_j,\delta_j)$, so that
\begin{align}\label{bad 1}
    2\,\Upsilon(v_j,\e_j,\delta_j) \leq \acj(u_j^0;\Om)+ C\,\vv_j^\infty\,;
\end{align}
at the same time, by \eqref{ball running}, we have
\begin{align}\label{bad 2}
2\,\Upsilon(v_j,\e_j,\delta_j) = \acj(u_j^0;\Om) + 2\,\Theta(\vv_j^\infty,\e_j) \ge \acj(u_j^0;\Om) + 2\, (n+1)\,\om_{n+1}^{1/(n+1)}\,(\vv_j^\infty)^{n/(n+1)}\,
\end{align}
where the combination of \eqref{bad 1} and \eqref{bad 2} leads to a contradiction since $\vv_j^\infty\to 0$.

\medskip

\noindent {\it Step three}: We are now left to consider the case when $v_0=0$. Thanks to
\eqref{wj minimizer of Upsilon vvjinf}, \eqref{ball running}, $v_0=0$, and \eqref{e/v vanishes}, $\{u_j^0\}_j$ is a sequence of minimizers for $\Upsilon(\vv_j,\e_j,\de_j)$ for some
\[
\vv_j:=v_j-\vv_j^\infty\in(0,v_j)
\]
such that $\acj(u_j^0;\Om)\le\Upsilon(v_j,\e_j,\de_j)$, $v_j\to 0^+$, and $\e_j/v_j\to 0^+$. Therefore, by Theorem \ref{theorem lambdaepsj to zero} and taking \eqref{EL inner wj} into account, we have that
\begin{equation}
  \label{bdd of l mult}
  \lim_{j\to\infty}\e_j\,\Lambda_j=0\,,
\end{equation}
as well as that (compare with \eqref{mu limit})
\begin{equation}
\label{mu limit 2}
2\,\Phi(\de_0)\,\H^n\mres K
={\rm w^*}\!\!\lim_{j\to\infty}|\nabla(\Phi\circ u_j^0)|\,\L^{n+1}\mres\Om
={\rm w^*}\!\!\lim_{j\to\infty}\frac{{\rm ac}_{\e_j}(u_j^0)}2\,\L^{n+1}\mres\Om\,,
\end{equation}
where $K$ is a minimizer of $\ell$.

\medskip

Recalling that $\zeta_j$ is a minimizer in $\Theta(\vv_j^\infty,\e_j)$ and that $\Lambda_j=\Lambda(\zeta_j)$, by \eqref{key inq for Lambda} (see Theorem \ref{theorem diffused isop}-(iii)) we have that
\begin{equation}
\label{key inq for Lambda proof}
  \Lambda_j\ge \frac{c(n,W)}{(\vv_j^\infty)^{1/(n+1)}}\,.
\end{equation}
This inequality, combined with \eqref{bdd of l mult}, gives in particular
\begin{align}\label{ball at infinity behaves}
\lim_{j\to\infty}\frac{\e_j}{(\vv_j^\infty)^{1/(n+1)}}=0\,.
\end{align}
Hence, for $j$ large enough, we have $\e_j<\s_0\,(2\,\vv_j^\infty)^{1/(n+1)}$ for $\s_0=\s_0(n,W)>0$ as in Theorem \ref{theorem diffused isop}-(iv). Setting $\s_j=\e_j/(2\,\vv_j^\infty)^{1/(n+1)}$, and recalling that $\zeta_{v,\e}$ denotes the unique modulo translations radially symmetric decreasing minimizer of $\Theta(v,\e)$ when $\e<\s_0\,v^{1/(n+1)}$, let us now consider that
\[
\zeta_{2\vv_j^\infty,\e_j}(x)=\zeta_{1,\s_j}\Big(x\big/(2\,\vv_j^\infty)^{1/(n+1)}\Big)\,,\qquad\forall x\in\R^{n+1}\,,
\]
where we have set $\pi_j(x)=x/(2\,\vv_j^\infty)^{1/(n+1)}$ ($x\in\R^{n+1}$); denoting by $B_{r(n)}$ the unit volume ball with center at the origin, the fact that $\s_j\to 0^+$ guarantees that
\[
\lim_{j\to\infty}\int_{\R^{n+1}}|\zeta_{1,\s_j}-1_{B_{r(n)}}|=0\,,\qquad
\g_j:=\max\Big\{\V(\zeta_{1,\s_j};B_{2\,r(n)}^c),\mathcal{AC}_{\s_j}(\zeta_{1,\s_j};B_{2\,r(n)}^c)\Big\}\to 0\,.
\]
Hence, by Lemma \ref{lemma volume fixing}, there exist $\eta_0>0$ such that for every $j$ large enough and every $|\eta|<\eta_0$ there is a radially symmetric diffeomorphism $f_j^\eta:B_{2\,r(n)}\to B_{2\,r(n)}$ with $\{f_j^\eta\ne\id\}\cc B_{2\,r(n)}$ and
\begin{eqnarray*}
&&\V(\zeta_{1,\s_j}\circ f_j^\eta)=\V(\zeta_{1,\s_j})+\eta=1+\eta\,,
\\
&&\big|\mathcal{AC}_{\s_j}(\zeta_{1,\s_j}\circ f_j^\eta)-\mathcal{AC}_{\s_j}(\zeta_{1,\s_j})\big|\le C(n)\,  \Theta(1,\s_j)\,|\eta|\,.
\end{eqnarray*}
By radial symmetry of $\zeta_{1,\s_j}\circ f_j^\eta$ and $\zeta_{1,\s_j}$, if we set $H=\{x_{n+1}>0\}$, then, for every $|\eta|<\eta_0$ and $j$ large enough,
\begin{eqnarray*}
&&\V(\zeta_{1,\s_j}\circ f_j^\eta;H)=\frac12+\eta\,,
\\
&&\mathcal{AC}_{\s_j}(\zeta_{1,\s_j}\circ f_j^\eta;H)\le\big(1+ C(n)\,|\eta|\big)\,  \Theta(1,\s_j)\,.
\end{eqnarray*}
Then the functions $\zeta_j^\eta=\zeta_{1,\s_j}\circ f_j^\eta\circ\pi_j$ satisfy
\begin{eqnarray*}
&&\V(\zeta_j^\eta;H)=2\,\vv_j^\infty\,\V(\zeta_{1,\s_j}\circ f_j^\eta;H)=\vv_j^\infty+\eta\,\vv_j^\infty\,,
\\
&&\mathcal{AC}_{\e_j}(\zeta_j^\eta;H)=(2\,\vv_j^\infty)^{n/(n+1)}\,\mathcal{AC}_{\s_j}(\zeta_j^\eta;H)\le \big(1+C(n)\,|\eta|\big)\,  \Theta(2\,\vv_j^\infty,\e_j)\,.
\end{eqnarray*}
At the same time, by \eqref{ball at infinity behaves} we can use \eqref{energy expansion for theta} to find
\begin{align}\label{conv to half profile 2}
\lim_{j\to\infty}\frac{\Theta(2\,\vv_j^\infty,\e_j)}{\Theta(\vv_j^\infty,\e_j)}=2^{n/(n+1)}\,,
\end{align}
so that there are $\beta(n)\in(0,1)$ and $J_0\in\N$ such that
\[
\frac{\Theta(2\,\vv_j^\infty,\e_j)}2 \leq (1-\beta(n))\,\Theta( \vv_j^\infty,\e_j)\,,\qquad\forall j\ge J_0\,,
\]
and, in summary,
\begin{equation}
  \label{beta inq}
  \frac{\mathcal{AC}_{\e_j}(\zeta_j^\eta;H)}2\le \big(1-\beta(n)+C(n)|\eta|\big)\,\Theta( \vv_j^\infty,\e_j)\,,\qquad\forall j\ge j_0\,,|\eta|<\eta_0\,.
\end{equation}
We next notice that, by the smoothness of $\pa\wire$ and up to a rigid motion that takes $0\in\pa\wire$ and $\nu_{\wire}(0)=e_{n+1}$, we can find positive constant $C$ and $r'$ depending on $\wire$ such that
\begin{equation}\label{cone containment}
\H^n((H \Delta \Om)\cap\pa B_r)\le C\,r^{n+1}\,,\qquad\forall r<r'\,,
\end{equation}
where $x=(x',x_{n+1})\in\R^{n+1}\equiv\R^n\times\R$. Therefore, if we set for brevity
\[
r_j:=(\vv_j^\infty)^{1/(n+1)}\,2\,r(n)
\]
them,, up to increase the value of $J_0$ so to have $r_j<r'$ when $j\ge J_0$, and noticing that $\mathrm{ac}_{\e_j}(\zeta_j^\eta)$ is a radial function, by $\zeta_{1,\s_j}=\zeta_{1,\s_j}\circ f_j^\eta$ on $B_{2\,r(n)}^c$ and by definition of $\g_j$,
\begin{eqnarray*}
&&\Big|\acj(\zeta_j^\eta;\Om)-\acj(\zeta_j^\eta;H)\Big|\le \int_{(\Om\Delta H)\cap B_{r_j}}\mathrm{ac}_{\e_j}(\zeta_j^\eta)+2\,\acj(\zeta_j^\eta;B_{r_j}^c)
\\
&&=\int_0^{r_j}\,\H^n((\Om\Delta H)\cap\pa B_r)\,\mathrm{ac}_{\e_j}(\zeta_j^\eta)\,dr+2\,(\vv_j^\infty)^{n/(n+1)}\,\mathcal{AC}_{\s_j}(\zeta_{1,\s_j}\circ f_j^\eta;B_{2\,r(n)}^c)
\\
&&\le\sup_{0<r<r_j}\Big\{\frac{\H^n((\Om\Delta H)\cap\pa B_r)}{\H^n(H\cap\pa B_r)}\Big\}\,\acj(\zeta_j^\eta;H)
+2\,(\vv_j^\infty)^{n/(n+1)}\,\g_j
\\
&&\le C(n,\wire)\,r_j\,\Theta(\vv_j^\infty,\e_j)+2\,(\vv_j^\infty)^{n/(n+1)}\,\g_j={\rm o}_{n,W}\big(\vv_j^\infty\big)^{n/(n+1)}\,.
\end{eqnarray*}
where in the last two inequalities we have used, in the order, by \eqref{cone containment}, \eqref{beta inq}, \eqref{energy expansion for theta}, and $\g_j\to 0$. By combining this inequality with \eqref{beta inq} we thus conclude that
\begin{align}\label{good cutoff property on W}
  \frac{\acj(\zeta_j^\eta;\Om)}{2}\le
  \Big(1-\beta(n)+C(n)\,|\eta|\Big)\,\Theta(\vv_j^\infty,\e_j)+{\rm o}_{n,W}\big(\vv_j^\infty\big)^{n/(n+1)}\,,
\end{align}
(where we recall that $\Theta(\vv_j^\infty,\e_j)/(\vv_j^\infty)^{n/(n+1)}\to c(n)>0$ as $j\to\infty$ by \eqref{energy expansion for theta}.) By an identical argument we also find that
\[
\V(\zeta_j^\eta;\Om)=\big(1+\eta\big)\,\vv_j^\infty+{\rm o}_{n,W}\big(\vv_j^\infty\big)\,.
\]
In summary, we can claim the existence of $J_0\in\N$ and $\eta_0>0$ such that for each $x\in\pa\wire$ we can find $\zeta_j^{x,\eta}$ with the properties that
\begin{eqnarray}\label{zeta jxeta small vol and ac}
&&\V(\zeta_j^{x,\eta};B_{r_j}(x)^c)={\rm o}_{n,W}\big(\vv_j^\infty\big)\,,\qquad
\ac(\zeta_j^{x,\eta};B_{r_j}(x)^c)={\rm o}_{n,W}\big(\vv_j^\infty\big)^{n/(n+1)}\,,\hspace{1cm}
\\\label{zeta jxeta volume}
&&\V(\zeta_j^{x,\eta};\Om)=\big(1+\eta\big)\,\vv_j^\infty+{\rm o}_{n,W}\big(\vv_j^\infty\big)
\\\label{zeta jxeta AC}
&&\frac{\acj(\zeta_j^{x,\eta};\Om)}{2}
\le\Big(1-\beta(n)+C(n)\,|\eta|\Big)\,\Theta(\vv_j^\infty,\e_j)+{\rm o}_{n,W}\big(\vv_j^\infty\big)^{n/(n+1)}\,,
\end{eqnarray}
where $r_j=(\vv_j^\infty)^{1/(n+1)}\,2\,r(n)$.

\medskip

The final step in the construction is making a choice of $x=x_j\in\pa\wire$ such that the interaction between $\zeta_j^{x,\eta}$ and $u_0^j$ in minimized. We claim that indeed $x_j\in\pa\wire$ can be found such that
\begin{equation}
  \label{good xjs}
  \V(u_0^j;\Om\cap B_{r_j}(x_j))={\rm o}(\vv_j^\infty)\,.
\end{equation}
We now show, first, how to derive a contradiction from \eqref{zeta jxeta small vol and ac}, \eqref{zeta jxeta volume}, \eqref{zeta jxeta AC}, and \eqref{good xjs}; and, finally, how to prove \eqref{good xjs}. In this way the proof of the theorem will be complete.

\medskip

\noindent {\it Derivation of a contradiction from \eqref{zeta jxeta small vol and ac}, \eqref{zeta jxeta volume}, \eqref{zeta jxeta AC}, and \eqref{good xjs}}: Let us consider the functions
\[
h_j^\eta=\max\big\{u_0^j,\zeta_j^{x_j,\eta}\big\}\,,\qquad j\ge J_0\,,|\eta|<\eta_0\,.
\]
Since $h_j^\eta\ge u_0^j$ on $\Om$, we have that $\{(h_j^\eta)^*\ge t\}$ is $\C$-spanning $\wire$ for every $t\in[1/2,\de_j)$. We claim that we can find $|\eta_j|<\eta_0$ such that
\begin{equation}
  \label{def of etaj}
  \V(h_j^{\eta_j};\Om)=v_j\,,\qquad \lim_{j\to\infty}\eta_j=0\,.
\end{equation}
We start noticing that by \eqref{zeta jxeta small vol and ac}, \eqref{good xjs}, $V(u_0^j)\le V(\zeta_j^{x_j,\eta})$ on $\{u_0^j\le\zeta_j^{x_j,\eta}\}$, and $V(u_0^j)\ge V(\zeta_j^{x_j,\eta})$ on $\{u_0^j\ge\zeta_j^{x_j,\eta}\}$, we find that
\begin{eqnarray*}
\V(h_j^\eta;\Om)=\V(u_j^0;\Om)+\V(\zeta_j^{x_j,\eta};\Om)+{\rm o}(\vv_j^\infty)\,.
\end{eqnarray*}
Therefore, by \eqref{zeta jxeta volume}, and recalling that, by definition $\V(u_j^0;\Om)=v_j-\vv_j^\infty$, we find that
\begin{equation}
  \label{backinto}
\V(h_j^\eta;\Om)=v_j+\eta\,\vv_j^\infty+{\rm o}(\vv_j^\infty)\,.
\end{equation}
In particular, up to increase the value of $J_0$, if $j\ge J_0$ we have $\V(h_j^{\eta_0/2};\Om)>v_j$ and $\V(h_j^{-\eta_0/2};\Om)<v_j$. By continuity of $\big(|\eta|\le\eta_0/2\big)\mapsto\V(h_j^\eta;\Om)$ we find $|\eta_j|\le\eta_0/2$ such that $\V(h_j^{\eta_j};\Om)=v_j$. Plugging this information back into \eqref{backinto} we find that $\eta_j\to 0$ as $j\to\infty$.

\medskip

To derive a contradiction we notice that, being $h_j^{\eta_j}$ admissible in $\Upsilon(v_j,\e_j,\de_j)$, by \eqref{ball running} and \eqref{zeta jxeta AC}
\begin{eqnarray*}
&&\ac(u_j^0;\Om)+2\,\Theta(\vv_j^\infty,\e_j)=2\,\Upsilon(v_j,\e_j,\de_j)\le \acj\big(h_j^{\eta_j};\Om\big)
\\
&&=\acj\big(u_j^0;\Om\cap\{u_j^0\ge\zeta_j^{x_j,\eta_j}\}\big)
+\acj\big(\zeta_j^{x_j,\eta_j};\Om\cap\{\zeta_j^{x_j,\eta_j}\ge u_0^j\}\big)
\\
&&\le\acj(u_j^0;\Om)+\acj(\zeta_j^{x_j,\eta_j};\Om)
\\
&&\le \acj(u_j^0;\Om)+
2\,\Big(1-\beta(n)+C(n)\,|\eta_j|\Big)\,\Theta(\vv_j^\infty,\e_j)+{\rm o}_{n,W}\big(\vv_j^\infty\big)^{n/(n+1)}\,,
\end{eqnarray*}
which leads to a contradiction with $\beta(n)\in(0,1)$ as soon as $j$ is large enough.

\medskip

\noindent {\it Proof of \eqref{good xjs}:} We finally prove the existence of $x_j\in\pa\wire$ such that \eqref{good xjs} holds. To this end, we recall the validity of \eqref{mu limit 2}, where $K$ is a minimizer of $\ell$. By exploiting this minimality property as done in \cite[Proof of Theorem 1.4, Step 6]{KingMaggiStuvard}, we see that $K$ does not concentrate area near $\pa\wire$, that is
\begin{align}
\label{nonconcentration}
\H^n\big(K \cap \{x:\dist(x,\partial \wire)<r\}\big)\le C\,r\,,\qquad\forall r>0\,,
\end{align}
with some $C$ depending on $K$. Moreover, as shown for example in \cite[Appendix B]{KingMaggiStuvard},
\begin{align}\label{bdry dens}
\H^n(K \cap B_r(x)) \geq c_0 r^n \quad\forall x\in \cl(K)\,,\,\, r\in (0,r_0)\,.
\end{align}
By combining \eqref{nonconcentration} and \eqref{bdry dens} we see that $K$ is not ``wetting'' the whole $\pa\wire$, that is
\begin{equation}
  \label{not full wire}
  (\pa\wire)\setminus\cl(K)\ne\varnothing\,.
\end{equation}
In particular, there are $x_0\in \partial \wire$ and $r_0>0$ such that $\cl B_{r_0}(x_0) \cap K=\varnothing$, so that, by \eqref{mu limit 2}, and taking also into account that $u_j^0\to 0$ in $L^1_{\rm loc}(\Om)$, we find
\begin{align}\label{ball sees no energy}
   \lim_{j\to\infty}\acj(u_j^0;B_{r_0}(x_0)\cap\Om)=0\,.
\end{align}
Correspondingly to $x_0$, and up to decreasing the value of $r_0$, we can find a cube $Q\subset\R^n$ and an embedding $g$ of $Q\times[0,r_0)$ into $\ov{\Om}\cap B_{r_0}(x_0)$ so that $g$ embeds $Q\times\{0\}$ into $(\pa\Om)\cap B_{r_0}(x_0)$ and $g$ is arbitrarily $C^1$-close to an isometry, in such a way that if $Q'$ is a cube contained in $Q$ with side length $2\,s$ and center $z$, then
\begin{equation}
  \label{almost an isometry}
  \Om\cap B_s(g(z))\subset g(Q'\times(0,s))\,.
\end{equation}
If $j$ is large enough, then we can find a partition $\F_j=\{Q_j^k\}_{k=1}^{N(j)}$ of $Q$ into subcubes of sidelength $s_j>r_j$ for some $s_j$ and $N_j$ satisfying
\begin{equation}
  \label{whatabout sj and Nj}
  s_j={\rm O}(\vv_j^\infty)^{1/(n+1)}\,,\qquad N_j={\rm O}(\vv_j^\infty)^{-n/(n+1)}\,.
\end{equation}
The subfamily $\G_j$ defined by
\[
\G_j=\Big\{Q_j^k:\acj\big(u_0^j;g(Q_j^k\times(0,s_j))\big)\le\frac{\acj(u_j^0;B_{r_0}(x_0)\cap\Om)^{1/2}}{N(j)}\Big\}
\]
is of course such that
\begin{equation}
  \label{Gj close to Nj}
  0\le 1-\frac{\#\,\G_j}{N(j)}\le\acj(u_j^0;B_{r_0}(x_0)\cap\Om)^{1/2}\,,
\end{equation}
and by \eqref{ball sees no energy} and \eqref{whatabout sj and Nj} we have
\begin{equation}
  \label{they are all good AC balls}
  \acj\big(u_0^j;g(Q_j^k\times(0,s_j))\big)={\rm o}\big(\vv_j^\infty\big)^{n/(n+1)}\,,\qquad\forall Q_j^k\in\G_j\,.
\end{equation}
We now claim that, for every $j$ large enough, there is $Q_j^{k(j)}\in\G_j$ such that
\begin{equation}
  \label{if not}
  \V(u_0^j;g(Q_j^{k(j)}\times(0,s_j))\big)={\rm o}\big(\vv_j^\infty\big)
\end{equation}
Denoting by $z_j$ the center of $Q_j^{k(j)}$ and setting $x_j=g(z_j)$, and by applying \eqref{almost an isometry} with $s=s_j>r_j$, we conclude that
\[
\V(u_0^j;\Om\cap B_{r_j}(x_j))\le \V\big(u_0^j;g(Q_j^{k(j)}\times(0,s_j))\big)={\rm o}\big(\vv_j^\infty\big)\,,
\]
and complete the proof of \eqref{good xjs}. To prove \eqref{if not}, we argue by contradiction. Should \eqref{if not} fail, then, up to extracting a subsequence in $j$ and for some constant $c_0>0$, we would have that
\begin{equation}
  \label{if not contra}
  \V(u_0^j;g(Q_j^k\times(0,s_j))\big)\ge c_0\,s_j^{n+1}\,,\qquad\forall Q_j^k\in\G_j\,,
\end{equation}
so that, for some $c_1\in(0,1)$,
\[
c_1\le\fint_{Q_j^k\times(0,s_j)} V(u_0^j\circ g)\,,\qquad\forall Q_j^k\in\G_j\,.
\]
Since $V(u_0^j\circ g)$ takes values in $[0,1]$ we find that
\begin{eqnarray*}
\frac{c_1}2\,|Q_j^k\times(0,s_j)|&\le&\Big|\Big\{(y,s)\in Q_j^k\times(0,s_j):V(u_0^j\circ g)\ge \frac{c_1}2\Big\}\Big|
\\
&\le&\H^n\Big(\Big\{y\in Q_j^k:\sup_{(0,s_j)}\,V((u_0^j)^*(g(y,\cdot)))\ge \frac{c_1}2\Big\}\Big)\,s_j\,,
\end{eqnarray*}
that is, for some $c\in(0,1)$ and recalling that $V$ is strictly increasing on $(0,1)$,
\[
  c\,\H^n(Q_j^k)\le \H^n\Big(\Big\{y\in Q_j^k:\sup_{(0,s_j)}\,(u_0^j)^*(g(y,\cdot))\ge c\Big\}\Big)\,,\qquad\forall Q_j^k\in\G_j\,.
\]
Adding up over all the cubes in $\G_j$, and recalling \eqref{Gj close to Nj} and that $\F_j$ is a partition of  $Q$, we thus conclude (up to further decrease the value of $c$) that
\begin{equation}
  \label{if not contra 2}
 \H^n\Big(\Big\{y\in Q:\sup_{(0,s_j)}\,(u_0^j)^*(g(y,\cdot))\ge c\Big\}\Big)\ge c\,\H^n(Q)\,,\qquad\forall j\,.
\end{equation}
However, by \eqref{ball sees no energy}, \eqref{modica mortola identity}, Fubini's theorem, the area formula, and the slicing theory for Sobolev functions (see, e.g. \cite[Section 4.9.2]{EvansGariepyBOOK}),
there is a set $Z\subset Q$ with full $\H^n$-measure in $Q$ such that, for every $y\in Z$, $(u_0^j)^*(g(y,\cdot))$ is absolutely continuous on $(0,r_0)$, $(u_0^j)^*(g(y,\cdot))\to 0$ $\H^1$-a.e. on $(0,r_0)$ as $j\to\infty$ (recall indeed that $\V(u_0^j;\Om)\to 0$), and
\[
2\,\int_{g(\{y\}\times(0,r_0))}|\nabla [\Phi\circ(u_0^j)^*]|\,d\H^1\leq \int_{g(\{y\}\times(0,r_0))}\e_j|\nabla (u_0^j)^*|^2 + \frac{W((u_0^j)^*)}{\e_j}\,d\mathcal{H}^1\to 0\,.
\]
In particular, we find that
\[
\lim_{j\to\infty}\sup_{(0,r_0)}(u_0^j)^*(g(y,\cdot))=0\,,\qquad \forall y\in Z\,,
\]
and obtain a contradiction with \eqref{if not contra 2}.
\end{proof}

\section{Euler-Lagrange equation and regularity (Proof of Theorem \ref{theorem main regularity})}\label{section EL proof} We finally prove Theorem \ref{theorem main regularity} and Proposition \ref{proposition conditional regularity}.

\begin{proof}[Proof of Theorem \ref{theorem main regularity}]
		By Theorem \ref{theorem upsilon existence convergence}.(i), there is $\l\in\R$ such that
		\begin{equation}
			\label{EL inner proof of reg}
			\int_\Omega\Big(\e|\nabla u|^2 + \frac{W(u)}{\e} \Big) \Div X - 2\,\e
			\nabla u\,\cdot\, \nabla X[\nabla u]=\lambda\, \int_\Omega V(u)\,\Div X\,,
		\end{equation}
		whenever $X\in C_c^\infty(\Om;\mathbb{R}^{n+1})$; also, \eqref{EL inner proof of reg} extends to $X\in C_c^1(\Omega;\mathbb{R}^{n+1})$ by density.
		
		\medskip
		
		\noindent {\it Step one}: We claim that if $\varphi \in C_c^1(\Om)$ and $h \in\Lip_c([0,1))$ with $\vphi\ge0$ and $h\ge 0$, then
		\begin{equation}\label{eq levelset Euler-Lagrange}
			0\le \int_{\Omega}  \varphi\, h'(u)\,|\nabla u|^2+ h(u)\,\nabla u\cdot\nabla \varphi + \vphi\,h(u)\,F_\e'(u)\,,
		\end{equation}
		where
		\[
		F_\e(t)=\frac1{2\,\e}\,\Big\{\frac{W(t)}\e-\l\,V(t)\Big\}\,,\qquad t\in[0,1]\,.
		\]
		To prove this, we start by noticing that there is $\s_0>0$ small enough (depending on $h$ and $\vphi$) such that, if $\s\in[0,\s_0)$, then
		\[
		\mbox{$u_\s=u+\s\,h(u)\,\vphi$ takes values in $[0,1]$}\,.
		\]
		Since $u_\s\ge u$ implies that  $\{u_\s^*\ge t\}$ is $\C$-spanning $\wire$ for every $t\in(1/2,\de)$, in order to make $u_\s$ admissible in $\Upsilon(v,\e,\de)$ we just need to compose with a diffeomorphism in order to restore the volume constraint. To this end, given $X\in C^\infty_c(\Om;\R^{n+1})$ with
		\begin{equation}\label{firt var par}
			\int_{\Omega}V(u)\,\Div\,X=1\,,
		\end{equation}
		let $\tau_0>0$ be small enough so that, defining $\Phi\in C^\infty((-\tau_0,\tau_0)\times\Om;\R^{n+1})$  by $\Phi(\tau,x)=\Phi_\tau(x)=x+\tau\,X(x)$, we have that $\Phi_\tau$ is a diffeomorphism of $\Om$ for every $|\tau|<\tau_0$. Denoting by $\Psi_\tau$ the inverse of $\Phi_\tau$, and letting $g(\s,\tau)=\V\big(u_\s\circ\Psi_\tau\big)$, we observe that $g(0,0)=v$ with $\pa_\tau g(0,0)=\int_\Om V(u)\,\Div X=1$ by \eqref{firt var par}. Therefore, by the implicit function theorem, up to decreasing the value of $\s_0$, we can find $m(\s)$ with $m(0)=0$, $|m(\s)|<\tau_0$, and $g(\s,m(\s))=v$ for every $\s\in[0,\s_0)$ -- in particular, differentiating $g(\s,m(\s))=v$ and recalling \eqref{firt var par}, we find
		\begin{equation}
			\label{mprime 0}
			0=m'(0)+\int_\Om h(u)\,\vphi\,V'(u)\,.
		\end{equation}
		Since $v_\s=u_\s\circ\Psi_{m(\s)}$ is admissible in $\Upsilon(v,\e,\de)$ for every $\s\in[0,\s_0)$, the minimality of $u=u_{\s=0}$ in $\Upsilon(v,\e,\de)$ implies that $f(\s)=\ac(u_\s\circ\Psi_{m(\s)})$ has a minimum on $[0,\s_0)$ at $\s=0$. By combining $f'(0)\ge0$ with \eqref{mprime 0}, \eqref{EL inner proof of reg} and \eqref{firt var par} we thus find
		\begin{eqnarray*}
			0&\leq &m'(0)\,\int_{\Omega} \Big\{\e\,|\nabla u|^2 +\frac{W(u)}{\e}\Big\}\,\Div X
			-2\,\nabla u \cdot \nabla X[\nabla u]
			\\
			&&+ \int_{\Omega} 2\,\e\, \nabla u\cdot \nabla [h(u)\,\vphi] +\frac{W'(u)}{\e}\,h(u)\,\vphi
			\\
			&=&-\l\,\int_\Om h(u)\,\vphi\,V'(u)+ \int_{\Omega} 2\,\e\, \vphi\,h'(u)|\nabla u|^2+2\,\e\,h(u)\,\nabla u\cdot\nabla\vphi+\frac{W'(u)}{\e}\,h(u)\,\vphi\,,
		\end{eqnarray*}
		that is \eqref{eq levelset Euler-Lagrange} by definition of $F_\e(t)$.
		
		\medskip
		
		\noindent {\it Step two}: We prove that
		\begin{equation}
			\label{eq one sided Euler-Lagrange}
			2\,\e^2 \,\Delta u\leq W'(u)-\l\,\e\,V'(u)\quad\mbox{as distributions on $\Om$}\,.
		\end{equation}
		Indeed, let $\{h_k\}_k\subset C^1_c([0,1))$ be a sequence such that $0\le h_k(t)\le h_{k+1}(t)\to 1$ for every $t\in[0,1)$ and such that $h_k'\le 0$ on $[0,1)$. Then for any $\varphi\in C_c^\infty(\Omega;[0,\infty))$, we have $\varphi h'_k(u)|\nabla u|^2\leq 0$. Therefore, letting $k\to\infty$ and applying \eqref{eq levelset Euler-Lagrange} we can deduce
		\begin{equation*}
			0\le \int_{\Omega \cap \{u<1\}} \nabla \varphi \cdot \nabla u\cdot + \vphi\,F_\e'(u)\qquad\forall\vphi\in C^1_c(\Om;[0,\infty))
		\end{equation*}
		by means of the dominated convergence theorem. Since $F_\e'(1)=0$ and $\nabla u=0$ a.e.~ on $\{u=1\}$, we immediately deduce \eqref{eq one sided Euler-Lagrange}.
		
		\medskip
		
		\noindent {\it Step three}: We prove that, if $x_0\in\Om$ and $r_0=\dist(x_0,\pa\Om)$, then the function
		\[
		g(r)= e^{-kr }\,\phi(r)\,,\quad\mbox{where}\quad \phi(r)= \fint_{B_r(x_0)} u\,,\qquad k=\sup_{[0,1]}|F_\e''|\,,
		\]
		is decreasing on $(0,r_1)$ where $r_1= \min\{r_0,2\}$. Indeed, assuming without loss of generality that $x_0=0$ and testing \eqref{eq one sided Euler-Lagrange} with a sequence $\{\varphi_k\}\subset C_c^1(\Om;[0,\infty))$ such that $\vphi_k(x)\to [(r^2-|x|^2)/2]_+$ uniformly and $\nabla \varphi_k \to \nabla [(r^2-|x|^2)/2]_+$ in $L^2$ yields
		\begin{equation}\label{eq derivative meanvalue esti}
			\int_{B_r} x\cdot\nabla u \leq  \int_{B_r} \frac{r^2-|x|^2}{2}|F_\e'(u)|\le k\,\frac{r^2}2\,\int_{B_r}u \,,\qquad\forall r<r_0\,,
		\end{equation}
		where we have used $F_\e'(0)=0$ and the fundamental theorem of calculus to bound $|F_\e'(u)|\leq ku$. Now, for a.e.  $r\in (0,r_0)$ we have that
		\begin{equation}\label{this}
			\phi'(r)=\fint_{B_1}(y\cdot\nabla u(r\,y))\, dy= \frac1{r}\,\fint_{B_r} (x\cdot\nabla u(x))\,dx\,.
		\end{equation}
		By combining \eqref{eq derivative meanvalue esti} with \eqref{this} and using $r_1\le 2$ we find $\phi'(t)\le k\,\phi(r)$ for a.e. $r<r_1$ and conclude.
		
		\medskip
		
		\noindent {\it Step four}: We prove that $u$ is (Lebesgue equivalent to) a lower semicontinuous function on $\Om$. Indeed, by step three, we can define $\tilde{u}:\Om\to[0,1]$ by setting
		\[
		\tilde{u}(x)=\lim_{r\to 0^+}e^{-k\,r}\fint_{B_r(x)}u\,,\qquad x\in\Om\,.
		\]
		Denoting by $\tilde{u}$ this limit, we have $\tilde{u}=u$ a.e. on $\Om$ by the Lebesgue points theorem. We conclude by proving that $\tilde{u}$ is lower semicontinuous: indeed, if $x_j\to x\in\Om$ as $j\to\infty$, then
		\[
		e^{-k\,r}\,\fint_{B_r(x)}u=\lim_{j\to\infty}e^{-k\,r}\,\fint_{B_r(x_j)}u\le\liminf_{j\to\infty} \tilde{u}(x_j)\,,
		\]
		where we have used $e^{-k\,r}\fint_{B_r(x_j)}u\le\tilde{u}(x_j)$ for every $r<\min\{2,\dist(x_j,\pa\Om)\}$. The conclusion follows by letting $r\to 0^+$.
		
		\medskip
		
		\noindent {\it Step five}: We prove that for every $\Om'$ connected component of $\Om$, either $u\equiv0$ on $\Om'$ or $u>0$ on $\Om'$; and that, if $\de<1$, then $u<1$ on $\Om$.
		
		\medskip
		
		To prove the first assertion we notice that the lower semicontinuity and the non-negativity of $u$ imply that $\{u=0\}$ is relatively closed in $\Om$. At the same time, if $x\in\{u=0\}$, then by step three $0=u(x)\ge e^{-k\,r}\,\fint_{B_r(x)}u\ge 0$ implies that $u\equiv 0$ on $B_r(x)$ for every $r\le\min\{\dist(x,\pa\Om),2\}$; in particular, $\{u=0\}$ is open. Since $\{u=0\}$ is both open and relatively closed in $\Om$, we conclude that $u\equiv 0$ on $\Om'$ or $u>0$ on $\Om'$ for any given connected component of $\Om$.
		
		\medskip
		
		Next, we show that if $\de<1$, then $\{u>\delta\}$ is open and
		\begin{equation}\label{eq EL above obstacle}
			2\e \Delta u = \frac{1}{\e}W'(u)-\lambda V'(u)\,,\quad\mbox{as distributions on $\{u>\de\}$}\,.
		\end{equation}	
		By a standard application of the strong maximum principle \cite[Theorem 6.2]{maggirestrepo}, \eqref{eq EL above obstacle} allows us to show that $u<1$ on $\Om$ if $\de<1$, finishing the proof of Theorem \ref{theorem main regularity}.(iii). We first notice that $\{u>\delta\}$ is open by step four. To prove \eqref{eq EL above obstacle}, the inequality \eqref{eq one sided Euler-Lagrange} reduces our task to showing that for every $B_r(x)\cc\{u>\de\}$ and every $\vphi\in C^\infty_c(B_r(x);[0,\infty))$,
		\begin{equation}
			\label{obs}
			0\ge \int_{\Omega}  \nabla u\cdot\nabla \varphi + \vphi\,F_\e'(u)\,.
		\end{equation}
		And indeed, by lower semicontinuity of $u$, there is $\de_0>0$ such that $u\ge\de+\de_0$ on $\cl(B_r(x))$. In particular, for every $\vphi\in C^\infty_c(B_r(x))$ with $\vphi\ge0$ there is $\s_0>0$ such that, for every $\s\in(0,\s_0]$, $u_\s=u-\s\,\vphi$ takes values in $(\de,1]$ on $B_r(x)$, and agrees with $u$ on $\Om\setminus B_r(x)$. It is therefore immediate to check that $\{u_\s^*\ge t\}=\{u^*\ge t\}$ for every $t\in(1/2,\de)$, so that $\{u_\s^*\ge t\}$ is $\C$-spanning $\wire$ for every $t\in(1/2,\de)$. We can then repeat the volume-fixing argument of step one and prove \eqref{obs}, as desired.
		
		\medskip

		\noindent {\it Step six}: We claim that for every $\varphi \in C_c^1(\Om)$ and $h \in\Lip_c([0,1]\setminus \{\delta\})$, it holds
		\begin{equation}\label{eq strong levelset Euler-Lagrange}
			0= \int_{\Omega}  \varphi\, h'(u)\,|\nabla u|^2+ h(u)\,\nabla u\cdot\nabla \varphi + \vphi\,h(u)\,F_\e'(u).
		\end{equation}
		By virtue of step 5, we can assume without loss of generality that $\varphi$ is supported in a connected component of $\Omega$ where $u>0$. Since $u$ is lower semicontinuous, then $\inf_{\text{supp}(\varphi)} u >0$. So, for $\sigma_0$ small enough, if $|\sigma|\leq \sigma_0$ then
		\[
		\mbox{$u_\s=u+\s\,h(u)\,\vphi$ \, takes values in $[0,1]$}\,.
		\]
		We wish to test the minimality of $u$ against $u_\sigma$, which requires verifying the spanning condition and then fixing volumes and testing as in step one. Regarding the spanning condition, if $\delta =1$, then $h \in\Lip_c([0,\delta'))$ for some $\delta'<\delta$, and so $\{u \geq \gamma\} =\{u_\s \geq \gamma\}$ for $\gamma \in [\delta',1]$ and small enough $\sigma$, implying that  $\{u_\s^*\ge t\}$ is $\C$-spanning $\wire$ for every $t\in(1/2,\de)$. If $\delta<1$, then $(\delta-\eta,\delta+\eta) \subset \text{supp}(h)^c$ for some $\eta>0$, and so again for small enough $\eta$ depending on $\sigma$ we find that $\{u \geq \gamma\} = \{u_\s \geq \gamma\}$ for $\de-\eta/2\leq \gamma \leq \delta+\eta/2$.  So we can repeat the volume-fixing argument of step one to obtain \eqref{eq strong levelset Euler-Lagrange}.
		
		\medskip
		
\noindent {\it Step seven:}  We finally show that $u$ satisfies \eqref{EL outer}. By the coarea formula and since $u\in W^{1,2}(\Om)$, $\L^1$-a.e. $t_0\in(0,\de_0)$ is a Lebesgue point of $t\to \int_{\{u=t\}}|\nabla u| d\mathcal{H}^{n}$. For such a value of $t_0$, let us consider the functions
		\begin{equation*}
			h_k(r)=
			\begin{cases}
				1, \quad \qquad \qquad r\in [0, t_0-2^{-k}],\\
				2^k(t_0-r), \, \, \quad r\in [t_0-2^{-k},t_0],\\
				0, \quad \qquad \qquad  r\in [t_0,1].
			\end{cases}
		\end{equation*}
		By plugging $h_k$ into \eqref{eq strong levelset Euler-Lagrange}, taking $k\to \infty$, and using the coarea formula, we deduce
		\begin{equation*}
			\int_{\{u=t_0\}} |\nabla u|\,\varphi\, d\mathcal{H}^{n}= \int_{\{u<t_0\}}\nabla u\cdot \nabla \varphi +\varphi F_\e'(u)\,.
		\end{equation*}
		Integrating between $0$ and $\delta$, using the coarea formula and Fubini's theorem, we find
		\begin{eqnarray}\notag
			\int_{ \{u<\delta\}}|\nabla u|^2\,\varphi &=&
\int_{0}^\delta dt\,\int_{\{u<t\}}\nabla u\cdot \nabla \varphi +\varphi F_\e'(u)
\\
\label{eq levelset value}
			&=& \int_{ \{u<\delta\}}(\delta-u)\Big(\nabla u\cdot \nabla \varphi +\varphi F_\e'(u)\Big)\,.
		\end{eqnarray}
		By analogous reasoning, we deduce
		\begin{equation*}
			-	\int_{\{u=t_0\}} |\nabla u|\,\varphi\, d\mathcal{H}^{n}= \int_{\{u>t_0\}}\nabla u\cdot \nabla \varphi +\varphi F_\e'(u)
		\end{equation*}
		for a.e. $t_0 \in ( \delta,1)$ and thus
		\begin{equation}\label{eq levelset value oposite}
			\int_{ \{u>\delta\}}|\nabla u|^2\,\varphi = \int_{ \{u>\delta\}}(\delta-u)\Big(\nabla u\cdot \nabla \varphi +\varphi F_\e'(u)\Big)\,.
		\end{equation}
		By combining \eqref{eq levelset value} and  \eqref{eq levelset value oposite}, we obtain \eqref{EL outer}, and complete the proof of the theorem.
\end{proof}

We finally prove Proposition \ref{proposition conditional regularity}.

\begin{proof}[Proof of Proposition \ref{proposition conditional regularity}] To prove statement (i), let us assume that  $u$ is continuous in $\Om$. This implies that $\{u=\delta\}$ is closed, and we can thus proceed as in step five to deduce that
\begin{equation}\label{eq full equation}
	\Delta u= F_\e'(u)\qquad\mbox{on $\{u\neq \delta\}$}\,.
\end{equation}
Since $W\in C^{2,1}[0,1]$ implies $V\in C^{2,\g(n)}[0,1]$ with $\g(n)=\min\{1,2/n\}$ (see \cite[Appendix 3]{maggirestrepo}), we have $F_\e'\in C^{1,\g(n)}[0,1]$. The continuity of $u$ implies that $\Delta u$ is continuous on $\{u\ne\de\}$, hence that $u\in C^{1,\a}_{\rm loc}(\{u\ne\de\})$ for every $\a<1$. Hence $F_\e'(u)\in C^{1,\a}_{\rm loc}(\{u\ne\de\})$ for every $\a<\g(n)$, and thus $u\in C^{3,\a}_{\rm loc}(\{u\ne\de\})$ for every $\a<\g(n)$, as claimed.

\medskip

Since statement (iii) follows easily from statement (ii), we give a detailed proof of the latter only. We employ an argument similar to \cite[Theorem 2.4]{alt1984variational}. Given $X\in C^\infty_c(\Om;\R^{n+1})$, let us set
\[
Y=\big\{|\nabla u|^2+2\,F_\e(u)\big\}\,X-2\,(X\cdot\nabla u)\,\nabla u\,.
\]
In this way $Y\in C^1(\{u\ne\de\};\R^{n+1})$, and by direct computation we find that, on $\{u\ne\de\}$,
\begin{eqnarray}\nonumber
  \Div\,Y&=&\big\{|\nabla u|^2+2\,F_\e(u)\big\}\,\Div\,X-2\,\nabla u\cdot\nabla X[\nabla u]+2\,(X\cdot\nabla u)\,(F_\e'(u)-\Delta u)
  \\\label{ka1}
  &=&
  \big\{|\nabla u|^2+2\,F_\e(u)\big\}\,\Div\,X-2\,\nabla u\cdot\nabla X[\nabla u]\,,
\end{eqnarray}
where in the second identity we have used \eqref{eq full equation}. Now, let us set
\[
\L[X]=\big\{|\nabla u|^2+2\,F_\e(u)\big\}\,\Div\,X-2\,\nabla u\cdot\nabla X[\nabla u]\,,
\]
so that, thanks to $\nabla u=0$ $\L^{n+1}$-a.e. on $\{u=\de\}$, we have
\begin{equation}
  \label{LXX}
  \L[X]\in L^1(\Om)\,,\qquad\L[X]=2\,F_\e(\de)\,\Div\,X\qquad\mbox{$\L^{n+1}$-a.e. on $\{u=\de\}$}\,.
\end{equation}
If we set $S_t=\{u>\de+t\}\cup\{u<\de-t\}$, $t>0$, then, by the inner variation critical point condition \eqref{EL inner proof of reg}, which gives $\int_\Om\L[X]=0$, and by $\Div Y=\L[X]$ on $\{u\ne\de\}$, we find
\[
\int_{S_t}\Div Y=\int_{S_t}\L[X]=-\int_{\{|u-\de|<t\}}\L[X]
\]
where, by \eqref{LXX},
\[
\lim_{t\to 0^+}\int_{\{|u-\de|<t\}}\L[X]\,d\L^{n+1}=\int_{\{u=\de\}}\L[X]\,d\L^{n+1}=2\,F_\e(\de)\,\int_{\{u=\de\}}\,\Div X\,d\L^{n+1}\,.
\]
We now write $Y=Y_1+Y_2$ with
\[
Y_1=|\nabla u|^2\,X-2\,(X\cdot\nabla u)\,\nabla u\,,\qquad Y_2=2\,F_\e(u)\,X\,.
\]
Since $u\in W^{1,2}(\Om)$ and $X\in C^\infty_c(\Om)$ it turns out that $Y_2\in W^{1,1}(\Om;\R^{n+1})$ with
\[
\lim_{t\to 0^+}\int_{S_t}\Div\,Y_2=\int_\Om\Div\,Y_2=0\,.
\]
We have thus proved that
\begin{equation}
  \label{ka7}
  \lim_{t\to 0^+}\int_{S_t}\Div Y_1=2\,F_\e(\de)\,\int_{\{u=\de\}}\,\Div X\,d\L^{n+1}\,,
\end{equation}
where we are stressing that the integral over $\{u=\de\}$ is respect with the Lebesgue measure. Now, for a.e. $t>0$, we have that $S_t$ is a set of finite perimeter in $\Om$, with $\pa^*S_t=\pa^*\{u>\de+t\}\cup\pa^*\{u<\de-t\}$ and
\begin{eqnarray*}
\nu_{S_t}=-\frac{\nabla u}{|\nabla u|}\,,&&\qquad\mbox{$\H^n$-a.e. on $\pa^*\{u>\de+t\}$}\,,
\\
\nu_{S_t}=\frac{\nabla u}{|\nabla u|}\,,&&\qquad\mbox{$\H^n$-a.e. on $\pa^*\{u<\de-t\}$}\,,
\end{eqnarray*}
and thus
\begin{eqnarray*}
\int_{S_t}\Div\,Y_1&=&\int_{\pa^*\{u<\de-t\}}Y_1\cdot\frac{\nabla u}{|\nabla u|}-\int_{\pa^*\{u>\de+t\}}Y_1\cdot\frac{\nabla u}{|\nabla u|}
  \\
  &=&
  -\int_{\pa^*\{u<\de-t\}}(X\cdot\nabla u)\,|\nabla u|+\int_{\pa^*\{u>\de+t\}}(X\cdot\nabla u)\,|\nabla u|\,.
\end{eqnarray*}
By \eqref{ka7}, if we assume that $|\{u=\de\}|=0$, we conclude that
\begin{equation}
  \label{ka8}
  \lim_{t\to 0^+}\int_{\pa^*\{u<\de-t\}}(X\cdot\nabla u)\,|\nabla u|-\int_{\pa^*\{u>\de+t\}}(X\cdot\nabla u)\,|\nabla u|
=0\,,
\end{equation}
with the limit taken with $t$ such that $S_t$ has finite perimeter. This is \eqref{eq fb condition}, and the proof of the proposition is complete.
\end{proof}

\appendix

\section{The diffused interface Euclidean isoperimetric problem}\label{section diffuse isop} In this appendix we collect some important properties of the diffused interface Euclidean isoperimetric problem considered in \cite{maggirestrepo}, i.e.
\begin{equation}\label{mr}
    \Theta(v,\e) := \inf \Big\{\frac{\ac(u)}2 : \V(u) = v\Big\}\,,
\end{equation}
(where $\V(u)=\V(u;\R^{n+1})$ and $\ac(u)=\ac(u;\R^{n+1})$), including the uniqueness of minimizers and the characterization of minimizers as the only critical points in the ``geometric regime'' where $\e\ll v^{1/(n+1)}$, which is one the main results proved in \cite{maggirestrepo}.

\begin{theorem}\label{theorem diffused isop} If $W\in C^{\,2,1}[0,1]$ satisfies \eqref{W nondegeneracy assumptions} and \eqref{W normalization}, then the following holds:

\medskip

\noindent {\bf (i):} for every $v$ and $\e$ positive, there exists a radial decreasing symmetric minimizer $\zeta$ of $\Theta(v,\e)$ such that $\zeta\in C^{2,\alpha}_\loc(\mathbb{R}^{n+1};(0,1))$ for some $\a\in(0,1)$;

\medskip
\noindent {\bf (ii):} $\Theta$ is continuous on $(0,\infty)\times(0,\infty)$ with $\Theta(v,\e)=r^n\,\Theta(v/r^{n+1},\e/r)$ for every $r>0$; moreover, for every $\e>0$, the function $v>0\mapsto\Theta(v,\e)/v$ is strictly decreasing;

\medskip

\noindent {\bf (iii):} if $\zeta$ is a minimizer of $\Theta(v,\e)$, then there exists $\Lambda(\zeta)\in \mathbb{R}$ such that for all $X\in C_c^\infty(\mathbb{R}^{n+1};\mathbb{R}^{n+1})$
\begin{equation}
  \label{inner appendix}
\int_{\mathbb{R}^{n+1}}\mathrm{ac}_\e(\zeta)\, \Div X - 2\,\e
    \nabla \zeta\,\cdot\, \nabla X[\nabla \zeta]=\Lambda(\zeta)\, \int_{\mathbb{R}^{n+1}} V(\zeta)\,\Div X\,,
\end{equation}
as well as
\begin{equation}
  \label{outer appendix}
  2\,\e^2\,\Delta \zeta=W'(\zeta)-\e\,\Lambda(\zeta)\,V'(\zeta)\,,\qquad\mbox{on $\R^{n+1}$}\,.
\end{equation}
Moreover, for some positive constant $c=c(n,W)$, we have
\begin{equation}
\label{key inq for Lambda}
  \Lambda(\zeta)\ge \frac{c}{v^{1/(n+1)}}\,.
\end{equation}
%and, finally, for every $r>0$, $\zeta\circ\eta_r(x):=\zeta(r\,x)$ is a minimizer of $\Theta(v/r^{n+1},\e/r)$ with $\Lambda(\zeta_r)=\Lambda(\zeta)\,r$.

\medskip

\noindent {\bf (iv):} there is $\s_0=\s_0(n,W)>0$ such that if $0<\e<\s_0\,v^{1/(n+1)}$, then there is a unique modulo translation radial decreasing symmetric minimizer $\zeta_{v,\e}$ of $\Theta(v,\e)$ with maximum at the origin, which satisfies
\begin{equation}
  \label{energy expansion for theta}
  \Theta(v,\e)=v^{n/(n+1)}\Big\{(n+1)\,\om_{n+1}^{1/(n+1)}+{\rm O}_{n,W}\Big(\frac{\e}{v^{1/(n+1)}}\Big)\Big\}\,,
\end{equation}
as $\e/v^{1/(n+1)}\to 0$.

%%%%. Moreover, $\Lambda(v,\e):=\Lambda(\zeta_{v,\e})$ is continuous on $\{(v,\e):0<\e<\s_0\,v^{1/(n+1)}\}$ and satisfies
%%%%\begin{equation}
%%%%  \label{limit of lagrange of Theta}
%%%%  \lim_{\e/v^{1/(n+1)}\to 0^+}v^{1/(n+1)}\,\Lambda(v,\e)=2\,n\,\om_{n+1}^{1/(n+1)}\,.
%%%%\end{equation}
%%%%%%%
%%%%%%%
%%%%%%%if $v_j^{1/(n+1)}/\e_j\to 0$ as $j\to\infty$ and, for each $j$, $u_j$ is a minimizer of $\Theta(v_j,\e_j)$, then
%%%%%%%\begin{align}\label{divergence of ratios and lagrange multipliers}
%%%%%%%   \lim_{j\to\infty}\Lambda(u_j)\, \e_j = +\infty\,;
%%%%%%%\end{align}
%%%%
%%%%\medskip
%%%%
%%%%\noindent {\bf (v):} there is $\nu_0=\nu_0(n,W)>0$ such that if $u\in C^2(\R^{n+1};[0,1])$, $u(x)\to0$ as $|x|\to\infty$, and, for some $\l$ such that $0<\e\,\l<\nu_0$, $u$ solves
%%%%\begin{equation}
%%%%  \label{diffused alexandrov}
%%%%  2\,\e^2\,\Delta u=W'(u)-\e\,\l\,V'(u)\,,\qquad\mbox{on $\R^{n+1}$}\,.
%%%%\end{equation}
%%%%then $\e<\s_0\,\V(u)^{1/(n+1)}$ and  $u$ is a minimizer of $\Theta(v,\e)$ for $v=\V(u)$. In particular $\l=\Lambda(v,\e)$ and, up to a translation, $u=\zeta_{v,\e}$.
%%%%%
%
%
%for every $\de>0$, there exist positive constants $R_\de$ and $\tau_\de$ such that if $v^{1/(n+1)}/\e\ge \de$ and $v_*>0$ satisfies $|1-(v_*/v)|<\tau_\de$, then there exists a radially symmetric function $u_*$ such that
%\begin{align}\label{close energy}
%\spt\, u_* \subset B_{R_\de\,v^{1/(n+1)}}\,,\qquad \V(u_*;\R^{n+1})=v_*\,,\qquad
%\frac{\ac(u_*;\R^{n+1})}2<(1+\de)\,\Theta(v,\e)\,.
%\end{align}
\end{theorem}

\begin{proof} It is convenient to notice that by \eqref{W nondegeneracy assumptions} there are $\beta_0\in(0,1)$ and $c_0>0$ (depending on $W$) such that
\begin{equation}
  \label{utili}
  \frac{t^2}{c_0}\ge W(t)\ge c_0\,t^2\,,\qquad \frac{V(t)}{c_0}\ge t\,V'(t)\qquad\forall t\in(0,\b_0)\,.
\end{equation}

\medskip

\noindent {\it Step one}: We prove the existence of $c=c(n,W)$ such that, if $u$ is a competitor of $\Theta(v,\e)$ and $\beta\in(0,\beta_0)$, then
\begin{equation}
\label{cost per unit volume equation}
    \frac{\ac(u;A)}{\mathcal{V}(u;A)} \geq \frac{c}{\e\,\b^{2/n}}\,,\qquad\forall A\subset\{0\le u\le \beta\}\,.
\end{equation}
Indeed, by \eqref{utili}, if $A\subset\{0\le u\le \b\}$ for some $\b\in(0,\b_0)$, then
\begin{align}\notag
    \ac(u;A) \geq \frac1\e\,\int_A W(u)\geq \frac{c_0}{\e}\,\int_A u^2\,dx\,,
\end{align}
while $V(t)=(\int_0^t\sqrt{W})^{(n+1)/n}\le C t^{2(n+1)/n}$ ($t<\beta_0$) implies $\V(u;A)\le C\,\beta^{2/n}\,\int_A u^2$.

\medskip

\noindent {\it Step two}: We prove\footnote{We notice that the analysis performed in \cite{maggirestrepo}, which is focused on uniqueness and stability issues, is limited to the regime where $\e/v^{1/(n+1)}$ is small enough in terms of $n$ and $W$.} conclusion (i). By the P\'{o}lya-Szeg\"{o} inequality \cite{BZ} we can consider a minimizing sequence $\{u_j\}_j$ of $\Theta(v,\e)$ such that each $u_j$ is radial decreasing symmetric with respect to the origin. Up to extracting subsequences we can assume that $u_j\to \zeta$ in $L^1_{\rm loc}(\R^{n+1})$, where $\zeta$ is radial decreasing symmetric with respect to the origin and such that $\ac(\zeta)\le2\,\Theta(v,\e)$. We prove that $\V(\zeta)=v$, and thus that $\zeta$ is a minimizer of $\Theta(v,\e)$, by showing that
\begin{equation}\label{tightness}
\lim_{R\to\infty}\sup_j\V(u_j;\R^{n+1}\setminus B_R)=0\,.
\end{equation}
To prove \eqref{tightness}, let us set $u_j(x)=g_j(|x|)$ and $\zeta(x)=g(|x|)$, and notice that, since $g_j\to g$ a.e. on $(0,\infty)$ with $g_j$ and $g$ decreasing on $(0,\infty)$ and $g(R)\to 0$ as $R\to\infty$, it holds that $\sup_j\,g_j(R)\to 0$ as $R\to\infty$. In particular, for every $R$ large enough to ensure $\sup_j\,g_j(R)\le\b_0$ we can apply \eqref{cost per unit volume equation} to conclude that
\[
\V(u_j;\R^{n+1}\setminus B_R)\le \frac{\e\,g_j(R)^{2/n}}{c_0}\,\sup_i\ac(u_i)\,,
\]
which implies \eqref{tightness} thanks (again) to $\sup_j\,g_j(R)\to 0$ as $R\to\infty$. The fact that $\zeta\in C^{2,\a}_{\rm loc}(\R^{n+1};(0,1))$ for some $\a\in(0,1)$ is proved as in \cite[Proof of Theorem 2.1, Step four]{maggirestrepo}.

\medskip

\noindent {\it Step three}: We prove conclusion (ii). Since the scaling property and the continuity of $\Theta$ can be proved as in \cite[Appendix A]{maggirestrepo} and \cite[Step 3, Proof of Theorem 2.1]{maggirestrepo}, we focus on showing that, for $\e>0$ fixed, $v\mapsto \Theta(v,\e)/v$ is strictly decreasing on $(0,\infty)$. Indeed, by Fubini's theorem, if $\zeta$ is a minimizer for $\Theta(v,\e)$ and we set
\[
Z(t)=\frac{(1/2)\int_{\{x_1=t\}}\mathrm{ac}_\e(\zeta) \,d\H^n}{\int_{\{x_1=t\}}V(\zeta)\,d\H^n}\,,\qquad t\in\R\,,
\]
then, trivially, $\Theta(v,\e)\ge v\,\inf_\R Z$, with equality if and only if $Z$ is constant on $\R$. Since $\zeta$ is radial decreasing symmetric, we have $\beta(t)=\sup_{\{x_1=t\}}\zeta\to 0$ as $t\to\infty$. Since $\V(\zeta)<\infty$, we can find $t_j\to\infty$ with $\int_{\{x_1=t_j\}}V(\zeta)\,d\H^n \to 0$ as $j\to\infty$. Correspondingly, $\beta(t_j)\to 0^+$ and, by \eqref{cost per unit volume equation}, $Z(t_j)\to+\infty$ as $j\to\infty$. In particular, $Z$ is not constant on $(0,\infty)$, so that $\Theta(v,\e)>v\,\inf_\R Z$, i.e., there is $t_0\in\R$ such that
\begin{equation}
  \label{ehh}
  \Theta(v,\e)>v\,Z(t_0)\,.
\end{equation}
Now, given $\tilde{v}>v$, if we pick $\de=(\tilde{v}-v)/(2\,\int_{\{x_1=t_0\}}V(\zeta))$, decompose $x=(x_1,x')\in\R\times\R^n\equiv\R^{n+1}$, and set
\begin{equation}\notag
   u(x) = \begin{cases}
    \zeta(t_0,x')\,, & \textup{if } t_0-\delta \leq x_1 \leq t_0+\delta\,,
    \\
    \zeta(x_1+\delta,x')\,, & \textup{if } x_1 \leq t_0-\delta\,,
    \\
    \zeta(x_1-\delta,x')\,, & \textup{if } x_1 \geq t_0+\delta\,,
    \end{cases}
\end{equation}
then
\begin{equation}
  \label{eh}
  \ac\big(u;\{|x_1-t_0|<\de\}\big)=2\,\de\,Z(t_0)\,\int_{\{x_1=t_0\}}V(\zeta)=Z(t_0)\,(\tilde{v}-v)\,,
\end{equation}
and, similarly, $\V(u)=\V(\zeta)+2\,\de\,\int_{\{x_1=t_0\}}V(\zeta)=\tilde{v}$. Since $u$ is admissible in $\Theta(\tilde{v},\e)$, by \eqref{eh} we find, as desired,
\[
\frac{\Theta(\tilde{v},\e)}{\tilde{v}}\leq \frac{\ac(u)/2}{\tilde{v}} \leq \frac{\Theta(v,\e)+ Z(t_0)\,(\tilde{v}-v)}{\tilde{v}} < \frac{\Theta(v,\e)}{v}\,,
\]
where in the last inequality we have used \eqref{ehh}.

\medskip

\noindent {\it Step four}: The validity of \eqref{inner appendix} is immediate from Lemma \ref{lemma volume fixing}, while \eqref{outer appendix} can be deduced from \eqref{inner appendix} via integration by parts. To prove \eqref{key inq for Lambda} (and thus complete the proof of conclusion (iii)) it is enough to show that, for some constant $C=C(n,W)$, it holds
\begin{equation}
  \label{enough}
  C\,\Lambda(\zeta)\,v\ge\ac(\zeta)\,.
\end{equation}
(Indeed, $\ac(\zeta)/v=\Theta(v,\e)/v\ge c(n)\,v^{1/(n+1)}$ thanks to $\Theta(v,\e)\ge c(n)\,v^{n/(n+1)}$.) To prove \eqref{enough} we first notice that testing \eqref{outer appendix} with suitable radial vector fields as done in \cite[Equation (2.32)]{maggirestrepo} one finds
\begin{eqnarray}\nonumber
(n+1)\,\Lambda(\zeta)\, v&=&n\,\mathcal{AC}_{\e}(\zeta)+\int_{\mathbb{R}^{n+1}}\frac{W(\zeta)}{\e} - \e\,|\nabla \zeta|^2
\\\label{lambda j expression}
&\ge&(n-1)\,\mathcal{AC}_{\e}(\zeta)+\int_{\mathbb{R}^{n+1}}\frac{W(\zeta)}{\e}\,.
\end{eqnarray}
Clearly \eqref{lambda j expression} implies \eqref{enough} when $n\ge2$, but leaves open the case $n=1$; however, it always ensures that $\Lambda(\zeta)\ge0$. Next, we notice that by testing \eqref{outer appendix} with $\vphi=\zeta\,\phi_k^2$ for $\phi_k\in C_c^\infty(B_{k+1};[0,1])$ with $\phi_k=1$ on $B_k$ and $\Lip(\phi_k)\le 2$ for every $k$, and keeping in mind that $V'\ge0$ and that $\Lambda(\zeta)\ge0$, we find
\[
2\,\e\, \int_{\R^{n+1}} \phi_k^2\,|\nabla \zeta|^2  \leq  2\,\e\,\int_{\R^{n+1}} \phi_k\,\zeta\,|\nabla \zeta|\,|\nabla\phi_k|+
\int_{\R^{n+1}}\frac{W'(\zeta)\,\zeta\,\phi_k}{\e}+ \Lambda(\zeta)\,\int_{\R^{n+1}}V'(\zeta) \zeta\,.
\]
Since for $k\ge k(\zeta)$ we have $\zeta\le\b_0$ on $\R^{n+1}\setminus B_k$, and we can thus use \eqref{utili} and the Cauchy--Schwartz inequality to deduce, for some $C=C(W)$,
\begin{eqnarray}\label{concludefrom}
\e\, \int_{\R^{n+1}}\phi_k^2 |\nabla\zeta|^2 \leq
 C\,\e\,\int_{\{\zeta<\b_0\}} \zeta^2+\int_{\R^{n+1}}\frac{[W'(\zeta)]^+\,\zeta}{\e}+
 C\,\Lambda(\zeta)\,\V(\zeta)\,,
\end{eqnarray}
where we have used $c_0\,t\,V'(t)\le t^{2(n+1)/n}\le V(t)/c_0$ for every $t\in(0,1)$ (which, in turns, easily follows from $W(t)\ge c_0\,t^2$ on $(0,\beta_0)$ and $W(t)\le t^2/c_0$ on $(0,1)$). Finally, by \eqref{W nondegeneracy assumptions}, and up to decreasing the value of $\b_0$, we have $W'<0$ on $(1-\b_0,1)$. Using this fact in combination with $\inf_{(\b_0,1-\b_0)}W>0$ and $t\,W'(t)\le W(t)/c_0$ for $t\in(0,\beta_0)$, we see that
\begin{eqnarray*}
\int_{\R^{n+1}}[W'(\zeta)]^+\,\zeta&\le&\int_{\{\zeta<\b_0\}}W'(\zeta)\,\zeta+ \Lip(W)\,\int_{\{\b_0\le \zeta\le 1-\b_0\}}\,\zeta
\\
&\le&\Big\{\frac1{c_0}+ \frac{\Lip(W)\,(1-\b_0)}{\inf_{(\b_0,1-\b_0)}W}\Big\}\,\int_{\R^{n+1}}W(\zeta)\,,
\end{eqnarray*}
and thus conclude from \eqref{concludefrom} that
\begin{eqnarray*}
\e\, \int_{\R^{n+1}}\phi_k^2\, |\nabla \zeta|^2  \le C\, \int_{\R^{n+1}}\frac{W(\zeta)}{\e}+  C\,\Lambda(\zeta)\,v\,.
\end{eqnarray*}
By letting $k\to\infty$, by adding $\int_{\R^{n+1}}W(\zeta)/\e$ to both sides of this inequality, and by noticing that \eqref{lambda j expression} implies $\int_{\R^{n+1}}W(\zeta)/\e\le C(n)\,\Lambda(\zeta)\,v$ for every $n\ge 1$, we conclude the proof of \eqref{enough}.

\medskip

\noindent {\it Step five}: The outer form of the Euler--Lagrange equation \eqref{outer appendix} follows from  Lemma \ref{lemma volume fixing}, and a classical computation (based on integration by parts made possible by the $C^2$-regularity of $\zeta$) allows one to derive \eqref{inner appendix} from \eqref{outer appendix}. This completes the proof of conclusion (iii). Conclusions (iv) is contained in \cite[Theorem 1.1]{maggirestrepo}.
\end{proof}

\bibliographystyle{alpha}
\bibliography{references}

\end{document}